\documentclass[11pt]{amsart}
\pdfoutput=1

\usepackage[utf8]{inputenc}
\usepackage[margin=3cm]{geometry}
\usepackage{amssymb, amsmath, amsthm, amsfonts}
\usepackage{color}
\usepackage{mathrsfs}
\usepackage{physics}
\usepackage{mathtools}
\usepackage{tikz}
\usepackage{faktor}
\usepackage{tikz-cd}
\usepackage{bm}
\usepackage{hyperref}
\usepackage{csquotes}
\usepackage{stmaryrd}
\usepackage{yhmath}
\usepackage{enumitem}
\usepackage{microtype}

\DeclareRobustCommand{\SkipTocEntry}[5]{}

\numberwithin{equation}{section}

\newtheorem{theorem}{Theorem}[section]
\newtheorem*{theorem*}{Theorem}
\newtheorem{lemma}[theorem]{Lemma}
\newtheorem{corollary}[theorem]{Corollary}

\newtheorem{prop}[theorem]{Proposition}

\newtheorem*{conjecture*}{Conjecture}
\theoremstyle{remark}
\newtheorem{remark}[theorem]{Remark}
\theoremstyle{remark}

\theoremstyle{remark}
\newtheorem*{notation*}{Notation}
\theoremstyle{remark}
\newtheorem*{conventions*}{Conventions}

\usepackage[UKenglish]{isodate}
\usepackage[UKenglish]{babel}
\cleanlookdateon

\newcommand{\R}{\mathbf{R}}
\newcommand{\N}{\mathbf{N}}

\newcommand{\E}{\mathbf{E}}
\newcommand{\Prob}{\mathcal{P}}

\newcommand{\Diff}{\operatorname{Diff}}
\newcommand{\Vect}{\mathfrak{X}}
\newcommand{\Hspace}{\mathcal{H}}
\newcommand{\Vspace}{\mathcal{V}}
\newcommand{\Tspace}{\mathcal{T}}

\newcommand{\Kop}{\mathsf{K}}
\newcommand{\Aop}{\mathsf{A}}
\newcommand{\Bop}{\mathsf{B}}
\newcommand{\Mop}{\mathsf{M}}
\newcommand{\Iop}{\mathsf{I}}
\newcommand{\Lop}{\mathsf{L}}
\newcommand{\Pop}{\mathsf{P}}
\newcommand{\Rop}{\mathsf{R}}
\newcommand{\Div}{\operatorname{div}}
\renewcommand{\grad}{\operatorname{grad}}

\newcommand{\Law}{\operatorname{Law}}
\newcommand{\Cov}{\operatorname{Cov}}
\newcommand{\Ent}{\operatorname{Ent}}

\renewcommand{\tr}{\operatorname{tr}}
\newcommand{\id}{\operatorname{id}}

\newcommand{\Ran}{\operatorname{Ran}}
\newcommand{\Ker}{\operatorname{Ker}}

\newcommand{\Sym}{\operatorname{Sym}}

\newcommand{\essinf}{\operatorname*{ess\,inf}}

\newcommand{\Wass}{\mathsf{W}}
\newcommand{\Stein}{\mathsf{S}}
\newcommand{\Fish}{\mathsf{F}}

\title[Otto and Stein descriptions of diffusion]{On two differing geometric descriptions of the passage from microscopy to macroscopy in Markov theory}

\author{Dalton A R Sakthivadivel}
\address{Department of Mathematics, CUNY Graduate Centre, 365 Fifth Avenue, New York, NY 10016}
\email{dsakthivadivel@gc.cuny.edu}
\date{\today}

\subjclass[2020]{Primary 49Q22, 46T05, 58B20, 60H10; Secondary 35Q84, 47D07, 53B12, 60F10, 62B10}

\begin{document}

\begin{abstract}

	In various disparate settings one studies how random processes give way to parabolic partial differential equations and in turn to functional inequalities involving gradients and variational calculus. Particles subject to gradient forces, the fields which move them, the empirical measures which they form, the smooth laws which those measures approximate, and differentiable functionals of these laws constitute a hierarchy of descriptions of diffusion. Intervening on this hierarchy is a choice of how a conintuity operator converts particle velocity to the evolution of measures. Here a central object is constructed mediating two different such descriptions. The space of smooth positive probability densities on a closed Riemannian manifold is treated as a Fr\'echet manifold whose full continuous cotangent space consists of nonconstant distributions and whose regular cotangent space consists of sufficiently regular nonconstant functions. Beginning from the operator on this space arising as the infinitesimal lift of diffeomorphisms of the base manifold, two different Hilbert completions of this space of densities account for a wide class of objects relevant, leading to a partial universalisation of the microscopic--mesoscopic--macroscopic hierarchy in Markov analysis.
\end{abstract}

\maketitle
\tableofcontents

\section{Introduction}\label{introduction-sec}

Diffusion spans several levels of description of the same physical and probabilistic evolution. A collection of particles follows random sample paths; the associated transition operators form a Markov semigroup; the one-time laws solve a parabolic equation; functionals of these laws satisfy variational identities and functional inequalities; the trend towards equilibrium may then be interpreted as a gradient evolution on a space of probability measures. Dynkin's formula relates the process and its generator, H\"ormander's theory of subellipticity supplies regularity under bracket hypotheses, and Bakry--\'Emery calculus controls gradients and entropy through the carr\'e du champ \cite{dynkin1965markov,hormander1967hypoelliptic,bakry1985diffusions}. Sample paths thereby act as if they were characteristics of a second order equation. Their one-time laws correspondingly admit a deterministic continuity representation whose velocity contains the score of the evolving density, whilst superposition and martingale problem results recover path measures from weak Fokker--Planck solutions.

The hierarchy studied here may be read in two directions. Particles determine empirical measures, empirical measures converge to laws, and differentiable functionals assign energies or observables to those laws. Variational evolution runs in the reverse direction: the differential of a functional gives a scalar covector, the spatial differential of that covector gives a force on the sample space, a geometry converts the force into a particle velocity, and the continuity equation converts the velocity into the tangent of the law. The two directions meet in the diagram
\[
	\begin{tikzcd}[column sep=large]
		(x_1,\ldots,x_N) \arrow[r,"\mathrm{empirical\ measure}"] & \mu^N \arrow[r,"\mathrm{limit}"] & \rho \arrow[r,"\mathcal F"] & \R \\
		(\dot x_1,\ldots,\dot x_N) & v_{\rho} \arrow[l,"\mathrm{evaluation}"'] & \dot\rho \arrow[l,"\mathrm{lift}"'] & \dd{\mathcal F}(\rho) \arrow[l,"\mathrm{geometry}"']
	\end{tikzcd}
\]
where the lower row suppresses the intermediate spatial differential only for the moment. The purpose of the paper is to construct every arrow, to identify the quotient ambiguities carried by each of them, and to determine when two geometries induce evolutions which agree in a chosen sense.

The common maping we cast our eyes towards is the kinematic operator arising from the infinitesimal pushforward of diffeomorphisms of a base manifold
\[
	\Aop_{\rho}v=-\Div(\rho v).
\]
A vector field $v$ moves sample points, whilst $\Aop_{\rho}v$ records the resulting motion of their law. Fields in $\Ker\Aop_{\rho}$ rearrange particles without changing the density to first order; consequently, a density tangent is a coset of particle velocities. A geometry enters the picture when one chooses a norm on admissible velocities and selects one representative of the coset generated by a functional force.

Two selections are considered. The weighted $L^2(\rho)$ norm selects the gradient itself and produces Otto geometry. A reproducing kernel Hilbert norm selects the kernel Riesz representative and produces Stein geometry. Their covector to tangent maps are
\[
	\Kop^{\Wass}_{\rho}f=-\Div(\rho\grad f)
\]
and
\[
	\Kop^{\Stein}_{\rho}f=-\Div\bigl(\rho\Iop_{\rho}^{*}\grad f\bigr),
\]
where $\Iop_{\rho}$ includes the vector-valued reproducing kernel Hilbert space into $L^2(\rho;TM)$. Both evolutions therefore begin from the same scalar covector and the same spatial force; however, a difference lies in the response rule which turns this force into particle motion.

The principal question is the meaning of agreement and circumstances of disagreement between these response rules. Equality of the selected particle fields gives the same trajectories before passage to the law. Equality modulo $\Ker\Aop_{\rho}$ gives the same law tangent whilst allowing different microscopic rearrangements. Positive proportionality gives the same unparametrised curve of laws with a different clock. Equality after a prescribed observation gives the same moments, statistical parameters or coarse variables without requiring equality of the ambient law. Equivalence of the two quadratic forms compares dissipation rates and fluctuation costs, although it does not by itself identify the deterministic trajectory. A single functional probes one covector direction at each density; agreement for one gradient flow is consequently weaker than agreement of the geometries on an entire cotangent space.

The comparison becomes explicit after the Stein velocity is viewed in the weighted $L^2(\rho)$ space. Let $\mathsf{P}^{\mathrm{hor}}_{\rho}$ be the orthogonal projection onto the closure of gradients and set
\[
	\Mop_{\rho}=\mathsf{P}^{\mathrm{hor}}_{\rho}\Iop_{\rho}\Iop_{\rho}^{*}\big|_{H^{\Wass}_{\rho}}.
\]
The operator $\Mop_{\rho}$ is positive and self-adjoint on the Otto horizontal space. The vertical component of $\Iop_{\rho}\Iop_{\rho}^{*}\grad f$ disappears under the continuity operator, giving
\[
	\Kop^{\Stein}_{\rho}f=\Aop_{\rho}\Mop_{\rho}\grad f.
\]
Thus the two density evolutions agree for the covector $[f]$ exactly when $\Mop_{\rho}\grad f=\grad f$. The same formula gives the defect in the Otto tangent norm, the time change criterion, the comparison of dissipation, and the finite dimensional design theorem developed in \S\ref{geometry-comparison-sec}.

\subsection{An historical account of the literature pertaining to these results}\label{history-literature-subsec}

Brownian motion, the Fokker--Planck equation and the Kolmogorov equations supplied the classical passage between random motion and the evolution of laws \cite{fokker1914mittlere,planck1917satz,dynkin1965markov}. H\"ormander's theorem concerns the regularity of the resulting differential operators, whilst Bakry--\'Emery theory derives functional inequalities and convergence estimates from their first and second order differential structure \cite{hormander1967hypoelliptic,bakry1985diffusions}. The probability current, forwards and backwards drifts, and stochastic flows give several pathwise and marginal representations of the same diffusion law \cite{nelson1967dynamical,haussmann1986time,kunita1990stochastic}.

Moser's theorem identifies smooth positive densities of fixed mass as a diffeomorphism orbit, whose differential is precisely $v\mapsto-\Div(\rho v)$ \cite{moser1965volume}. Ebin--Marsden and the subsequent geometry of diffeomorphism groups provide Sobolev and Fr\'echet settings for such actions \cite{ebin1970groups,khesin2013geometry,hamilton1982inverse,kriegl1997convenient}. Benamou--Brenier turned the continuity equation and kinetic energy into a dynamic formulation of optimal transport \cite{benamou2000computational}. Jordan--Kinderlehrer--Otto then obtained the Fokker--Planck equation by a Wasserstein minimizing movement scheme, Otto formulated the corresponding formal Riemannian calculus, and Lott established connection and curvature formulae on the smooth positive-density locus \cite{jordan1998variational,otto2001geometry,lott2008geometric}. The metric theory of Ambrosio--Gigli--Savar\'e supplies the completed formulation used beyond smooth densities \cite{ambrosio2008gradient}.

Fisher and Rao initiated the differential geometry of statistical models, whilst Chentsov's theorem and modern information geometry explain the distinguished role of the Fisher metric and its dual coordinate systems \cite{fisher1922mathematical,rao1945information,cencov1982statistical,ay2017information}. Pistone--Sempi, Gibilisco--Pistone and Newton constructed nonparametric manifolds on which logarithms, divergences and statistical derivatives possess a rigorous global calculus \cite{pistone1995infinite,gibilisco1998connections,newton2012hilbert,newton2016balanced,newton2018differentiable,newton2019sobolev}. These carriers organise differentiable functionals of measures. Spatial transport introduces the further choice of a gradient, divergence and velocity norm.

Stein's operator method began as a means of characterising a target law and controlling approximation error \cite{stein1972bound,stein1986approximate}. Kernel Stein discrepancies convert the operator into a computable Hilbert norm \cite{liu2016kernelized,gorham2017measuring}, whilst Stein variational gradient descent converts the same Hilbert structure into a deterministic interacting particle system \cite{liu2016stein,liu2017stein}. Lu--Lu--Nolen proved a mean field limit, and subsequent work developed population and finite particle convergence estimates \cite{lu2019scaling,korba2020nonasymptotic,salim2022convergence,shi2023finite,banerjee2025improved}. Duncan--N\"usken--Szpruch and N\"usken--Renger constructed the associated tangent and cotangent geometries and related them to many particle and long time asymptotics \cite{duncan2023geometry,nusken2023large}. Operator valued reproducing kernels provide the natural language for vector fields on manifolds and Hilbert spaces \cite{micchelli2005learning,carmeli2010vector}.

Onsager's reciprocal relations already formulate dissipative evolution as a relation between thermodynamic forces and fluxes \cite{onsager1931reciprocal1,onsager1931reciprocal2}. Large deviations theory makes this relation quantitative: independent Brownian particles select the Wasserstein action, whilst correlated particle noise can select a kernel action \cite{adams2011large,duong2013wasserstein,erbar2015multiple,mielke2014variational,nusken2023large}. Agreement of the zero-cost evolution does not force agreement of the fluctuation action; this distinction becomes one of the interpretations of geometric equivalence below.

The contribution of this paper is the common particle-to-functional hierarchy and the comparison theory attached to it. The quotient--Riesz construction gives rigorous fibrewise Otto and Stein Hilbert pairs; the effective mobility $\Mop_{\rho}$ describes their relation on the common horizontal space; exact, time changed, observed and quantitative equivalences are separated; finite rank kernels are shown to realise a prescribed positive response on a chosen force subspace. Relative entropy, statistical submanifolds, Gaussian models, diffusion currents, resolvents and path laws then serve as instances of the same hierarchy. A smoothly varying Stein Hilbert bundle, connection or global flow requires uniform range and regularity hypotheses beyond the fibrewise construction.

\subsection{Organisation of the manuscript}

The first part constructs the hierarchy up to the choice of geometry. Smooth densities and their regular and distributional cotangents are treated in \S\ref{density-manifold-sec}. The diffeomorphism quotient and the continuity operator are established in \S\ref{continuity-operator-sec}. Hilbert velocity norms and their Riesz selections are introduced in \S\ref{hilbert-transport-sec}; Otto and Stein geometry follow in \S\ref{otto-sec} and \S\ref{stein-sec}. The central comparison theorem occupies \S\ref{geometry-comparison-sec}, where equality of particles, laws, time changes, observations and quadratic costs are distinguished.

The second part turns differentiable functionals into force fields. Regular first variations and relative entropy are treated in \S\ref{functional-sec}. Statistical submanifolds in \S\ref{statistical-submanifolds-sec} show how a prescribed finite dimensional evolution can be realised by particle fields, whilst the Gaussian model in \S\ref{gaussian-sec} makes law equivalent but particle inequivalent affine lifts explicit.

The third part descends towards particles and paths. Empirical measures and interacting systems are studied in \S\ref{particles-sec}. Diffusion laws, probability flow maps and superposition are compared in \S\ref{diffusion-transport-sec}. Scalar and Lions derivatives in \S\ref{lions-sec} identify the cotangent potential and its spatial force. Large deviations actions in \S\ref{large-deviation-sec} explain why equal deterministic evolutions may retain different microscopic fluctuation costs.

The final part returns to operators. Generators, semigroups and tangent propagation are separated in \S\ref{semigroup-sec}. Reversible semigroups and resolvents construct Stein solutions in \S\ref{resolvent-sec}. Feynman--Kac weights in \S\ref{feynman-kac-sec} give an equivalence after path evaluation and normalisation. Abstract Wiener spaces and path laws are treated in \S\ref{infinite-dimensional-sec}; \S\ref{synthesis-sec} gathers the hierarchy and its equivalence relations.

\begin{remark}\label{scope-rem}
	Compact base manifolds remove boundary and decay terms from the geometric core. Euclidean, noncompact and infinite dimensional extensions are stated under their respective coercivity, integrability, closability and quasi-regularity assumptions. Every comparison theorem is fibrewise unless a common chart and well posed flow have been specified. Uniform control of the varying kernel range and nullspace is the additional analytic requirement for a global Stein bundle or connection.
\end{remark}

Our main result is gathered together in the following statement (proven as Theorem \ref{common-otto-stein-thm}).
\begin{theorem*}
	Let $M$ be a connected closed Riemannian manifold, let $\Prob^{\infty}_+(M)$ be its Fr\'echet manifold of smooth positive probability densities and let $\Hspace_k\hookrightarrow C^1(M;TM)$ be a vector-valued reproducing kernel Hilbert space. Then the following statements hold.
	\begin{enumerate}[label=(\roman*)]
		\item The continuity operator $\Aop_{\rho}v=-\Div(\rho v)$ is the differential of the pushforward action, is surjective on the smooth core and identifies a density tangent with a particle velocity modulo the $\rho$ divergence free kernel.
		\item The weighted $L^2(\rho)$ norm and the reproducing kernel Hilbert norm produce fibrewise Otto and Stein tangent--cotangent Hilbert pairs by quotient and Riesz representation; their Onsager maps factor as $\Kop^{\bullet}_{\rho}=\Aop_{\rho}\Bop^{\bullet}_{\rho}$.
		\item Let $\Mop_{\rho}=\mathsf{P}^{\mathrm{hor}}_{\rho}\Iop_{\rho}\Iop_{\rho}^{*}$ on the Otto horizontal space. A regular covector $[f]$ induces the Otto tangent $\Aop_{\rho}\grad f$ and the Stein tangent $\Aop_{\rho}\Mop_{\rho}\grad f$. These tangents agree exactly when $\Mop_{\rho}\grad f=\grad f$, whilst the particle fields agree exactly when $\Iop_{\rho}\Iop_{\rho}^{*}\grad f=\grad f$.
		\item Positive proportionality of the two tangents gives the same unparametrised law curve, equality after a differentiable observation gives the same prescribed evolution of that observation, and spectral bounds on $\Mop_{\rho}$ compare dissipation and tangent cost. Agreement for every regular force in a closed force space is equivalent to $\Mop_{\rho}=\id$ there. Any positive self-adjoint response on a finite dimensional smooth force subspace can be realised by a finite rank reproducing kernel.
		\item Relative entropy supplies the common covector $[\log(\rho/\pi)]$. Its Wasserstein and Stein descents agree according to the preceding criterion; equality of their instantaneous dissipations is weaker than equality of their law tangents.
		\item Every regular minimal exponential family has smooth gradient solutions of the required Stein equations on the closed-manifold core. Membership in a prescribed kernel space is an additional range condition; when it holds, covector descent in expectation coordinates is realised by particle transport and agrees with the prescribed Fisher mirror evolution on the model.
	\end{enumerate}
\end{theorem*}

\subsection{Acknowledgements}

The author is grateful to K D Elworthy and J Glazebrook, as well as D Sullivan, for many useful discussions on global analysis in probability, and the fundamental notions of transport and Markov processes respectively. Also thanked warmly is M Sharp for motivating the development of The Machinery for parabolicity on Fr\'echet manifolds present in part here. The author acknowledges the Einstein Chair programme at the Graduate Centre of the City University of New York for support. 

\section{The differentiable manifold of positive probability densities}\label{density-manifold-sec}

\subsection{Smooth densities and their ambient affine space}

Let $M$ be a connected, closed, oriented smooth manifold of dimension $d$, furnished with a Riemannian metric $g$ and Riemannian volume measure $\mathrm{vol}$. Compactness will remain in force throughout the geometric construction; Euclidean and infinite dimensional formulae will be stated separately under their required decay and integrability hypotheses. Let $C^{\infty}(M)$ carry its usual Fr\'echet topology, generated by the seminorms $\|f\|_{C^r}$ for $r \in \N$. Set
\[
	C^{\infty}_0(M) = \left\{ \sigma \in C^{\infty}(M) : \int_M \sigma \, \dd{\mathrm{vol}} = 0 \right\}.
\]
The affine hyperplane of smooth functions with integral one is
\[
	\mathcal A^{\infty}_1(M) = \left\{ \rho \in C^{\infty}(M) : \int_M \rho \, \dd{\mathrm{vol}} = 1 \right\}.
\]
Its translation space is $C^{\infty}_0(M)$. We shall identify a density function $\rho$ with the measure $\rho \, \mathrm{vol}$ whenever no ambiguity can arise.

Let
\[
	\Prob^{\infty}_+(M) = \left\{ \rho \in \mathcal A^{\infty}_1(M) : \rho(x) > 0 \text{ for every } x \in M \right\}.
\]
Compactness makes strict positivity an open condition in the $C^0$ topology and hence in the Fr\'echet topology. One therefore obtains the following elementary manifold statement, which is the smooth carrier used by Lott \cite{lott2008geometric}.

\begin{prop}\label{smooth-density-manifold-prop}
	The set $\Prob^{\infty}_+(M)$ is an open Fr\'echet submanifold of the affine hyperplane $\mathcal A^{\infty}_1(M)$, modelled on $C^{\infty}_0(M)$. Its tangent space at every $\rho \in \Prob^{\infty}_+(M)$ is
	\[
		T_{\rho}\Prob^{\infty}_+(M) = C^{\infty}_0(M).
	\]
\end{prop}
\begin{proof}
	The continuous linear functional $I \colon C^{\infty}(M) \to \R$ defined by $I(f) = \int_M f \, \dd{\mathrm{vol}}$ is surjective. Its kernel $C^{\infty}_0(M)$ is complemented by the constant functions, whence $\mathcal A^{\infty}_1(M)$ is a split affine Fr\'echet submanifold. Every $\rho \in \Prob^{\infty}_+(M)$ has a positive minimum $m_{\rho}$. The $C^0$ ball of radius $m_{\rho}/2$ about $\rho$ lies in the positive cone; consequently, $\Prob^{\infty}_+(M)$ is open in $\mathcal A^{\infty}_1(M)$. The tangent space of an open subset of an affine Fr\'echet space is its translation space, which gives the assertion.
\end{proof}

The chart $\chi_{\rho_0}(\rho) = \rho - \rho_0$ is sufficient for the differential topology used below. Newton's mixture, exponential and balanced charts give stronger statistical structures, especially when densities are defined relative to a general reference measure and when logarithmic integrability must be controlled \cite{newton2012hilbert,newton2016balanced,newton2018differentiable,newton2019sobolev}. The affine chart is retained here because the transport operator differentiates densities as signed measures, whilst logarithmic charts will reappear when relative entropy and exponential families are discussed.

\subsection{The full and regular cotangent spaces}

The tangent identification in Proposition \ref{smooth-density-manifold-prop} is elementary; its continuous dual requires more care. Let $\mathcal D'(M)$ denote the distributions on $M$. The continuous dual of $C^{\infty}(M)$ is $\mathcal D'(M)$, and the annihilator of $C^{\infty}_0(M)$ consists precisely of the constant multiples of the distribution $f \mapsto \int_M f \, \dd{\mathrm{vol}}$.

\begin{prop}\label{full-cotangent-prop}
	The continuous Fr\'echet cotangent at $\rho \in \Prob^{\infty}_+(M)$ is canonically
	\[
		T_{\rho}^{*}\Prob^{\infty}_+(M) \cong \mathcal D'(M) / \R.
	\]
	A smooth function $f$ determines the regular covector $[f]$ through
	\[
		\langle [f], \sigma \rangle = \int_M f \sigma \, \dd{\mathrm{vol}},
	\]
	and two smooth functions determine the same regular covector exactly when their difference is constant.
\end{prop}
\begin{proof}
	Since $C^{\infty}_0(M)$ is a closed subspace of $C^{\infty}(M)$, the restriction map $\mathcal D'(M) \to (C^{\infty}_0(M))'$ is surjective by the Hahn--Banach theorem applied to locally convex spaces. Its kernel is the annihilator of $C^{\infty}_0(M)$. Let $T$ belong to this annihilator. Every $f \in C^{\infty}(M)$ decomposes as
	\[
		f = \left(f - \frac{1}{\mathrm{vol}(M)}\int_M f \, \dd{\mathrm{vol}}\right) + \frac{1}{\mathrm{vol}(M)}\int_M f \, \dd{\mathrm{vol}}.
	\]
	The first summand has integral zero; hence $T(f)$ is a constant multiple of $\int_M f \, \dd{\mathrm{vol}}$. The kernel is therefore one dimensional and is denoted by $\R$. The assertion concerning regular covectors follows by applying the same argument to distributions induced by smooth functions.
\end{proof}

\begin{remark}\label{regular-cotangent-rem}
	The quotient $C^{\infty}(M)/\R$ is a distinguished regular subspace of the cotangent. The complete Fr\'echet cotangent is the larger distributional quotient in Proposition \ref{full-cotangent-prop}. Variational derivatives such as $\log(\rho/\pi)$ usually lie in the regular subspace. Weak transport metrics complete another regular subspace in norms which depend on $\rho$; the Wasserstein and Stein cotangents constructed later are Hilbert cotangent models attached to the Fr\'echet manifold and occupy a different analytic role from its complete continuous dual.
\end{remark}

The probability constraint explains the quotient by constants. A functional derivative $f$ may be changed by any $\rho$ dependent scalar without changing its evaluation on a zero mass tangent. This elementary gauge freedom persists in the Lions derivative, exponential coordinates and Fisher--Rao reaction terms.

\subsection{Sobolev and statistical manifold carriers}

A Hilbert manifold carrier can be obtained by fixing $s > d/2$ and taking
\[
	\Prob^s_+(M) = \left\{ \rho \in H^s(M) : \rho > 0, \int_M \rho \, \dd{\mathrm{vol}} = 1 \right\}.
\]
Sobolev embedding makes positivity open; thus $\Prob^s_+(M)$ is a Hilbert manifold modelled on $H^s_0(M)$. Its Hilbert dual is $H^{-s}(M)/\R$. Higher values of $s$ permit multiplication and differential operations, although a fixed Sobolev index does not remove every derivative loss arising from the diffeomorphism action. The Fr\'echet scale, a tame category, or an inverse limit of Sobolev manifolds is consequently better adapted to the full pushforward action \cite{ebin1970groups,khesin2013geometry}.

Newton's constructions answer a complementary question. They begin with classes of measures equivalent to a reference measure and choose balanced, logarithmic, differentiable or Sobolev charts so that statistical divergences, Fisher information and Amari's covariant derivatives become smooth \cite{newton2012hilbert,newton2016balanced,newton2018differentiable,newton2019sobolev}. Pistone--Sempi use Orlicz spaces to obtain an exponential manifold containing all measures connected by exponential arcs \cite{pistone1995infinite}; Gibilisco-- Pistone develop the corresponding nonparametric connections \cite{gibilisco1998connections}. These carriers rigorously organise the information geometry of densities. Otto and Stein geometry require additional spatial data on the sample space, namely a gradient, a divergence and a velocity norm. The construction below may be placed over any density manifold for which these operations are defined and smooth on a common core.

\begin{remark}\label{carrier-choice-rem}
	No single topology is forced by all the structures. Statistical charts control logarithms and moments, the diffeomorphism quotient controls spatial transport, and weak tangent norms control dissipative evolution. The present paper uses $\Prob^{\infty}_+(M)$ as the common differentiable carrier and constructs Hilbert tangent and cotangent completions fibrewise. Newton's spaces supply alternatives when the reference measure, tail behaviour or filtering problem determines a more suitable ambient topology.
\end{remark}

\section{The diffeomorphism quotient and the continuity operator}\label{continuity-operator-sec}

\subsection{Pushforwards and their differential}

Let $\Diff(M)$ denote the Fr\'echet Lie group of smooth orientation preserving diffeomorphisms of $M$. It acts on positive densities by the pushforward. Fix $\rho \in \Prob^{\infty}_+(M)$ and define
\[
	\Pi_{\rho} \colon \Diff(M) \to \Prob^{\infty}_+(M)
\]
by
\[
	\Pi_{\rho}(\varphi) = \varphi_{\#}(\rho \, \mathrm{vol}).
\]
The Lie algebra of $\Diff(M)$ is $\Vect(M)$, the smooth vector fields on $M$. A smooth curve $\varphi_t$ has Eulerian velocity $v_t = \dot{\varphi}_t \circ \varphi_t^{-1}$.

\begin{lemma}\label{pushforward-differential-lem}
	Let $\rho_t = (\varphi_t)_{\#}(\rho_0 \, \mathrm{vol})$ and let $v_t$ be the Eulerian velocity of $\varphi_t$. Then
	\[
		\partial_t \rho_t + \Div(\rho_t v_t) = 0.
	\]
	Consequently, the differential of the orbit map at the identity is
	\[
		T_{\id}\Pi_{\rho}(v) = \Aop_{\rho}v = - \Div(\rho v).
	\]
\end{lemma}
\begin{proof}
	Let $f \in C^{\infty}(M)$. The definition of the pushforward gives
	\[
		\int_M f(x) \rho_t(x) \, \dd{\mathrm{vol}}(x) = \int_M f(\varphi_t(x)) \rho_0(x) \, \dd{\mathrm{vol}}(x).
	\]
	Differentiating and changing variables yields
	\[
		\dv{}{t}\int_M f \rho_t \, \dd{\mathrm{vol}} = \int_M \langle \grad f, v_t \rangle_g \rho_t \, \dd{\mathrm{vol}}.
	\]
	The manifold has no boundary; integration by parts therefore gives
	\[
		\dv{}{t}\int_M f \rho_t \, \dd{\mathrm{vol}} = - \int_M f \Div(\rho_t v_t) \, \dd{\mathrm{vol}}.
	\]
	The weak identity holds for every smooth $f$ and all quantities are smooth, which proves the continuity equation. Evaluation at $t=0$ proves the formula for the differential.
\end{proof}

The operator $\Aop_{\rho}$ will be called the continuity operator. Its domain and norm will vary, but its distributional formula remains fixed. The sign convention has been chosen so that a velocity $v$ in the continuity equation represents the tangent $\dot{\rho} = \Aop_{\rho}v$.

\subsection{Kernel, range and the quotient of velocities}

The kernel is
\[
	\Ker \Aop_{\rho} = \left\{ v \in \Vect(M) : \Div(\rho v) = 0 \right\}.
\]
These fields generate infinitesimal diffeomorphisms preserving the measure $\rho \, \mathrm{vol}$. They are vertical directions of the orbit map. Surjectivity on smooth density tangents follows from a weighted elliptic equation.

\begin{lemma}\label{weighted-poisson-lem}
	Let $\rho \in \Prob^{\infty}_+(M)$ and $\sigma \in C^{\infty}_0(M)$. There exists a smooth function $\phi$, unique modulo constants, such that
	\[
		- \Div(\rho \grad \phi) = \sigma.
	\]
\end{lemma}
\begin{proof}
	Consider the bilinear form on $H^1(M)/\R$
	\[
		a_{\rho}([\phi],[\psi]) = \int_M \rho \langle \grad \phi, \grad \psi \rangle_g \, \dd{\mathrm{vol}}.
	\]
	Strict positivity and compactness give constants $0 < m_{\rho} \leqslant M_{\rho}$ such that $m_{\rho} \leqslant \rho \leqslant M_{\rho}$. Poincar\'e's inequality consequently makes $a_{\rho}$ coercive on the quotient. The functional $[\psi] \mapsto \int_M \sigma \psi \, \dd{\mathrm{vol}}$ is well defined because $\sigma$ has integral zero. Lax--Milgram supplies a weak solution in $H^1(M)/\R$. The operator $-\Div(\rho \grad)$ is uniformly elliptic with smooth coefficients; elliptic regularity implies $\phi \in C^{\infty}(M)$. The energy identity shows that a homogeneous solution has zero gradient and is therefore constant.
\end{proof}

\begin{prop}\label{continuity-quotient-prop}
	The continuity operator is surjective from $\Vect(M)$ onto $C^{\infty}_0(M)$, and it induces a vector space isomorphism
	\[
		\faktor{\Vect(M)}{\Ker \Aop_{\rho}} \cong T_{\rho}\Prob^{\infty}_+(M).
	\]
\end{prop}
\begin{proof}
	Lemma \ref{weighted-poisson-lem} represents every $\sigma$ as $\Aop_{\rho}(\grad \phi)$, proving surjectivity. The quotient statement is the first isomorphism theorem.
\end{proof}

Moser's theorem shows globally that $\Diff(M)$ acts transitively on smooth positive densities of fixed total mass \cite{moser1965volume}. If $\Diff_{\rho}(M)$ denotes the subgroup preserving $\rho \, \mathrm{vol}$, one therefore has the homogeneous space representation
\[
	\Prob^{\infty}_+(M) \cong \faktor{\Diff(M)}{\Diff_{\rho}(M)}.
\]
Proposition \ref{continuity-quotient-prop} is its infinitesimal form. The quotient already explains why particle velocities and density tangents cannot be identified canonically: a horizontal rule must choose one velocity from every coset.

\subsection{Regular covectors evaluated on transport tangents}

Let $[f]$ be a regular covector and let $\sigma = \Aop_{\rho}v$. Integration by parts gives
\begin{equation}\label{transport-pairing-eq}
	\langle [f], \sigma \rangle = \int_M f \sigma \, \dd{\mathrm{vol}} = \int_M \rho \langle \grad f, v \rangle_g \, \dd{\mathrm{vol}}.
\end{equation}
The right hand side depends only on the coset of $v$. Indeed, if $w \in \Ker \Aop_{\rho}$, then
\[
	\int_M \rho \langle \grad f,w \rangle_g \, \dd{\mathrm{vol}} = 0.
\]
Formula \eqref{transport-pairing-eq} is the common duality from which Otto and Stein musical maps will be obtained.

\begin{remark}\label{three-objects-rem}
	Formula \eqref{transport-pairing-eq} contains the middle levels of the hierarchy. The scalar $[f]$ is a density covector, $\grad f$ is its force representative on the sample space, $v$ is a particle velocity and $\sigma=\Aop_{\rho}v$ is a density tangent. A differentiable functional lies one level above $[f]$, whilst evaluation of $v$ on particles lies one level below it. The geometry appears between the force and the velocity; the continuity operator appears between the velocity and the law.
\end{remark}

\section{Gradient forces and Hilbert selections of particle velocity}\label{hilbert-transport-sec}

Fix a density $\rho$. A regular covector $[f]$ produces the force $\grad f$ through the transport pairing \eqref{transport-pairing-eq}. A transport geometry chooses a Hilbert class of admissible particle velocities and uses the Riesz theorem to convert this force into one selected velocity. The continuity operator then produces the corresponding law tangent. This construction is the common middle of the Otto and Stein hierarchies.

Let $V_{\rho}$ be a real Hilbert space continuously embedded in distributional vector fields on $M$, and suppose that $\Aop_{\rho}\colon V_{\rho}\to\mathcal D'(M)$ is continuous. Set $N_{\rho}=\Ker\Aop_{\rho}$ and $H_{\rho}=N_{\rho}^{\perp}$. Give $\Aop_{\rho}(V_{\rho})$ the quotient norm
\begin{equation}\label{quotient-tangent-norm-eq}
	\|\sigma\|_{T_{\rho,V}}=\inf\left\{\|v\|_{V_{\rho}}:v\in V_{\rho},\ \Aop_{\rho}v=\sigma\right\}.
\end{equation}
The kernel is closed, whence $V_{\rho}/N_{\rho}$ is Hilbert. The restriction $\Aop_{\rho}|_{H_{\rho}}$ is consequently an isometric isomorphism from $H_{\rho}$ onto $\Aop_{\rho}(V_{\rho})$ with the quotient norm. The range is already complete and will be denoted by $T_{\rho,V}$.

Let $\mathcal C_{\rho}$ be a linear class of regular covectors such that
\[
	v\longmapsto\int_M\rho\langle\grad f,v\rangle_g\,\dd{\mathrm{vol}}
\]
is continuous on $V_{\rho}$ for every $f\in\mathcal C_{\rho}$. The Riesz representation theorem supplies a unique $\Bop_{\rho}f\in V_{\rho}$ satisfying
\begin{equation}\label{velocity-riesz-eq}
	\langle\Bop_{\rho}f,v\rangle_{V_{\rho}}=\int_M\rho\langle\grad f,v\rangle_g\,\dd{\mathrm{vol}}.
\end{equation}
Every $\Bop_{\rho}f$ belongs to $H_{\rho}$ because the right hand side vanishes on $N_{\rho}$. Define
\begin{equation}\label{onsager-definition-eq}
	\Kop_{\rho}f=\Aop_{\rho}\Bop_{\rho}f.
\end{equation}

\begin{prop}\label{abstract-transport-prop}
	The form
	\[
		\langle f,h\rangle_{T^{*}_{\rho,V}}=\langle\Bop_{\rho}f,\Bop_{\rho}h\rangle_{V_{\rho}}
	\]
	is positive semidefinite on $\mathcal C_{\rho}$. Quotienting by $\Ker\Bop_{\rho}$ and completing gives a Hilbert cotangent $T^{*}_{\rho,V}$. The operator $\Kop_{\rho}$ extends uniquely to an isometric isomorphism
	\[
		\Kop_{\rho}\colon T^{*}_{\rho,V}\to\Aop_{\rho}\overline{\Bop_{\rho}(\mathcal C_{\rho})}^{V_{\rho}}\subseteq T_{\rho,V}.
	\]
	Regular covectors satisfy
	\begin{equation}\label{onsager-duality-eq}
		\langle f,\Kop_{\rho}h\rangle=\langle\Bop_{\rho}f,\Bop_{\rho}h\rangle_{V_{\rho}}.
	\end{equation}
	If the only $w\in H_{\rho}$ satisfying $\int_M\rho\langle\grad f,w\rangle_g\,\dd{\mathrm{vol}}=0$ for every $f\in\mathcal C_{\rho}$ is $w=0$, then $\Bop_{\rho}(\mathcal C_{\rho})$ is dense in $H_{\rho}$ and $\Kop_{\rho}$ is onto the full tangent $T_{\rho,V}$.
\end{prop}
\begin{proof}
	The quotient $\mathcal C_{\rho}/\Ker\Bop_{\rho}$ is isometric to $\Bop_{\rho}(\mathcal C_{\rho})$, so its completion is $\overline{\Bop_{\rho}(\mathcal C_{\rho})}^{V_{\rho}}$. The restriction of $\Aop_{\rho}$ to $H_{\rho}$ is an isometry, whence $\Aop_{\rho}\Bop_{\rho}$ extends to the stated isometry. Formula \eqref{onsager-duality-eq} follows from \eqref{transport-pairing-eq} and \eqref{velocity-riesz-eq}. The orthogonal complement in $H_{\rho}$ of $\Bop_{\rho}(\mathcal C_{\rho})$ is zero under the final hypothesis, which proves density.
\end{proof}

\begin{remark}\label{raw-completed-rem}
	The isomorphism in Proposition \ref{abstract-transport-prop} is formed from the continuity pairing and the chosen Hilbert norm. A raw integral or differential operator may remain noninvertible on a preassigned Sobolev space. Compact kernel operators commonly fail to be surjective there; quotienting the nullspace and using the operator induced norms produces the Hilbert pair on which the musical relation is valid.
\end{remark}

\begin{remark}\label{fibrewise-status-rem}
	The construction is fibrewise. A smooth weak Riemannian or cometric structure on a manifold of densities additionally requires a common model for the fibres and smooth dependence of $N_{\rho}$, $\Bop_{\rho}$ and the induced norms on $\rho$. The Otto operator has this property on uniformly positive Sobolev charts under \S\ref{elliptic-appendix-sec}. Variation of the range and nullspace of a general kernel is a separate analytic problem.
\end{remark}

A differentiable functional $\mathcal F$ whose differential belongs to $T^{*}_{\rho,V}$ descends through the hierarchy
\[
	\mathcal F\longmapsto\dd{\mathcal F}(\rho)=[f]\longmapsto\grad f\longmapsto-\Bop_{\rho}f\longmapsto-\Aop_{\rho}\Bop_{\rho}f.
\]
The resulting law evolution is
\begin{equation}\label{abstract-gradient-flow-eq}
	\partial_t\rho_t=-\Kop_{\rho_t}\dd{\mathcal F}(\rho_t),
\end{equation}
and the selected particle velocity is $-\Bop_{\rho_t}\dd{\mathcal F}(\rho_t)$. Whenever the chain rule and evolution are justified,
\begin{equation}\label{abstract-dissipation-eq}
	\dv{}{t}\mathcal F(\rho_t)=-\|\dd{\mathcal F}(\rho_t)\|_{T^{*}_{\rho_t,V}}^2=-\|\partial_t\rho_t\|_{T_{\rho_t,V}}^2.
\end{equation}
\section{Otto geometry}\label{otto-sec}

\subsection{Weighted Sobolev cotangents and tangents}

Fix $\rho \in \Prob^{\infty}_+(M)$. The weighted homogeneous Sobolev space is
\[
	\dot{H}^{1}_{\rho}(M) = H^1(M)/\R
\]
with inner product
\begin{equation}\label{otto-cotangent-inner-eq}
	\langle [f],[h] \rangle_{\dot{H}^{1}_{\rho}} = \int_M \rho \langle \grad f,\grad h \rangle_g \, \dd{\mathrm{vol}}.
\end{equation}
Strict positivity of $\rho$ and Poincar\'e's inequality make this norm equivalent to the usual quotient $H^1$ norm. Its Hilbert dual is denoted by $\dot{H}^{-1}_{\rho}(M)$.

Take $V_{\rho}^{\Wass}=L^2(\rho;TM)$ and initially let $\Aop_{\rho}$ act distributionally. Formula \eqref{velocity-riesz-eq} has the Riesz representative
\[
	\Bop_{\rho}^{\Wass}f = \grad f.
\]
The corresponding Onsager operator is
\begin{equation}\label{wasserstein-onsager-eq}
	\Kop_{\rho}^{\Wass}f = - \Div(\rho \grad f).
\end{equation}

\begin{theorem}\label{otto-riesz-thm}
	The operator $\Kop_{\rho}^{\Wass}$ extends uniquely to an isometric isomorphism
	\[
		\Kop_{\rho}^{\Wass} \colon \dot{H}^{1}_{\rho}(M) \to \dot{H}^{-1}_{\rho}(M).
	\]
	Every smooth zero mass tangent $\sigma$ has a unique zero mean weak solution $\phi_{\sigma}$ of
	\[
		- \Div(\rho \grad \phi_{\sigma}) = \sigma.
	\]
	This solution determines the norm
	\begin{equation}\label{otto-tangent-norm-eq}
		\|\sigma\|_{\dot{H}^{-1}_{\rho}}^2 = \int_M \rho |\grad \phi_{\sigma}|_g^2 \, \dd{\mathrm{vol}}.
	\end{equation}
\end{theorem}
\begin{proof}
	The inner product \eqref{otto-cotangent-inner-eq} identifies $\dot{H}^{1}_{\rho}$ with its dual by the Riesz theorem. The weak operator associated with the form is exactly $-\Div(\rho\grad)$, since
	\[
		\langle [f],\Kop_{\rho}^{\Wass}[h] \rangle = \int_M \rho \langle \grad f,\grad h \rangle_g \, \dd{\mathrm{vol}}.
	\]
	The induced map is therefore an isometric isomorphism onto $\dot{H}^{-1}_{\rho}$. Taking $h=\phi_{\sigma}$ gives \eqref{otto-tangent-norm-eq}.
\end{proof}

\subsection{Horizontal gradients and minimal kinetic energy}

Weighted Helmholtz decomposition gives
\[
	L^2(\rho;TM) = \overline{\{ \grad f : f \in C^{\infty}(M) \}}^{L^2(\rho)} \oplus \Ker \Aop_{\rho}.
\]
The first summand is the Wasserstein horizontal space. It consists of the metric tangent vectors usually assigned to a measure in the Wasserstein space \cite{ambrosio2008gradient,villani2003topics,villani2009optimal}.

\begin{prop}\label{minimal-kinetic-prop}
	Let $\sigma \in C^{\infty}_0(M)$ and let $\phi_{\sigma}$ be as in Theorem \ref{otto-riesz-thm}. The gradient $\grad \phi_{\sigma}$ uniquely minimises the weighted kinetic energy, modulo equality in $L^2(\rho)$, amongst all $v \in L^2(\rho;TM)$ satisfying $-\Div(\rho v)=\sigma$. Moreover,
	\[
		\inf_{-\Div(\rho v)=\sigma}\int_M |v|_g^2 \rho \, \dd{\mathrm{vol}} = \int_M |\grad \phi_{\sigma}|_g^2 \rho \, \dd{\mathrm{vol}}.
	\]
\end{prop}
\begin{proof}
	Every admissible $v$ decomposes as $v=\grad\phi_{\sigma}+w$ with $w\in\Ker\Aop_{\rho}$. Formula \eqref{transport-pairing-eq} gives
	\[
		\int_M \rho \langle \grad\phi_{\sigma},w\rangle_g \, \dd{\mathrm{vol}} = \langle [\phi_{\sigma}],\Aop_{\rho}w\rangle = 0.
	\]
	Pythagoras' identity yields
	\[
		\|v\|_{L^2(\rho)}^2 = \|\grad\phi_{\sigma}\|_{L^2(\rho)}^2 + \|w\|_{L^2(\rho)}^2,
	\]
	which proves the assertion.
\end{proof}

Smooth tangents $\sigma$ and $\tau$ carry Otto's weak Riemannian metric
\begin{equation}\label{otto-metric-eq}
	G^{\Wass}_{\rho}(\sigma,\tau) = \int_M \rho \langle \grad\phi_{\sigma},\grad\phi_{\tau}\rangle_g \, \dd{\mathrm{vol}}.
\end{equation}
The metric is weak relative to the Fr\'echet topology: the norm in \eqref{otto-tangent-norm-eq} is of order minus one. Lott's connection, curvature and geodesic calculations are rigorous on $\Prob^{\infty}_+(M)$ whenever the relevant curves remain smooth and positive \cite{lott2008geometric}. The length metric associated with \eqref{otto-metric-eq} agrees on smooth immersed curves with the quadratic Wasserstein length. The metric completion reaches the full quadratic Wasserstein space $\Prob_2(M)$, whose tangent at a general measure is the $L^2(\rho)$ closure of gradients rather than a Fr\'echet tangent space \cite{ambrosio2008gradient}.

\subsection{Covectors, vector fields and density tangents}

The Otto dictionary at a smooth positive density may now be stated without suppressing any map. A regular cotangent potential is $[f]$. Its base manifold differential is $\grad f \in L^2(\rho;TM)$. The Wasserstein velocity representative is
\[
	\Bop^{\Wass}_{\rho}f = \grad f.
\]
The Wasserstein density tangent is
\[
	\Kop^{\Wass}_{\rho}f = \Aop_{\rho}\grad f.
\]
Conversely, a tangent $\sigma$ determines the potential $[\phi_{\sigma}]=(\Kop^{\Wass}_{\rho})^{-1}\sigma$ and horizontal velocity $\grad\phi_{\sigma}$. The musical notation therefore denotes an operator between the Hilbert cotangent and tangent completions, whilst the spatial gradient denotes an intermediate map into the velocity space.

\begin{remark}\label{otto-sharp-rem}
	The phrase ``Wasserstein gradient'' is used for both $\grad f$ and $-\Div(\rho\grad f)$ in the literature. The former is the horizontal velocity field and the latter is the density tangent. Every formula below displays $\Bop_{\rho}$ and $\Aop_{\rho}$ separately whenever this distinction matters.
\end{remark}

\section{Stein geometry}\label{stein-sec}

\subsection{Vector-valued reproducing kernels}

Let $k$ be a smooth positive definite vector bundle kernel on $TM$. Thus $k(x,y) \colon T_xM \to T_yM$ is linear, $k(y,x)=k(x,y)^{*}$, and for every finite collection $x_1,\ldots,x_n \in M$ and $\xi_i \in T_{x_i}M$ one has
\[
	\sum_{i,j=1}^{n} \langle k(x_i,x_j)\xi_i,\xi_j\rangle_g \geqslant 0.
\]
Let $\Hspace_k$ be the associated Hilbert space of vector fields, characterised by the reproducing identity
\begin{equation}\label{vector-reproducing-eq}
	\langle v,k(x,\cdot)\xi\rangle_{\Hspace_k} = \langle v(x),\xi\rangle_g.
\end{equation}
Smoothness of $k$ and compactness of $M$ imply continuous inclusions into spaces of continuously differentiable vector fields after a sufficient number of kernel derivatives has been assumed \cite{aronszajn1950theory,micchelli2005learning,carmeli2010vector}. We henceforth assume
\begin{equation}\label{kernel-inclusion-eq}
	\Hspace_k \hookrightarrow C^1(M;TM)
\end{equation}
continuously. A scalar kernel on a Euclidean space gives the familiar special case $k(x,y)\operatorname{id}_{\R^d}$.

Fix $\rho \in \Prob^{\infty}_+(M)$ and let
\[
	\Iop_{\rho} \colon \Hspace_k \to L^2(\rho;TM)
\]
be the inclusion. It is bounded, and its Hilbert adjoint satisfies
\begin{equation}\label{kernel-adjoint-eq}
	(\Iop_{\rho}^{*}u)(y) = \int_M k(x,y)u(x)\rho(x) \, \dd{\mathrm{vol}}(x).
\end{equation}
The equality is first verified against kernel sections and then extended by density. Indeed, for $v \in \Hspace_k$ one obtains from \eqref{vector-reproducing-eq}
\[
	\left\langle v,\int_M k(x,\cdot)u(x)\rho(x) \, \dd{\mathrm{vol}}(x)\right\rangle_{\Hspace_k} = \int_M \langle v(x),u(x)\rangle_g \rho(x) \, \dd{\mathrm{vol}}(x).
\]

Restrict the continuity operator to $\Hspace_k$ and set
\[
	\Aop^{\Stein}_{\rho}v = -\Div(\rho v).
\]
A regular covector $[f]$ turns formula \eqref{velocity-riesz-eq} into
\[
	\langle \Bop^{\Stein}_{\rho}f,v\rangle_{\Hspace_k} = \int_M \rho \langle \grad f,v\rangle_g \, \dd{\mathrm{vol}}.
\]
Formula \eqref{kernel-adjoint-eq} gives
\begin{equation}\label{stein-velocity-map-eq}
	\Bop^{\Stein}_{\rho}f = \Iop_{\rho}^{*}\grad f.
\end{equation}
The Stein Onsager operator is therefore
\begin{equation}\label{stein-onsager-eq}
	\Kop^{\Stein}_{\rho}f = -\Div\bigl(\rho \Iop_{\rho}^{*}\grad f\bigr).
\end{equation}

\subsection{The tangent and cotangent Hilbert spaces}

Let
\[
	N^{\Stein}_{\rho}=\left\{v\in\Hspace_k:\Div(\rho v)=0\right\}
\]
and let $H^{\Stein}_{\rho}=(N^{\Stein}_{\rho})^{\perp}$. Define $T_{\rho,\Stein}=\Aop^{\Stein}_{\rho}(\Hspace_k)$ and give it the quotient norm
\begin{equation}\label{stein-tangent-norm-eq}
	\|\sigma\|_{T_{\rho,\Stein}}^2=\inf\left\{\|v\|_{\Hspace_k}^2:-\Div(\rho v)=\sigma\right\}.
\end{equation}
The quotient by the closed kernel is Hilbert, so $T_{\rho,\Stein}$ is already complete and $\Aop^{\Stein}_{\rho}|_{H^{\Stein}_{\rho}}$ is an isometric isomorphism onto it. Define on $C^{\infty}(M)$
\begin{equation}\label{stein-cotangent-inner-eq}
	\langle f,h\rangle_{T^{*}_{\rho,\Stein}}=\left\langle\Iop_{\rho}^{*}\grad f,\Iop_{\rho}^{*}\grad h\right\rangle_{\Hspace_k}.
\end{equation}
Using the adjoint twice gives
\begin{equation}\label{stein-cotangent-kernel-eq}
	\langle f,h\rangle_{T^{*}_{\rho,\Stein}}=\int_{M\times M}\left\langle\grad f(x),k(y,x)\grad h(y)\right\rangle_g\rho(x)\rho(y)\,\dd{\mathrm{vol}}(x)\dd{\mathrm{vol}}(y).
\end{equation}
The nullspace may contain more than the constants. Quotient by all $f$ for which $\Iop_{\rho}^{*}\grad f=0$ and complete; the resulting Hilbert space is denoted by $T^{*}_{\rho,\Stein}$.

\begin{theorem}\label{stein-musical-thm}
	The operator in \eqref{stein-onsager-eq} extends uniquely to an isometric isomorphism
	\[
		\Kop^{\Stein}_{\rho}\colon T^{*}_{\rho,\Stein}\to T_{\rho,\Stein}.
	\]
	Its minimum norm velocity representative is $\Bop^{\Stein}_{\rho}f=\Iop_{\rho}^{*}\grad f$, and regular covectors satisfy
	\begin{equation}\label{stein-musical-duality-eq}
		\langle f,\Kop^{\Stein}_{\rho}h\rangle=\langle f,h\rangle_{T^{*}_{\rho,\Stein}}.
	\end{equation}
\end{theorem}
\begin{proof}
	Apply Proposition \ref{abstract-transport-prop} with $V_{\rho}=\Hspace_k$ and $\mathcal C_{\rho}=C^{\infty}(M)$. It remains to prove that $\Iop_{\rho}^{*}\grad C^{\infty}(M)$ is dense in $H^{\Stein}_{\rho}$. Let $w\in H^{\Stein}_{\rho}$ be orthogonal to this range. One then has
	\[
		0=\left\langle w,\Iop_{\rho}^{*}\grad f\right\rangle_{\Hspace_k}=\int_M\rho\langle w,\grad f\rangle_g\,\dd{\mathrm{vol}}=\langle\Aop^{\Stein}_{\rho}w,f\rangle
	\]
	for every $f\in C^{\infty}(M)$. Thus $w\in N^{\Stein}_{\rho}$; orthogonality gives $w=0$.
\end{proof}

Integral strict positive definiteness makes the cotangent nullspace amongst gradients consist only of constants; related assumptions are used by N\"usken and Renger \cite{nusken2023large}. The theorem retains the full nullspace and consequently covers kernels which cannot detect every regular force.

\begin{remark}\label{stein-raw-operator-rem}
	The integral map $\Iop_{\rho}\Iop_{\rho}^{*}\colon L^2(\rho;TM)\to L^2(\rho;TM)$ is commonly compact. Its inverse is then unbounded on its range and the map is not surjective on the ambient $L^2$ space. Theorem \ref{stein-musical-thm} concerns the quotient Hilbert spaces induced by this map, which is compatible with noninvertibility of the raw kernel operator on ordinary Sobolev spaces.
\end{remark}

\subsection{Comparison of the Hilbert completions}

The bounded inclusion $\Iop_{\rho}$ gives
\begin{equation}\label{stein-otto-cotangent-bound-eq}
	\|f\|_{T^{*}_{\rho,\Stein}}=\|\Iop_{\rho}^{*}\grad f\|_{\Hspace_k}\leqslant\|\Iop_{\rho}\|\|\grad f\|_{L^2(\rho)}.
\end{equation}
Hence the identity on the common smooth core extends to a bounded map
\[
	J_{\rho}\colon\dot H^1_{\rho}(M)\to T^{*}_{\rho,\Stein}.
\]
The map need not be injective; its kernel consists of Otto covectors invisible to the chosen kernel. Its range is dense by the definition of the Stein cotangent completion.

\begin{prop}\label{stein-otto-embedding-prop}
	The adjoint
	\[
		J_{\rho}^{*}\colon T_{\rho,\Stein}\to\dot H^{-1}_{\rho}(M)
	\]
	is a continuous injection. Its action on regular Stein tangents is the natural inclusion of the distribution $-\Div(\rho v)$ into $\dot H^{-1}_{\rho}(M)$.
\end{prop}
\begin{proof}
	The adjoint of a bounded map with dense range is injective. The identification on regular elements follows from the common continuity pairing \eqref{transport-pairing-eq}.
\end{proof}

The comparison is fibrewise. The Stein cotangent norm is weaker on the common core, whilst every Stein tangent has a canonical minimum-$\Hspace_k$ velocity representative. Neither assertion gives smooth dependence of these quotient spaces on $\rho$ or global existence of the nonlinear Stein flow.

\subsection{Restricted Riesz representation and projection}

The Stein velocity maximises the regularised first variation
\begin{equation}\label{stein-variational-velocity-eq}
	\Bop^{\Stein}_{\rho}f = \operatorname*{arg\,max}_{v\in\Hspace_k}\left\{ \int_M \rho\langle\grad f,v\rangle_g \, \dd{\mathrm{vol}} - \frac12\|v\|_{\Hspace_k}^2 \right\}.
\end{equation}
The Euler equation for this strictly concave problem is precisely \eqref{velocity-riesz-eq}. Formula \eqref{stein-variational-velocity-eq} gives the universally valid sense in which a Stein field is the gradient of a restricted differential: it is the Riesz representative in the reproducing kernel Hilbert norm.

An $L^2(\rho)$ orthogonal projection statement requires additional structure. Suppose $\Hspace_k$ is realised as a closed subspace of $L^2(\rho;TM)$ and its Hilbert norm equals the inherited $L^2(\rho)$ norm. Then $\Iop_{\rho}^{*}$ is the orthogonal projection and $\Bop^{\Stein}_{\rho}f$ is the $L^2(\rho)$ projection of $\grad f$. General reproducing kernel norms do not satisfy this identity. Their adjoint inclusion is an integral regularising operator, and its Hilbert norm differs from the Wasserstein velocity norm.

\section{Agreement of the geometries in prescribed comparisons of their evolutions}\label{geometry-comparison-sec}

The Otto and Stein constructions begin with the same scalar covector and the same spatial force. Their difference can therefore be expressed by one positive operator on the Otto horizontal space. This operator determines when the induced particles, laws, clocks, observations and quadratic costs agree.

\subsection{The common force and the effective mobility}\label{effective-mobility-subsec}

Let
\[
	H^{\Wass}_{\rho}=\overline{\{\grad f:f\in C^{\infty}(M)\}}^{L^2(\rho)}=(\Ker\Aop_{\rho})^{\perp}
\]
and let $\mathsf{P}^{\mathrm{hor}}_{\rho}$ be the weighted $L^2(\rho)$ orthogonal projection onto this space. Define
\[
	\mathsf{S}_{\rho}=\Iop_{\rho}\Iop_{\rho}^{*}\colon L^2(\rho;TM)\to L^2(\rho;TM)
\]
and
\begin{equation}\label{effective-mobility-eq}
	\Mop_{\rho}=\mathsf{P}^{\mathrm{hor}}_{\rho}\mathsf{S}_{\rho}\big|_{H^{\Wass}_{\rho}}.
\end{equation}
The inclusion $\Iop_{\rho}$ is written explicitly in $\mathsf{S}_{\rho}$ because the Stein Riesz representative originally belongs to $\Hspace_k$, whereas particles and the continuity operator see its included vector field in $L^2(\rho;TM)$.

\begin{lemma}\label{effective-mobility-lem}
	The operator $\Mop_{\rho}$ is bounded, positive and self-adjoint on $H^{\Wass}_{\rho}$. Every pair of regular covectors satisfies
	\begin{equation}\label{stein-relative-cometric-eq}
		\langle f,\Kop^{\Stein}_{\rho}h\rangle=\left\langle\grad f,\Mop_{\rho}\grad h\right\rangle_{L^2(\rho)},
	\end{equation}
	and
	\begin{equation}\label{stein-effective-tangent-eq}
		\Kop^{\Stein}_{\rho}f=\Aop_{\rho}\Mop_{\rho}\grad f.
	\end{equation}
\end{lemma}
\begin{proof}
	The operator $\mathsf{S}_{\rho}=\Iop_{\rho}\Iop_{\rho}^{*}$ is bounded, positive and self-adjoint on $L^2(\rho;TM)$. If $u,v\in H^{\Wass}_{\rho}$, then
	\[
		\langle u,\Mop_{\rho}v\rangle_{L^2(\rho)}=\langle u,\mathsf{S}_{\rho}v\rangle_{L^2(\rho)}=\langle\mathsf{S}_{\rho}u,v\rangle_{L^2(\rho)}=\langle\Mop_{\rho}u,v\rangle_{L^2(\rho)},
	\]
	and positivity follows in the same manner. Formula \eqref{stein-relative-cometric-eq} follows from the adjoint relation. The vector $\mathsf{S}_{\rho}\grad f-\Mop_{\rho}\grad f$ belongs to $(H^{\Wass}_{\rho})^{\perp}=\Ker\Aop_{\rho}$, whence the continuity operator removes it and gives \eqref{stein-effective-tangent-eq}.
\end{proof}

The operator $\Mop_{\rho}$ is the Stein response viewed through the Otto horizontal quotient. It records every feature of the kernel which can affect the law and discards precisely the vertical component which affects particles without affecting the density.

\subsection{Exact agreement of particles and laws}\label{exact-comparison-subsec}

Let $[f]$ be a regular covector and write $u=\grad f$. The two selected descending particle fields, included in $L^2(\rho;TM)$, are
\[
	v^{\Wass}_{f}=-u
\]
and
\[
	v^{\Stein}_{f}=-\mathsf{S}_{\rho}u.
\]
Their density gradient tangents are $\Kop^{\Wass}_{\rho}f=\Aop_{\rho}u$ and $\Kop^{\Stein}_{\rho}f=\Aop_{\rho}\Mop_{\rho}u$.

\begin{theorem}\label{comparison-levels-thm}
	Let $[f]$ be a regular covector. Then the following statements hold.
	\begin{enumerate}[label=(\roman*)]
		\item The Otto and Stein density tangents agree if and only if
		\[
			\Mop_{\rho}\grad f=\grad f.
		\]
		\item The selected particle fields agree in $L^2(\rho;TM)$ if and only if
		\[
			\mathsf{S}_{\rho}\grad f=\grad f.
		\]
		\item Equality of the density tangents implies that the difference $\mathsf{S}_{\rho}\grad f-\grad f$ is vertical, namely it belongs to $\Ker\Aop_{\rho}$. Thus equal laws may retain different particle motions.
		\item The tangent defect satisfies the exact identity
		\begin{equation}\label{tangent-defect-eq}
			\|\Kop^{\Stein}_{\rho}f-\Kop^{\Wass}_{\rho}f\|_{\dot H^{-1}_{\rho}}=\|(\Mop_{\rho}-\id)\grad f\|_{L^2(\rho)}.
		\end{equation}
	\end{enumerate}
\end{theorem}
\begin{proof}
	Lemma \ref{effective-mobility-lem} gives
	\[
		\Kop^{\Stein}_{\rho}f-\Kop^{\Wass}_{\rho}f=\Aop_{\rho}(\Mop_{\rho}-\id)u.
	\]
	The vector $(\Mop_{\rho}-\id)u$ is horizontal, and $\Aop_{\rho}$ is injective and isometric from the horizontal space into $\dot H^{-1}_{\rho}$. This proves the first and fourth statements. The second statement is immediate from the two particle fields. If the density tangents agree, then $\Aop_{\rho}(\mathsf{S}_{\rho}u-u)=0$, which proves the third.
\end{proof}

\begin{corollary}\label{full-fibre-equivalence-cor}
	Let $E\subseteq H^{\Wass}_{\rho}$ be closed and suppose the gradients of regular covectors which belong to $E$ are dense in $E$. The Otto and Stein law gradients agree for every such covector if and only if
	\[
		\Mop_{\rho}|_E=\id_E.
	\]
	This condition makes their cotangent forms and tangent metrics agree on the law directions generated by $E$. The included particle fields agree for every such covector if and only if
	\[
		\mathsf{S}_{\rho}|_E=\id_E.
	\]
	Taking $E=H^{\Wass}_{\rho}$ characterises agreement of the complete fibrewise law geometries; agreement of the complete particle response remains the stronger condition.
\end{corollary}
\begin{proof}
	Theorem \ref{comparison-levels-thm} gives the two equivalences on the dense set of regular forces, and continuity extends them to $E$. If $\Mop_{\rho}|_E=\id_E$, formula \eqref{stein-relative-cometric-eq} identifies the cotangent forms. The restricted tangent formula in Proposition \ref{metric-comparison-prop}, with $\Mop_{\rho}^{-1}=\id_E$, identifies the tangent metrics.
\end{proof}

Particle equivalence is therefore stronger than law equivalence. The remaining vertical field can generate circulation, relabelling or another measure preserving rearrangement of particles; its disappearance after passage to the law is a quotient phenomenon rather than an assertion that the microscopic motions coincide.

\subsection{Time changes and equality of trajectories}\label{time-change-subsec}

Equality of vector fields is stronger than equality of their integral curves as unparametrised subsets of the density manifold. A positive scalar multiple changes the clock whilst preserving the route through the space of laws.

\begin{prop}\label{time-change-prop}
	Let $[f]$ be a regular covector and let $a>0$. One has
	\[
		\Kop^{\Stein}_{\rho}f=a\Kop^{\Wass}_{\rho}f
	\]
	if and only if
	\begin{equation}\label{mobility-eigenforce-eq}
		\Mop_{\rho}\grad f=a\grad f.
	\end{equation}
	Let $\mathcal F$ be a differentiable functional on a region where both descending flows are uniquely well posed. If a positive function $a(\rho)$ satisfies \eqref{mobility-eigenforce-eq} with $f=\delta\mathcal F/\delta\rho$ throughout that region, then the Otto and Stein flows of $\mathcal F$ have the same unparametrised trajectories. If $\rho^{\Wass}_s$ is the Otto solution, the Stein solution is $\rho^{\Stein}_t=\rho^{\Wass}_{s(t)}$, where
	\[
		\dv{s}{t}=a(\rho^{\Wass}_{s(t)}).
	\]
\end{prop}
\begin{proof}
	The first equivalence follows from the injectivity of $\Aop_{\rho}$ on $H^{\Wass}_{\rho}$. Differentiating the reparametrised flow gives
	\[
		\dv{}{t}\rho^{\Wass}_{s(t)}=-\dv{s}{t}\Kop^{\Wass}_{\rho^{\Wass}_{s(t)}}\dd{\mathcal F}=-\Kop^{\Stein}_{\rho^{\Wass}_{s(t)}}\dd{\mathcal F}.
	\]
	Uniqueness identifies the reparametrised curve with the Stein solution.
\end{proof}

The scalar $a$ changes the speed of relaxation and the rate of dissipation. The sequence of laws remains the same. This distinction is useful when a kernel preconditions the force without changing its direction along the chosen functional.

\subsection{Prescribed observations and finite dimensional responses}\label{prescribed-observation-subsec}

Let $\mathcal O\colon\Prob^{\infty}_+(M)\to\mathcal N$ be a differentiable observation with values in a finite dimensional manifold or Banach space, and let $R_{\rho}$ be a prescribed linear map on $T_{\mathcal O(\rho)}\mathcal N$. The two gradient tangents induced by $[f]$ will be called equivalent through the prescription $(\mathcal O,R)$ at $\rho$ when
\begin{equation}\label{prescribed-equivalence-eq}
	D\mathcal O(\rho)[\Kop^{\Stein}_{\rho}f]=R_{\rho}D\mathcal O(\rho)[\Kop^{\Wass}_{\rho}f].
\end{equation}
Exact law equality corresponds to $\mathcal O=\id$ and $R_{\rho}=\id$, whilst a prescribed time scale corresponds to $\mathcal O=\id$ and $R_{\rho}=a(\rho)\id$. Let $\mathscr F$ be a class of differentiable functionals and let $U\subseteq\Prob^{\infty}_+(M)$. One calls the two geometries $(\mathcal O,R)$ equivalent for $\mathscr F$ on $U$ when \eqref{prescribed-equivalence-eq} holds at every $\rho\in U$ with $f=\delta\mathcal F/\delta\rho$ for every $\mathcal F\in\mathscr F$. Thus the prescribed equivalence records both the forces which are tested and the information which is observed.

\begin{prop}\label{prescribed-observation-prop}
	Let $T_1,\ldots,T_m\in C^{\infty}(M)$ and let
	\[
		\mathcal O_T(\rho)=\left(\int_MT_1\rho\,\dd{\mathrm{vol}},\ldots,\int_MT_m\rho\,\dd{\mathrm{vol}}\right).
	\]
	If $R_{\rho}$ is represented by the matrix $(R_{ij}(\rho))$, then \eqref{prescribed-equivalence-eq} holds exactly when
	\begin{equation}\label{moment-comparison-eq}
		\int_M\rho\langle\grad T_i,\Mop_{\rho}\grad f\rangle_g\,\dd{\mathrm{vol}}=\sum_{j=1}^{m}R_{ij}(\rho)\int_M\rho\langle\grad T_j,\grad f\rangle_g\,\dd{\mathrm{vol}}
	\end{equation}
	for every $i$. Consequently, equality of the observed moment evolutions means that $(\Mop_{\rho}-\id)\grad f$ is orthogonal to the gradients of the chosen observables.
\end{prop}
\begin{proof}
	Formula \eqref{transport-pairing-eq} gives
	\[
		D\mathcal O_{T,i}(\rho)[\Aop_{\rho}v]=\int_M\rho\langle\grad T_i,v\rangle_g\,\dd{\mathrm{vol}}.
	\]
	Substitute $v=\Mop_{\rho}\grad f$ and $v=\grad f$ into \eqref{prescribed-equivalence-eq}.
\end{proof}

A finite dimensional response can be prescribed by the kernel itself whenever the required response is positive and symmetric on the chosen force space.

\begin{prop}\label{finite-rank-prescription-prop}
	Fix $\rho$ and let $E\subset H^{\Wass}_{\rho}\cap C^{\infty}(M;TM)$ be finite dimensional. Let $R\colon E\to E$ be positive definite and self-adjoint in $L^2(\rho;TM)$. There exists a smooth finite rank vector-valued kernel whose reproducing kernel Hilbert space is $E$ and for which
	\[
		\mathsf{S}_{\rho}u=Ru
	\]
	for every $u\in E$. Consequently, the Stein law tangent generated by a force $u=\grad f\in E$ is $\Aop_{\rho}Ru$. The choices $R=\id$ and $R=a\id$ produce respectively exact particle agreement and a prescribed time change on $E$.
\end{prop}
\begin{proof}
	Give $E$ the Hilbert inner product
	\[
		\langle v,w\rangle_{\Hspace}=\langle R^{-1}v,w\rangle_{L^2(\rho)}.
	\]
	Let $\Iop\colon E\to L^2(\rho;TM)$ be the inclusion and let $\mathsf{P}_E$ denote the weighted $L^2(\rho)$ projection onto $E$. The adjoint identity gives $\Iop^{*}u=R\mathsf{P}_Eu$, hence $\Iop\Iop^{*}=R\mathsf{P}_E$ and its restriction to $E$ is $R$. Evaluation is continuous because $E$ is finite dimensional and consists of smooth fields, so $E$ has a smooth reproducing kernel. If $e_1,\ldots,e_m$ is an orthonormal basis in the chosen Hilbert norm, the kernel is
	\[
		k(x,y)\xi=\sum_{j=1}^{m}e_j(y)\langle e_j(x),\xi\rangle_g.
	\]
\end{proof}

The construction is fibrewise because both the $L^2(\rho)$ inner product and the desired force space may vary with $\rho$. A smoothly varying prescribed response requires a smooth finite rank bundle $E_{\rho}$ and a smooth positive family $R_{\rho}$.

\subsection{Quantitative comparability}\label{quantitative-comparison-subsec}

Equality is often more than an application requires. Bounds on the spectrum of $\Mop_{\rho}$ compare the two dissipation structures on a chosen force subspace.

\begin{prop}\label{metric-comparison-prop}
	Let $E\subset H^{\Wass}_{\rho}$ be a closed $\Mop_{\rho}$ invariant subspace and suppose that constants $0<c\leqslant C<\infty$ satisfy
	\begin{equation}\label{mobility-spectral-bounds-eq}
		c\|u\|_{L^2(\rho)}^2\leqslant\langle u,\Mop_{\rho}u\rangle_{L^2(\rho)}\leqslant C\|u\|_{L^2(\rho)}^2
	\end{equation}
	for every $u\in E$. If $\grad f\in E$, then
	\begin{equation}\label{cotangent-comparison-eq}
		c\|f\|_{\dot H^1_{\rho}}^2\leqslant\|f\|_{T^{*}_{\rho,\Stein}}^2\leqslant C\|f\|_{\dot H^1_{\rho}}^2.
	\end{equation}
	The tangent metrics induced by restricting both Hilbert pairs to $E$ satisfy, for $\sigma=\Aop_{\rho}u$,
	\begin{equation}\label{tangent-metric-comparison-eq}
		C^{-1}\|\sigma\|_{\dot H^{-1}_{\rho}}^2\leqslant\|\sigma\|_{T_{\rho,\Stein;E}}^2\leqslant c^{-1}\|\sigma\|_{\dot H^{-1}_{\rho}}^2.
	\end{equation}
\end{prop}
\begin{proof}
	Formula \eqref{stein-relative-cometric-eq} identifies the Stein cotangent norm with $\langle u,\Mop_{\rho}u\rangle$ for $u=\grad f$, which proves \eqref{cotangent-comparison-eq}. The lower bound makes $\Mop_{\rho}|_E$ invertible. The identification $\Aop_{\rho}\colon E\to\Aop_{\rho}(E)$ makes the restricted Stein tangent metric $\langle\Mop_{\rho}^{-1}u,u\rangle$. The spectrum of the inverse lies in $[C^{-1},c^{-1}]$, which proves \eqref{tangent-metric-comparison-eq}.
\end{proof}

Norm equivalence gives comparable steepest descent speeds and comparable energy dissipation for every force in $E$. It does not require $\Mop_{\rho}$ to preserve each force direction; consequently, the trajectories may still differ.

A chosen ambient norm turns the infinitesimal defect into a comparison of the resulting curves.

\begin{prop}\label{flow-comparison-prop}
	Let $X$ be a Banach chart containing two law curves with the same initial value, and let their vector fields $V^{\Wass}$ and $V^{\Stein}$ be $L$ Lipschitz on the region traversed. Suppose
	\[
		\|V^{\Stein}(\rho)-V^{\Wass}(\rho)\|_X\leqslant\varepsilon
	\]
	throughout that region. Then
	\[
		\|\rho_t^{\Stein}-\rho_t^{\Wass}\|_X\leqslant
		\begin{cases}
			\varepsilon\bigl(e^{Lt}-1\bigr)/L, & L>0,\\
			\varepsilon t, & L=0.
		\end{cases}
	\]
\end{prop}
\begin{proof}
	The integral equations and the triangle inequality give
	\[
		\|\rho_t^{\Stein}-\rho_t^{\Wass}\|_X\leqslant\int_0^tL\|\rho_s^{\Stein}-\rho_s^{\Wass}\|_X\,\dd{s}+\varepsilon t.
	\]
	Gronwall's inequality gives the result.
\end{proof}

The identity \eqref{tangent-defect-eq} supplies the local value of $\varepsilon$ when the chosen chart norm is controlled by the Otto tangent norm and the force remains in a region on which $\Mop_{\rho}$ varies continuously.

\subsection{Interpretation of the equivalences}\label{comparison-interpretation-subsec}

A transport geometry is a constitutive relation between a covector force and a flux of probability. Otto geometry uses the identity response on horizontal $L^2(\rho)$ forces. Stein geometry uses the positive response $\Mop_{\rho}$ after the vertical part of the kernel field has been removed. Agreement of the geometries on a force subspace means agreement of these constitutive relations on that subspace.

Full agreement of the law geometry at $\rho$ is the identity $\Mop_{\rho}=\id$ on the Otto horizontal space. The Onsager operators, cotangent forms and tangent metrics then coincide on every law direction. The kernel response $\mathsf{S}_{\rho}$ may retain a vertical component, so the corresponding particle fields can remain different even when the complete fibrewise law geometry agrees.

Law equivalence identifies particle fields modulo measure preserving rearrangements. Two particle systems may therefore transport every initial sample along different trajectories whilst producing the same evolving population law. Finite empirical configurations can distinguish these lifts even when the limiting law cannot.

Observation equivalence identifies only the information retained by a chosen map $\mathcal O$. Equality of sufficient statistics, moments or coarse variables says that the geometric defect lies in the kernel of $D\mathcal O$. The ambient laws may differ in all directions which the observation forgets.

Time change equivalence identifies the route towards equilibrium and permits a different clock. Dissipation is rescaled by the same positive factor along the chosen force. Norm comparability is weaker: it controls cost and rate up to constants whilst allowing the force direction to rotate.

Equality of a single gradient flow does not identify the full geometry. The flow tests only the covector $\dd{\mathcal F}(\rho)$ at each visited density. Large deviations actions, responses to other functionals, perturbations transverse to the curve and finite particle fluctuations may continue to distinguish the two structures.

\section{Differentiable functionals and the hierarchy of gradient forces}\label{functional-sec}

\subsection{Regular first variations}

Let $\mathcal F \colon \Prob^{\infty}_+(M) \to \R$ be differentiable. A regular first variation at $\rho$ is a smooth function $\delta\mathcal F/\delta\rho(\rho)$ for which
\begin{equation}\label{regular-first-variation-eq}
	D\mathcal F(\rho)[\sigma] = \int_M \frac{\delta\mathcal F}{\delta\rho}(\rho)\sigma \, \dd{\mathrm{vol}}
\end{equation}
for every $\sigma \in C^{\infty}_0(M)$. It is determined modulo constants. If $\rho_t$ satisfies the continuity equation with velocity $v_t$, formulae \eqref{transport-pairing-eq} and \eqref{regular-first-variation-eq} give
\begin{equation}\label{transport-chain-rule-eq}
	\dv{}{t}\mathcal F(\rho_t) = \int_M \rho_t\left\langle \grad\frac{\delta\mathcal F}{\delta\rho}(\rho_t),v_t\right\rangle_g \, \dd{\mathrm{vol}}.
\end{equation}
The spatial gradient removes the constant ambiguity and supplies the Otto representative of the covector.

Choose either the Wasserstein or Stein velocity map and write $\bullet\in\{\Wass,\Stein\}$. The descending velocity and density tangent are
\begin{equation}\label{two-gradient-fields-eq}
	v^{\bullet}_{\mathcal F}(\rho) = -\Bop^{\bullet}_{\rho}\frac{\delta\mathcal F}{\delta\rho}(\rho)
\end{equation}
and
\[
	\operatorname{grad}^{\bullet}\mathcal F(\rho) = \Kop^{\bullet}_{\rho}\frac{\delta\mathcal F}{\delta\rho}(\rho).
\]
The descending flow has $\partial_t\rho=-\operatorname{grad}^{\bullet}\mathcal F(\rho)$. Subject to the regularity needed for \eqref{transport-chain-rule-eq}, its dissipation identity is
\begin{equation}\label{two-dissipation-eq}
	\dv{}{t}\mathcal F(\rho_t) = -\left\|\frac{\delta\mathcal F}{\delta\rho}(\rho_t)\right\|_{T^{*}_{\rho_t,\bullet}}^2.
\end{equation}

\subsection{Relative entropy and the target Stein operator}

Fix $\pi \in \Prob^{\infty}_+(M)$ and define
\[
	\Ent_{\pi}(\rho) = \int_M \rho\log\frac{\rho}{\pi} \, \dd{\mathrm{vol}}.
\]

\begin{lemma}\label{entropy-differential-lem}
	The regular differential of $\Ent_{\pi}$ is represented by
	\[
		\frac{\delta\Ent_{\pi}}{\delta\rho}(\rho) = \log\frac{\rho}{\pi}+1.
	\]
	Its cotangent class is $[\log(\rho/\pi)]$.
\end{lemma}
\begin{proof}
	Let $\rho_t=\rho+t\sigma$ with $\int_M\sigma\,\dd{\mathrm{vol}}=0$. Direct differentiation gives
	\[
		\dv{}{t}\bigg|_{t=0}\Ent_{\pi}(\rho_t)=\int_M\left(\log\frac{\rho}{\pi}+1\right)\sigma\,\dd{\mathrm{vol}}.
	\]
	The integral of the constant term against $\sigma$ vanishes, proving the cotangent statement.
\end{proof}

A smooth vector field $v$ has target Stein operator
\begin{equation}\label{target-stein-operator-eq}
	\Tspace_{\pi}v = \Div v + \langle\grad\log\pi,v\rangle_g.
\end{equation}
Its expectation under $\rho$ is related exactly to the entropy differential.

\begin{theorem}\label{stein-entropy-derivative-thm}
	Every $v \in \Vect(M)$ satisfies
	\begin{equation}\label{stein-entropy-identity-eq}
		D\Ent_{\pi}(\rho)[\Aop_{\rho}v] = -\int_M \rho\,\Tspace_{\pi}v \, \dd{\mathrm{vol}}.
	\end{equation}
\end{theorem}
\begin{proof}
	Lemma \ref{entropy-differential-lem} and \eqref{transport-pairing-eq} give
	\[
		D\Ent_{\pi}(\rho)[\Aop_{\rho}v] = \int_M \rho\left\langle\grad\log\frac{\rho}{\pi},v\right\rangle_g\,\dd{\mathrm{vol}}.
	\]
	The identity $\rho\grad\log\rho=\grad\rho$ and integration by parts yield
	\[
		\int_M \rho\langle\grad\log\rho,v\rangle_g\,\dd{\mathrm{vol}} = \int_M\langle\grad\rho,v\rangle_g\,\dd{\mathrm{vol}} = -\int_M\rho\Div v\,\dd{\mathrm{vol}}.
	\]
	Subtracting the term involving $\grad\log\pi$ proves \eqref{stein-entropy-identity-eq}.
\end{proof}

The theorem gives a precise variational form of Stein's method: the target Stein expectation is the negative directional derivative of relative entropy along the transport generated by $v$. The maximising field in a reproducing kernel Hilbert unit ball is therefore the steepest entropy descent field for the Stein metric.

\subsection{The Wasserstein and Stein entropy flows}

The Wasserstein velocity of relative entropy is
\[
	v^{\Wass}_{\Ent}(\rho) = -\grad\log\frac{\rho}{\pi}.
\]
Its continuity equation is
\begin{equation}\label{wasserstein-entropy-flow-eq}
	\partial_t\rho = \Div\left(\rho\grad\log\frac{\rho}{\pi}\right) = \Delta\rho-\Div(\rho\grad\log\pi).
\end{equation}
This is the Fokker--Planck equation with invariant density $\pi$. The dissipation is the relative Fisher information
\begin{equation}\label{wasserstein-entropy-dissipation-eq}
	\dv{}{t}\Ent_{\pi}(\rho_t) = -\int_M\left|\grad\log\frac{\rho_t}{\pi}\right|_g^2\rho_t\,\dd{\mathrm{vol}}.
\end{equation}
The variational time discretisation of Jordan, Kinderlehrer and Otto and Otto's geometric formulation make \eqref{wasserstein-entropy-flow-eq} rigorous beyond the smooth setting \cite{jordan1998variational,otto2001geometry,ambrosio2008gradient}.

The Stein velocity is
\begin{equation}\label{stein-entropy-velocity-eq}
	v^{\Stein}_{\Ent}(\rho) = -\Iop_{\rho}^{*}\grad\log\frac{\rho}{\pi}.
\end{equation}
Its density evolution is
\begin{equation}\label{stein-entropy-flow-eq}
	\partial_t\rho = \Div\left(\rho\Iop_{\rho}^{*}\grad\log\frac{\rho}{\pi}\right).
\end{equation}
The dissipation identity becomes
\begin{equation}\label{stein-entropy-dissipation-eq}
	\dv{}{t}\Ent_{\pi}(\rho_t) = -\left\|\Iop_{\rho_t}^{*}\grad\log\frac{\rho_t}{\pi}\right\|_{\Hspace_k}^2.
\end{equation}
The right hand side is the squared kernel Stein discrepancy, also called the Stein--Fisher information in the gradient flow literature \cite{liu2017stein,duncan2023geometry,nusken2023large}.

Let $M=\R^d$ for the following formal calculation and assume all boundary terms vanish. A scalar kernel permits integration by parts in the $x$ variable, which gives
\begin{align}
	v^{\Stein}_{\Ent}(\rho)(y)
	&= -\int_{\R^d}k(x,y)\grad_x\log\frac{\rho(x)}{\pi(x)}\rho(x)\,\dd{x} \notag\\
	&= \int_{\R^d}\left[k(x,y)\grad\log\pi(x)+\grad_x k(x,y)\right]\rho(x)\,\dd{x}. \label{stein-variational-field-eq}
\end{align}
Formula \eqref{stein-variational-field-eq} is the mean field velocity used by Stein variational gradient descent \cite{liu2016stein}. The score of $\rho$ has disappeared through integration by parts, whilst the cotangent from which the field was obtained remains $[\log(\rho/\pi)]$.

\begin{corollary}\label{entropy-equivalence-cor}
	Let $u_{\rho}=\grad\log(\rho/\pi)$. The Wasserstein and Stein entropy law tangents agree at $\rho$ if and only if
	\[
		\Mop_{\rho}u_{\rho}=u_{\rho}.
	\]
	They determine the same unparametrised entropy trajectory on a region if $\Mop_{\rho}u_{\rho}=a(\rho)u_{\rho}$ for a positive function $a$. Their instantaneous entropy dissipations agree if and only if
	\begin{equation}\label{entropy-dissipation-comparison-eq}
		\langle u_{\rho},\Mop_{\rho}u_{\rho}\rangle_{L^2(\rho)}=\|u_{\rho}\|_{L^2(\rho)}^2.
	\end{equation}
	Equality in \eqref{entropy-dissipation-comparison-eq} need not imply equality of the law tangents.
\end{corollary}
\begin{proof}
	The first two statements are Theorem \ref{comparison-levels-thm} and Proposition \ref{time-change-prop}. The two dissipation formulae are $-\|u_{\rho}\|_{L^2(\rho)}^2$ and $-\langle u_{\rho},\Mop_{\rho}u_{\rho}\rangle_{L^2(\rho)}$. Equality of one quadratic value does not force $\Mop_{\rho}u_{\rho}=u_{\rho}$ unless a further sign or spectral hypothesis is imposed on $\Mop_{\rho}-\id$.
\end{proof}

\section{Prescribed finite dimensional evolutions and covector descent}\label{statistical-submanifolds-sec}

\subsection{Immersed models}

Let $\Theta$ be a smooth finite dimensional manifold and let $q \colon \Theta \to \Prob^{\infty}_+(M)$ be a smooth immersion. Local coordinates $\theta=(\theta^1,\ldots,\theta^r)$ give the tangent map
\[
	T_{\theta}q(\dot{\theta}) = \sum_{i=1}^{r}\dot{\theta}^i\partial_i q_{\theta}.
\]
A transport lift of the coordinate tangent $\partial_iq_{\theta}$ is any vector field $v_i$ satisfying
\begin{equation}\label{model-lift-eq}
	-\Div(q_{\theta}v_i)=\partial_iq_{\theta}.
\end{equation}
The lift is defined modulo $q_{\theta}$ divergence free fields. A chosen transport geometry selects a minimal lift when the tangent lies in its range. Pulling the metric back along $q$ gives
\[
	G^{\bullet}_{ij}(\theta)=G^{\bullet}_{q_{\theta}}(\partial_iq_{\theta},\partial_jq_{\theta}).
\]
The Fisher metric supplies another pullback geometry,
\begin{equation}\label{fisher-model-metric-eq}
	G^{\Fish}_{ij}(\theta)=\int_M\partial_i\log q_{\theta}\,\partial_j\log q_{\theta}\,q_{\theta}\,\dd{\mathrm{vol}}.
\end{equation}
No equality between these three metrics is automatic. Their relation depends on the transport lifts and kernel.

Let $T_1,\ldots,T_m$ be observables and define the expectation map
\[
	\mathfrak m(q)=\left(\int_MT_1q\,\dd{\mathrm{vol}},\ldots,\int_MT_mq\,\dd{\mathrm{vol}}\right).
\]
A transport tangent $\sigma=-\Div(qv)$ has expectation derivative
\begin{equation}\label{moment-transport-derivative-eq}
	D\mathfrak m_j(q)[\sigma]=\int_MT_j\sigma\,\dd{\mathrm{vol}}=\int_Mq\langle\grad T_j,v\rangle_g\,\dd{\mathrm{vol}}.
\end{equation}
An exact moment update therefore requires only these pairings. An exact density update inside a model requires the stronger identity \eqref{model-lift-eq}.

\subsection{Regular exponential families}

Let $T=(T_1,\ldots,T_r) \colon M \to \R^r$ be smooth and let $h$ be a positive smooth reference density. An open parameter set $\Omega\subset\R^r$ carries the family
\begin{equation}\label{exponential-family-eq}
	q_{\eta}(x)=h(x)\exp\left(\eta^{\mathsf T}T(x)-A(\eta)\right),
\end{equation}
where
\[
	A(\eta)=\log\int_Mh(x)e^{\eta^{\mathsf T}T(x)}\,\dd{\mathrm{vol}}(x).
\]
Assume the family is minimal, so that the covariance matrix below is positive definite. Its expectation coordinate and covariance are
\begin{equation}\label{exponential-moments-eq}
	m(\eta)=\E_{q_{\eta}}[T]=\grad_{\eta}A(\eta)
\end{equation}
and
\begin{equation}\label{exponential-covariance-eq}
	C(\eta)=\Cov_{q_{\eta}}(T)=\grad_{\eta}^2A(\eta).
\end{equation}
The density tangent generated by $\dot{\eta}$ is
\begin{equation}\label{exponential-tangent-eq}
	\partial_tq_{\eta}=q_{\eta}(T-m)^{\mathsf T}\dot{\eta}.
\end{equation}
The Fisher metric in natural coordinates is $C(\eta)$ and in expectation coordinates is $C(\eta)^{-1}$. Let $A^{*}$ be the Legendre transform of $A$. Then $\grad_mA^{*}(m)=\eta$ and $\grad_m^2A^{*}(m)=C(\eta)^{-1}$ \cite{ay2017information,rockafellar1970convex}.

A covector in expectation coordinates has the form
\[
	\beta=\sum_{j=1}^{r}\beta_j\,\dd{m_j}.
\]
Its Fisher sharp is $C(\eta)\beta$, and its descending Fisher mirror equation is
\begin{equation}\label{fisher-mirror-moment-eq}
	\dot{m}=-C(\eta)\beta.
\end{equation}
The identity $\dot m=C(\eta)\dot\eta$ makes this equivalent to $\dot\eta=-\beta$.

\subsection{Stein realisation of covector descent}

A positive density $q$ has Stein operator
\begin{equation}\label{density-stein-operator-eq}
	\Tspace_qv=\Div v+\langle\grad\log q,v\rangle_g.
\end{equation}
The continuity operator factorises as
\begin{equation}\label{continuity-stein-factorisation-eq}
	-\Div(qv)=-q\Tspace_qv.
\end{equation}

\begin{theorem}\label{exponential-stein-covector-thm}
	Let $q_{\eta}$ be the regular minimal exponential family in \eqref{exponential-family-eq}. Suppose that, for every $\eta$ under consideration and every $j$, there exists an admissible vector field $\psi_j^{\eta}$ satisfying
	\begin{equation}\label{statistic-stein-equation-eq}
		\Tspace_{q_{\eta}}\psi_j^{\eta}=-(T_j-m_j(\eta)).
	\end{equation}
	A covector $\beta=\sum_j\beta_j\dd{m_j}$ directly determines
	\begin{equation}\label{covector-particle-field-eq}
		v_{\beta}^{\eta}=-\sum_{j=1}^{r}\beta_j\psi_j^{\eta}.
	\end{equation}
	Then the continuity equation $\partial_tq=-\Div(qv_{\beta}^{\eta})$ agrees exactly with natural parameter descent $\dot\eta=-\beta$. Consequently, its expectation coordinate satisfies Fisher mirror descent \eqref{fisher-mirror-moment-eq}.
\end{theorem}
\begin{proof}
	Linearity and \eqref{statistic-stein-equation-eq} give
	\[
		\Tspace_{q_{\eta}}v_{\beta}^{\eta}=\sum_{j=1}^{r}\beta_j(T_j-m_j)=\beta^{\mathsf T}(T-m).
	\]
	The factorisation \eqref{continuity-stein-factorisation-eq} therefore gives
	\begin{equation}\label{covector-density-update-eq}
		\partial_tq_{\eta}=-q_{\eta}\beta^{\mathsf T}(T-m).
	\end{equation}
	Comparison with \eqref{exponential-tangent-eq} and minimality of the family yield $\dot\eta=-\beta$. Differentiating $m=\grad A(\eta)$ now gives $\dot m=C(\eta)\dot\eta=-C(\eta)\beta$.
\end{proof}

\begin{remark}\label{smooth-stein-solutions-rem}
	The existence of smooth vector solutions of \eqref{statistic-stein-equation-eq} is automatic on the connected closed manifold when admissible means smooth. Every centred smooth statistic $g=T_j-m_j(\eta)$ has a unique potential $u_j^{\eta}\in C^{\infty}(M)/\R$ satisfying
	\[
		\Lop_{q_{\eta}}u_j^{\eta}=\Delta u_j^{\eta}+\langle\grad\log q_{\eta},\grad u_j^{\eta}\rangle_g=-g,
	\]
	and $\psi_j^{\eta}=\grad u_j^{\eta}$ solves the Stein equation. The additional kernel hypothesis is $\psi_j^{\eta}\in\Hspace_k$, or an approximation result in the kernel admissible range. Different smooth solutions differ by a $q_{\eta}$ divergence free field and hence induce the same density tangent whilst moving particles differently.
\end{remark}

The theorem identifies the complete density tangent, and its moment evolution follows as a consequence. If the time dependent vector field in \eqref{covector-particle-field-eq} generates a unique flow map $\Phi_{s,t}$, then
\[
	(\Phi_{0,t})_{\#}q_{\eta_0}=q_{\eta_t}.
\]
Every particle moved by the exact flow therefore has the required law. A finite Euler map $x\mapsto x+\varepsilon v_{\beta}^{\eta}(x)$ agrees with the density equation to first order in $\varepsilon$ and generally leaves the exponential family at order $\varepsilon^2$.

\begin{corollary}\label{general-mirror-cor}
	Let $\Psi$ be a strictly convex smooth mirror potential in expectation coordinates. Suppose $\grad^2\Psi(m)$ is positive definite and the hypotheses of Theorem \ref{exponential-stein-covector-thm} hold. Set
	\[
		c=C(\eta)^{-1}\bigl(\grad^2\Psi(m)\bigr)^{-1}\beta
	\]
	and use the field $v=-\sum_jc_j\psi_j^{\eta}$. Then
	\[
		\dot m=-\bigl(\grad^2\Psi(m)\bigr)^{-1}\beta.
	\]
\end{corollary}
\begin{proof}
	The proof of Theorem \ref{exponential-stein-covector-thm} gives $\dot\eta=-c$. Hence $\dot m=-C(\eta)c$, which is the asserted equation.
\end{proof}

\subsection{Minimal hypotheses and the hidden cometric}

Theorem \ref{exponential-stein-covector-thm} requires four substantive facts. The density must admit integration by parts for the operator $\Tspace_q$; the centred sufficient statistics must lie in the range of this operator on the selected velocity class; the resulting fields must generate a well posed continuity equation; and the model tangent representation must be injective. Reproducing kernel implementation adds the range condition $\psi_j^{\eta}\in\Hspace_k$, or an approximation statement in the closure of the kernel admissible fields. These assumptions are close to minimal because failure of the range condition prevents the desired density tangent from being generated by that velocity class.

The direct map $\beta\mapsto v_{\beta}^{\eta}$ avoids computing a density tangent and then inverting a metric tensor. Its geometry nevertheless contains a cotangent to velocity map. Choosing the solutions $\psi_j^{\eta}$ chooses a right inverse of the Stein operator on the span of centred sufficient statistics; equivalently, it chooses a cometric on this finite dimensional cotangent space. The sharp map has consequently been factorised into precomputed Stein solutions; the geometric choice remains present in that factorisation.

\begin{prop}\label{finite-covector-cometric-prop}
	Let $G_{\eta}\colon\R^r\to\Vspace_{\eta}$ be given by $G_{\eta}e_j=\psi_j^{\eta}$, where $\Vspace_{\eta}$ is a Hilbert velocity space. Then
	\[
		\langle\beta,\gamma\rangle_{*,\eta}=\langle G_{\eta}\beta,G_{\eta}\gamma\rangle_{\Vspace_{\eta}}
	\]
	is positive semidefinite, and $G_{\eta}$ is an isometric isomorphism from $\R^r/\Ker G_{\eta}$ onto its finite dimensional closed range. If the family is minimal and the fields satisfy \eqref{statistic-stein-equation-eq}, then $G_{\eta}$ is injective. The minimum-$\Vspace_{\eta}$ right inverse gives a canonical cometric amongst the chosen velocity class; adding vertical fields changes the particles and the cometric without changing the model density evolution.
\end{prop}
\begin{proof}
	The first assertions follow from the quotient norm. If $G_{\eta}\beta=0$, apply $\Tspace_{q_{\eta}}$ and use \eqref{statistic-stein-equation-eq} to obtain $\beta^{\mathsf T}(T-m)=0$ in $L^2(q_{\eta})$. Minimality makes the covariance of $T$ positive definite, whence $\beta=0$. Minimum-norm solutions are the orthogonal representatives modulo the vertical kernel.
\end{proof}

\section{The Gaussian model and affine particle fields}\label{gaussian-sec}

Let $M=\R^d$ and consider the Gaussian family
\[
	q_{\mu,\Sigma}=\mathcal N(\mu,\Sigma),
\]
where $\mu\in\R^d$ and $\Sigma\in\Sym_{++}(d)$. A tangent is a pair $(u,U)\in\R^d\times\Sym(d)$ and a covector is a pair $(a,A)$ with pairing
\begin{equation}\label{gaussian-pairing-eq}
	\langle(a,A),(u,U)\rangle=a^{\mathsf T}u+\tr(AU).
\end{equation}
The Fisher metric is
\begin{equation}\label{gaussian-fisher-metric-eq}
	G^{\Fish}_{(\mu,\Sigma)}((u,U),(v,V))=u^{\mathsf T}\Sigma^{-1}v+\frac12\tr(\Sigma^{-1}U\Sigma^{-1}V).
\end{equation}

\begin{prop}\label{gaussian-fisher-sharp-prop}
	The Fisher sharp of the covector $(a,A)$ is
	\[
		(a,A)^{\sharp_{\Fish}}=(\Sigma a,2\Sigma A\Sigma).
	\]
\end{prop}
\begin{proof}
	Substitution in \eqref{gaussian-fisher-metric-eq} gives
	\[
		G^{\Fish}_{(\mu,\Sigma)}((\Sigma a,2\Sigma A\Sigma),(u,U))=a^{\mathsf T}u+\tr(AU),
	\]
	which is \eqref{gaussian-pairing-eq}. Nondegeneracy of the metric proves uniqueness.
\end{proof}

\subsection{Affine lifts and their vertical freedom}

Let $y=x-\mu$ and consider an affine velocity
\begin{equation}\label{affine-velocity-eq}
	v(x)=u+By.
\end{equation}

\begin{lemma}\label{affine-moment-lem}
	If $X_t$ follows the deterministic equation $\dot X_t=u_t+B_t(X_t-\mu_t)$ and $\Law(X_t)=\mathcal N(\mu_t,\Sigma_t)$, then
	\begin{equation}\label{affine-mean-eq}
		\dot\mu_t=u_t
	\end{equation}
	and
	\begin{equation}\label{affine-covariance-eq}
		\dot\Sigma_t=B_t\Sigma_t+\Sigma_tB_t^{\mathsf T}.
	\end{equation}
\end{lemma}
\begin{proof}
	Taking expectations in the particle equation gives \eqref{affine-mean-eq}. The centred variable $Y_t=X_t-\mu_t$ satisfies $\dot Y_t=B_tY_t$; differentiation of $\E[Y_tY_t^{\mathsf T}]$ gives
	\[
		\dot\Sigma_t=\E[B_tY_tY_t^{\mathsf T}+Y_tY_t^{\mathsf T}B_t^{\mathsf T}]=B_t\Sigma_t+\Sigma_tB_t^{\mathsf T}.
	\]
\end{proof}

Every Gaussian tangent $(u,U)$ therefore has the affine lift
\begin{equation}\label{gaussian-affine-right-inverse-eq}
	\Rop_{\mu,\Sigma}(u,U)(x)=u+\frac12U\Sigma^{-1}(x-\mu).
\end{equation}
Indeed, the matrix $B=\frac12U\Sigma^{-1}$ satisfies $B\Sigma+\Sigma B^{\mathsf T}=U$. This lift is one right inverse of the moment map; it need not be horizontal for every transport metric.

The vertical affine matrices satisfy
\begin{equation}\label{gaussian-affine-gauge-eq}
	B_0\Sigma+\Sigma B_0^{\mathsf T}=0.
\end{equation}
They generate linear motions preserving the Gaussian covariance and hence the entire Gaussian law when the mean is fixed. Every affine lift of $(u,U)$ is obtained by adding such a $B_0(x-\mu)$ to \eqref{gaussian-affine-right-inverse-eq}.

\begin{prop}\label{gaussian-equivalent-lifts-prop}
	Let $v_i(x)=u_i+B_i(x-\mu)$ for $i=1,2$. The two affine fields induce the same tangent of the Gaussian law $\mathcal N(\mu,\Sigma)$ if and only if
	\[
		u_1=u_2
	\]
	and
	\begin{equation}\label{gaussian-vertical-difference-eq}
		(B_1-B_2)\Sigma+\Sigma(B_1-B_2)^{\mathsf T}=0.
	\end{equation}
	They induce the same particle field if and only if $u_1=u_2$ and $B_1=B_2$.
\end{prop}
\begin{proof}
	Lemma \ref{affine-moment-lem} shows that the induced mean and covariance tangents are $u_i$ and $B_i\Sigma+\Sigma B_i^{\mathsf T}$. Equality gives the stated conditions. Equality of affine fields on a set with nonempty interior gives equality of their constant and linear parts.
\end{proof}

Equation \eqref{gaussian-vertical-difference-eq} is the Gaussian form of law equivalence modulo vertical particle motion. The residual matrices generate covariance-preserving rotations in the $\Sigma$ geometry; a finite particle cloud detects these motions even though the Gaussian density does not.

\subsection{Stein solutions and Fisher covector descent}

The Gaussian score is
\[
	\grad\log q_{\mu,\Sigma}(x)=-\Sigma^{-1}(x-\mu).
\]
The constant field $\psi_a(x)=\Sigma a$, for $a\in\R^d$, satisfies
\[
	\Tspace_q\psi_a=-a^{\mathsf T}(x-\mu).
\]
The linear field $\psi_A(x)=\Sigma A(x-\mu)$, for $A\in\Sym(d)$, satisfies
\begin{equation}\label{gaussian-quadratic-stein-eq}
	\Tspace_q\psi_A=-\left((x-\mu)^{\mathsf T}A(x-\mu)-\tr(A\Sigma)\right).
\end{equation}
The verification follows from $\Div(\Sigma Ay)=\tr(\Sigma A)$ and $\langle-\Sigma^{-1}y,\Sigma Ay\rangle=-y^{\mathsf T}Ay$.

A covector $(a,A)$ in mean and covariance coordinates therefore produces the affine descent field
\begin{equation}\label{gaussian-covector-field-eq}
	v_{a,A}(x)=-\Sigma a-\Sigma A(x-\mu).
\end{equation}
Lemma \ref{affine-moment-lem} gives
\begin{equation}\label{gaussian-fisher-flow-eq}
	\dot\mu=-\Sigma a
\end{equation}
and
\[
	\dot\Sigma=-2\Sigma A\Sigma.
\]
These are exactly the negative Fisher sharp equations from Proposition \ref{gaussian-fisher-sharp-prop}. Thus Fisher covector descent is realised by affine Stein fields without passing computationally through the density tangent.

\subsection{Wasserstein and kernel selected affine lifts}

An affine velocity is a Wasserstein horizontal gradient exactly when its linear matrix $B$ is symmetric. Given a covariance tangent $U$, the unique symmetric matrix $B_{\Wass}$ satisfying
\begin{equation}\label{gaussian-sylvester-eq}
	B_{\Wass}\Sigma+\Sigma B_{\Wass}=U
\end{equation}
is the Wasserstein minimal lift. Existence and uniqueness follow because the Sylvester operator $B\mapsto B\Sigma+\Sigma B$ is positive definite on $\Sym(d)$. The right inverse $\frac12U\Sigma^{-1}$ in \eqref{gaussian-affine-right-inverse-eq} equals $B_{\Wass}$ only under an additional commutation relation.

A reproducing kernel selected lift must belong to $\Hspace_k$ and satisfy the Riesz equation defining $\Bop^{\Stein}_{q}$. A Gaussian radial basis kernel on all of $\R^d$ does not generally contain arbitrary constants and unbounded linear fields. Exact affine covector updates therefore require a kernel whose Hilbert space contains the fields in \eqref{gaussian-covector-field-eq}, such as an augmented kernel with constant and polynomial components, or a finite dimensional affine velocity space treated directly. Approximation in a radial basis Hilbert space may still recover the tangent in a weaker sense on compact sets or under weighted norms, although this is an approximation theorem rather than an exact range identity.

\begin{remark}\label{affine-does-not-characterise-rem}
	Affine particle realisability does not characterise Gaussian measures. Every location and scatter family generated by affine pushforwards of a fixed non-Gaussian reference law has affine tangent fields. Gaussianity follows from stronger hypotheses, for example that the score $\grad\log q$ is affine on all of $\R^d$, or that a normalisable full exponential family has sufficient statistics $x$ and $xx^{\mathsf T}$ with the usual Euclidean support. The distinction prevents a property of the transport orbit from being mistaken for a characterisation of the reference density.
\end{remark}

\section{Particles, empirical laws and interacting transport}\label{particles-sec}

\subsection{The empirical continuity equation}

Let $x_1(t),\ldots,x_N(t)$ be differentiable curves in $M$ and define the empirical probability measure
\begin{equation}\label{empirical-measure-eq}
	\mu_t^N=\frac1N\sum_{i=1}^{N}\delta_{x_i(t)}.
\end{equation}
Atomic measures lie outside $\Prob^{\infty}_+(M)$, although they belong to $\Prob_2(M)$ and act continuously on smooth test functions. Their tangent is therefore interpreted distributionally.

\begin{prop}\label{empirical-continuity-prop}
	Suppose $\dot x_i(t)=v_t(x_i(t))$ for a smooth time dependent vector field $v_t$. Then
	\begin{equation}\label{empirical-continuity-eq}
		\partial_t\mu_t^N=-\Div(\mu_t^Nv_t)
	\end{equation}
	in the sense of distributions. Equivalently,
	\[
		\partial_t\mu_t^N=-\frac1N\sum_{i=1}^{N}v_t(x_i(t))\cdot\grad\delta_{x_i(t)}.
	\]
\end{prop}
\begin{proof}
	Every $f\in C^{\infty}(M)$ satisfies
	\[
		\dv{}{t}\int_Mf\,\dd{\mu_t^N}=\frac1N\sum_{i=1}^{N}\langle\grad f(x_i(t)),v_t(x_i(t))\rangle_g.
	\]
	The right hand side is $\int_M\langle\grad f,v_t\rangle_g\,\dd{\mu_t^N}$, which is the distributional pairing of $-\Div(\mu_t^Nv_t)$ with $f$.
\end{proof}

The proposition shows that the same continuity operator acts on smooth and atomic measures. Its domain changes: a smooth tangent is a function of zero integral, whilst an empirical tangent is a finite sum of derivatives of Dirac distributions. The Wasserstein completion accommodates both through absolutely continuous curves in $\Prob_2(M)$; the smooth Fr\'echet manifold supplies the differential calculations used to construct the velocity.

\subsection{The finite particle Stein equation}

Take $M=\R^d$ and evaluate the mean field velocity \eqref{stein-variational-field-eq} at the empirical measure. The resulting system is
\begin{equation}\label{finite-stein-particle-eq}
	\dot x_i=\frac1N\sum_{j=1}^{N}\left[k(x_j,x_i)\grad\log\pi(x_j)+\grad_{x_j}k(x_j,x_i)\right].
\end{equation}
Proposition \ref{empirical-continuity-prop} makes \eqref{finite-stein-particle-eq} a characteristic system for the empirical continuity equation. The derivative-of-kernel term in \eqref{finite-stein-particle-eq} is the integration-by-parts representation of the entropy contribution to the reproducing kernel Riesz gradient after the unknown score of the current law has been removed. The target-score and kernel-derivative terms together form one empirical evaluation of the variational velocity.

The exact covector fields of Theorem \ref{exponential-stein-covector-thm} give another particle system,
\begin{equation}\label{finite-covector-particles-eq}
	\dot x_i=-\sum_{j=1}^{r}\beta_j(t)\psi_j^{\eta_t}(x_i).
\end{equation}
If $x_i(0)$ are independent with law $q_{\eta_0}$, then every $x_i(t)$ has law $q_{\eta_t}$ whenever the deterministic flow is well posed. The empirical measure remains random at finite $N$, whilst its expected law and its almost sure large particle limit follow the exact statistical model. The distinction between exact law transport and exact empirical moments matters: no finite sample is forced to have the population moments of $q_{\eta_t}$.

Law-equivalent fields need not define equivalent finite particle algorithms. If $v^{\Stein}-v^{\Wass}$ is $\rho$ divergence free, the corresponding population laws agree when initial particles have law $\rho$, although each sampled configuration is transported along a different set of trajectories. Finite empirical measures therefore retain information which disappears under the smooth law quotient; convergence to the same limiting law does not identify the transient particle system.

\subsection{A quantitative mean field passage}

Let $b \colon \Prob_1(\R^d)\times\R^d\to\R^d$ satisfy
\begin{equation}\label{mean-field-lipschitz-eq}
	|b(\mu,x)-b(\nu,y)|\leqslant L_x|x-y|+L_{\mu}W_1(\mu,\nu)
\end{equation}
and $|b(\mu,0)|\leqslant C(1+\int|z|\,\dd{\mu}(z))$. The nonlinear characteristic equation is
\begin{gather}\label{nonlinear-characteristic-eq}
	\begin{aligned}
		\dot X_t & = b(\mu_t,X_t), \\
		\mu_t & = \Law(X_t).
	\end{aligned}
\end{gather}
and its law solves
\begin{equation}\label{mean-field-continuity-eq}
	\partial_t\mu_t+\Div\bigl(\mu_tb(\mu_t,\cdot)\bigr)=0.
\end{equation}
The particle approximation is
\begin{gather}\label{mean-field-particle-eq}
	\begin{aligned}
		\dot X_t^{i,N} & = b(\mu_t^N,X_t^{i,N}), \\
		\mu_t^N & = \frac1N\sum_{i=1}^{N}\delta_{X_t^{i,N}}.
	\end{aligned}
\end{gather}

\begin{theorem}\label{mean-field-stability-thm}
	Assume that the characteristic equations below are globally well posed, that $b$ satisfies \eqref{mean-field-lipschitz-eq}, and that a linear growth bound propagates second moments. If $\mu_t$ and $\nu_t$ are characteristic solutions in $\Prob_2(\R^d)$, then
	\begin{equation}\label{mean-field-stability-eq}
		W_2(\mu_t,\nu_t)\leqslant e^{(L_x+L_{\mu})t}W_2(\mu_0,\nu_0).
	\end{equation}
	The same estimate compares the empirical characteristic solution of \eqref{mean-field-particle-eq} with the mean field solution, with initial discrepancy $W_2(\mu_0^N,\mu_0)$.
\end{theorem}
\begin{proof}
	Let $(X_0,Y_0)$ be an optimal coupling and push it through the two characteristic flows. Whenever $\E|X_t-Y_t|^2>0$,
	\begin{align*}
		\dv{}{t}\E|X_t-Y_t|^2
		&\leqslant 2L_x\E|X_t-Y_t|^2+2L_{\mu}W_1(\mu_t,\nu_t)\bigl(\E|X_t-Y_t|^2\bigr)^{1/2}\\
		&\leqslant 2(L_x+L_{\mu})\E|X_t-Y_t|^2,
	\end{align*}
	where $W_1\leqslant W_2$ and the transported coupling bounds $W_2$. Gronwall's inequality proves \eqref{mean-field-stability-eq}. An atomic initial law is permitted, and Proposition \ref{empirical-continuity-prop} identifies its characteristic solution with \eqref{mean-field-particle-eq}.
\end{proof}

A bounded Lipschitz target score together with bounded mixed first and second derivatives of the kernel gives a convenient sufficient condition for \eqref{mean-field-lipschitz-eq}. Confining targets with unbounded score require moment-dependent estimates. Theorem \ref{mean-field-stability-thm} establishes deterministic stability of the mean field equation; quantitative convergence of a finite particle algorithm to the target is a distinct problem treated under different assumptions in \cite{lu2019scaling,korba2020nonasymptotic,salim2022convergence,shi2023finite,banerjee2025improved,carrillo2025stability}.

\section{Diffusion paths and equivalent marginal transport}\label{diffusion-transport-sec}

\subsection{The probability current}

Let $X_t$ solve the It\^o equation on $\R^d$
\begin{equation}\label{diffusion-equation-eq}
	\dd{X_t}=b_t(X_t)\,\dd{t}+\sigma_t(X_t)\,\dd{W_t},
\end{equation}
where $a_t=\sigma_t\sigma_t^{\mathsf T}$. Its generator on smooth observables is
\begin{equation}\label{diffusion-generator-eq}
	\Lop_tf=b_t\cdot\grad f+\frac12a_t:\grad^2f.
\end{equation}
Assume $X_t$ has a positive smooth density $\rho_t$. The adjoint equation is
\begin{equation}\label{fokker-planck-eq}
	\partial_t\rho_t=-\partial_i(b_{t,i}\rho_t)+\frac12\partial_i\partial_j(a_{t,ij}\rho_t).
\end{equation}
Define the current by
\begin{equation}\label{probability-current-eq}
	j_{t,i}=b_{t,i}\rho_t-\frac12\partial_j(a_{t,ij}\rho_t).
\end{equation}
Then \eqref{fokker-planck-eq} is $\partial_t\rho_t+\Div j_t=0$. Wherever $\rho_t>0$, define the current velocity $v_t=j_t/\rho_t$.

\begin{prop}\label{diffusion-current-prop}
	The current velocity is
	\begin{equation}\label{current-velocity-eq}
		v_{t,i}=b_{t,i}-\frac12\partial_ja_{t,ij}-\frac12a_{t,ij}\partial_j\log\rho_t,
	\end{equation}
	and the diffusion law satisfies the continuity equation
	\begin{equation}\label{diffusion-continuity-eq}
		\partial_t\rho_t+\Div(\rho_tv_t)=0.
	\end{equation}
\end{prop}
\begin{proof}
	Expanding the derivative in \eqref{probability-current-eq} and dividing by $\rho_t$ gives \eqref{current-velocity-eq}. The continuity equation is the definition of the current rewritten using $j_t=\rho_tv_t$.
\end{proof}

The second order generator has thus produced a first order continuity equation whose velocity depends on the evolving density. This reduction concerns one-time laws. The deterministic curves solving $\dot Y_t=v_t(Y_t)$ have zero quadratic variation and generally do not reproduce the diffusion path law.

\subsection{Probability flow maps}

Assume $v_t$ is locally Lipschitz in space, has sufficient growth control, and generates a unique flow $\Phi_{s,t}$. Let $\bar\rho_t=(\Phi_{0,t})_{\#}\rho_0$. Lemma \ref{pushforward-differential-lem} shows that $\bar\rho_t$ solves \eqref{diffusion-continuity-eq}. Uniqueness for that continuity equation gives the following consequence.

\begin{corollary}\label{probability-flow-cor}
	The stated regularity and uniqueness assumptions imply
	\[
		(\Phi_{0,t})_{\#}\rho_0=\rho_t.
	\]
	The deterministic ordinary differential equation $\dot Y_t=v_t(Y_t)$ therefore has the same one-time marginals as the diffusion \eqref{diffusion-equation-eq}.
\end{corollary}

A constant diffusion matrix $a$ reduces formula \eqref{current-velocity-eq} to
\[
	v_t=b_t-\frac12a\grad\log\rho_t.
\]
The density score is the correction converting stochastic dispersion into deterministic law transport. Time reversal formulae express the same score through the difference between forwards and backwards diffusion drifts \cite{haussmann1986time,nelson1967dynamical}.

\subsection{Random flow maps and their expectation}

Smooth stochastic differential equations generate stochastic flows of diffeomorphisms $\Phi_{s,t}^{\omega}$ under Kunita's hypotheses \cite{kunita1990stochastic}. Conditional on the noise realisation, the initial point is moved by a random map. Averaging the pushforward gives the dual Markov semigroup
\begin{equation}\label{random-flow-semigroup-eq}
	\Pop_{s,t}^{*}\mu=\E\left[(\Phi_{s,t}^{\omega})_{\#}\mu\right].
\end{equation}
The deterministic probability flow of Corollary \ref{probability-flow-cor} and the random flow representation \eqref{random-flow-semigroup-eq} are two transports of different kinds. The former is nonlinear through its dependence on $\rho_t$ and reproduces the marginal curve with one map. The latter is linear in $\mu$ and represents the Markov kernel as an average of random maps.

The relation can be placed in the diagram
\[
\begin{tikzcd}[column sep=large,row sep=large]
	X_0 \arrow[r,"\Phi_{0,t}^{\omega}"] \arrow[d,"\Law"'] & X_t^{\omega} \arrow[d,"\Law\;\mathrm{and}\;\E"] \\
	\rho_0 \arrow[r,"\Pop_{0,t}^{*}"'] \arrow[dr,"(\Phi_{0,t})_{\#}"'] & \rho_t \\
	& \rho_t \arrow[u,equal]
\end{tikzcd}
\]
where the diagonal map is the deterministic density dependent probability flow.

\subsection{Superposition by path measures}

Let $\Gamma=C([0,T];\R^d)$ and $e_t(\omega)=\omega_t$. If a diffusion has law $Q$ on $\Gamma$ and one-time measure $\mu_t=\rho_t\,\dd{x}$, then $(e_t)_{\#}Q=\mu_t$. Conversely, superposition principles reconstruct a path measure solving the martingale problem from a weak Fokker--Planck solution. The resulting paths are martingale trajectories rather than classical characteristics.

\begin{theorem}\label{superposition-thm}
	Let $(\mu_t)_{t\in[0,T]}$ be a narrowly continuous family of probability measures on $\R^d$ solving the Fokker--Planck equation with Borel coefficients $b$ and $a$, where $a$ is symmetric and nonnegative. Assume
	\[
		\int_0^T\int_{\R^d}\bigl(|b_t(x)|+\|a_t(x)\|\bigr)\,\mu_t(\dd{x})\,\dd{t}<\infty.
	\]
	Under the standard measurability formulation of the weak equation, there exists a probability measure $Q$ on $C([0,T];\R^d)$ such that $(e_t)_{\#}Q=\mu_t$ for every $t$ and the coordinate process solves the martingale problem for $\Lop_t$. Variants with weaker Lyapunov-weighted hypotheses are proved in \cite{figalli2008existence,trevisan2016well,bogachev2015fokker}.
\end{theorem}

\begin{remark}\label{superposition-proof-rem}
	Theorem \ref{superposition-thm} is invoked from the cited superposition theory. Its hypotheses are independent of the smooth Fr\'echet manifold construction and permit singular measures and rough coefficients. The theorem supplies a path space representative of the marginal curve without identifying that representative uniquely.
\end{remark}

\subsection{Langevin diffusion and Otto entropy descent}

Take
\begin{equation}\label{langevin-equation-eq}
	\dd{X_t}=\grad\log\pi(X_t)\,\dd{t}+\sqrt{2}\,\dd{W_t}.
\end{equation}
Here $a=2\operatorname{id}$ and the current velocity is
\begin{equation}\label{langevin-current-velocity-eq}
	v_t=\grad\log\pi-\grad\log\rho_t=-\grad\log\frac{\rho_t}{\pi}.
\end{equation}
The continuity equation generated by \eqref{langevin-current-velocity-eq} is exactly \eqref{wasserstein-entropy-flow-eq}. The sample diffusion, deterministic probability flow and Otto gradient flow therefore have the same one-time law curve. Their information differs at path level: only the diffusion records the martingale noise, whilst Otto space records the evolving marginal and its minimal current velocity.

The stochastic diffusion and the deterministic probability flow are therefore equivalent under every evaluation observation $Q\mapsto(e_t)_{\#}Q$. Their pathwise quadratic variation, temporal correlations and conditional laws remain different. Marginal equivalence expresses precisely the information retained by the chosen observation.

Suppose that $\pi$ is invariant for $\dd{X_t}=b(X_t)\,\dd{t}+\sqrt{2}\,\dd{W_t}$. Writing $c=b-\grad\log\pi$, stationarity is equivalent to $\Div(\pi c)=0$, and the equilibrium current is $\pi c$. Thus $c$ is vertical at $\pi$: it transports particles along a stationary current while producing no density tangent. The identity need not hold without the invariance assumption.

\section{Functional derivatives on the space of laws}\label{lions-sec}

\subsection{The scalar functional derivative}

Let $U\colon\Prob_2(\R^d)\to\R$. A scalar linear functional derivative is a function $\delta U/\delta\mu(\mu,x)$ such that, for suitable $\mu$ and $\nu$,
\begin{equation}\label{linear-functional-derivative-eq}
	U(\nu)-U(\mu)=\int_0^1\int_{\R^d}\frac{\delta U}{\delta\mu}\bigl((1-s)\mu+s\nu,x\bigr)(\nu-\mu)(\dd{x})\dd{s}.
\end{equation}
It is determined modulo a scalar depending on $\mu$. A smooth density identifies this object precisely as a regular cotangent potential: the signed measure $\nu-\mu$ has total mass zero, so constants vanish from the pairing.

Relative entropy between smooth positive densities has
\[
	\frac{\delta\Ent_{\pi}}{\delta\mu}(\mu,x)=\log\frac{\rho(x)}{\pi(x)}+1.
\]
The interaction functional
\[
	U(\mu)=\frac12\int_{\R^d\times\R^d}W(x,y)\mu(\dd{x})\mu(\dd{y})
\]
has scalar derivative
\[
	\frac{\delta U}{\delta\mu}(\mu,x)=\int_{\R^d}W(x,y)\mu(\dd{y})
\]
when $W$ is symmetric.

The scalar functional derivative and the Lions derivative used below require separate hypotheses. Relative entropy has the displayed scalar derivative on a regular absolutely continuous locus; it is not a globally smooth functional on all of $\Prob_2(\R^d)$. Compatibility with a Lions derivative additionally requires spatial differentiability and lift regularity.

\subsection{The lifted derivative}

Let $(\Omega,\mathcal F,\mathbf P)$ be an atomless probability space rich enough to realise every law in $\Prob_2(\R^d)$. Lift $U$ to the Hilbert space $L^2(\Omega;\R^d)$ by
\begin{equation}\label{lions-lift-eq}
	\widetilde U(X)=U(\Law(X)).
\end{equation}
The neighbourhood differentiability in Theorem \ref{lions-structure-thm} supplies the usual representative and lift independence. Pointwise differentiability at a single lift requires a correspondingly pointwise structure theorem.

\begin{theorem}\label{lions-structure-thm}
	Let $\widetilde U\colon L^2(\Omega;\R^d)\to\R$ be law invariant and continuously Fr\'echet differentiable on a neighbourhood of $X$, where the atomless probability space is rich enough for the usual lifting construction. If $\mu=\Law(X)$, there exists a Borel map $\partial_{\mu}U(\mu,\cdot)\in L^2(\mu;\R^d)$ such that
	\begin{equation}\label{lions-representation-eq}
		D\widetilde U(X)[Y]=\E\left[\partial_{\mu}U(\mu,X)\cdot Y\right]
	\end{equation}
	for every $Y\in L^2(\Omega;\R^d)$. The representative depends only on $\mu$ up to $\mu$ almost everywhere equality \cite{wu2020lions}.
\end{theorem}
\begin{proof}
	Law invariance makes the Hilbert gradient equivariant under measure preserving transformations of the atomless probability space. The structure theorem for Lions derivatives gives a Borel function of $X$ alone and proves independence of the chosen lift under the stated neighbourhood regularity \cite{wu2020lions}. Equation \eqref{lions-representation-eq} is the Riesz representation of the derivative.
\end{proof}

When the scalar functional derivative is differentiable in $x$ and the two notions are compatible, the convention used in mean field analysis is
\begin{equation}\label{lions-scalar-relation-eq}
	\partial_{\mu}U(\mu,x)=\grad_x\frac{\delta U}{\delta\mu}(\mu,x).
\end{equation}
The scalar derivative is the density cotangent potential, whilst the Lions derivative is its base space differential. This is the same passage from $[f]$ to $\grad f$ which occurs before the Wasserstein sharp map.

\subsection{The transport chain rule}

\begin{theorem}\label{lions-transport-chain-thm}
	Let $\mu_t$ be an absolutely continuous quadratic Wasserstein curve satisfying
	\[
		\partial_t\mu_t+\Div(\mu_tv_t)=0
	\]
	with $v_t\in L^2(\mu_t;\R^d)$. Suppose $U$ has a scalar functional derivative satisfying \eqref{lions-scalar-relation-eq} and the integrability required to differentiate under the integral. Then
	\begin{equation}\label{lions-transport-chain-eq}
		\dv{}{t}U(\mu_t)=\int_{\R^d}\partial_{\mu}U(\mu_t,x)\cdot v_t(x)\mu_t(\dd{x})
	\end{equation}
	for almost every $t$.
\end{theorem}
\begin{proof}
	Apply the scalar first variation to $\partial_t\mu_t$ and use the weak continuity equation:
	\[
		\dv{}{t}U(\mu_t)=\int_{\R^d}\frac{\delta U}{\delta\mu}(\mu_t,x)\partial_t\mu_t(\dd{x})=\int_{\R^d}\grad_x\frac{\delta U}{\delta\mu}(\mu_t,x)\cdot v_t(x)\mu_t(\dd{x}).
	\]
	Formula \eqref{lions-scalar-relation-eq} gives the result.
\end{proof}

The two descending velocities of $U$ are consequently
\begin{equation}\label{lions-otto-velocity-eq}
	v_U^{\Wass}(\mu,x)=-\partial_{\mu}U(\mu,x)
\end{equation}
and
\begin{equation}\label{lions-stein-velocity-eq}
	v_U^{\Stein}(\mu,\cdot)=-\Iop_{\mu}^{*}\partial_{\mu}U(\mu,\cdot).
\end{equation}
The Stein field is the reproducing kernel image of the Lions derivative. These formulae recover \eqref{langevin-current-velocity-eq} and \eqref{stein-entropy-velocity-eq} when $U$ is relative entropy.

\subsection{The diffusion chain rule}

Let $\mu_t=\Law(X_t)$ for \eqref{diffusion-equation-eq}. Assume $\partial_{\mu}U$ is differentiable in $x$. The Fokker--Planck equation and integration by parts give
\begin{equation}\label{lions-diffusion-chain-eq}
	\dv{}{t}U(\mu_t)=\int_{\R^d}\partial_{\mu}U(\mu_t,x)\cdot b_t(x)\mu_t(\dd{x})+\frac12\int_{\R^d}\tr\left(a_t(x)\partial_x\partial_{\mu}U(\mu_t,x)\right)\mu_t(\dd{x}).
\end{equation}
The second derivative in \eqref{lions-diffusion-chain-eq} is the law functional's response to microscopic quadratic variation. Substituting the current velocity \eqref{current-velocity-eq} into \eqref{lions-transport-chain-eq} and integrating the score term by parts recovers the same formula. The diffusion and deterministic transport chain rules use different representations and have the same action on the marginal functional.

Common noise or a random conditional law introduces a second Lions derivative with respect to two measure variables. Equation \eqref{lions-diffusion-chain-eq} concerns an unconditional deterministic law and hence contains only the spatial derivative of the first Lions derivative \cite{carmona2018probabilistic1,carmona2018probabilistic2,cardaliaguet2019master}.

\subsection{Derivatives of law flows}

Let $\Phi_t\colon\Prob_2(\R^d)\to\Prob_2(\R^d)$ be a nonlinear law flow and suppose
\begin{gather}\label{law-vector-field-eq}
	\begin{aligned}
		\partial_t\mu_t & = \mathcal V(\mu_t), \\
		\mu_t & = \Phi_t(\mu_0).
	\end{aligned}
\end{gather}
A perturbation $\eta_0$ of zero total mass is propagated by
\[
	\eta_t=D\Phi_t(\mu_0)[\eta_0].
\]
Formal differentiation of \eqref{law-vector-field-eq} gives the variational equation
\begin{equation}\label{law-linearisation-eq}
	\partial_t\eta_t=D\mathcal V(\mu_t)[\eta_t].
\end{equation}
This operator is the tangent propagator of the nonlinear semigroup.

The McKean--Vlasov continuity equation
\[
	\partial_t\rho=-\Div\bigl(\rho b(\rho,\cdot)\bigr)
\]
has linearisation
\begin{equation}\label{mckean-vlasov-linearisation-eq}
	\partial_t\eta=-\Div\left(\eta b(\rho,\cdot)+\rho D_{\rho}b(\rho,\cdot)[\eta]\right).
\end{equation}
If $b$ has a scalar measure derivative, then
\[
	D_{\rho}b(\rho,x)[\eta]=\int_{\R^d}\frac{\delta b}{\delta\rho}(\rho,x,y)\eta(\dd{y}).
\]
The Lions calculus therefore produces both a cotangent representative for scalar functionals and the derivative needed to linearise the vector field on the law manifold.

\section{Microscopic fluctuations and the geometry selected by noise}\label{large-deviation-sec}

\subsection{The empirical short time rate}

Consider $N$ independent Brownian particles with generator $\Delta$ and deterministic empirical initial data converging to a density $\rho_0$. Let $p_h(x,\dd{y})$ be the heat transition kernel at time $h$. The joint empirical distribution of initial and final positions has a conditional large deviations principle. Contracting to the final empirical density gives the rate
\begin{equation}\label{adpz-rate-eq}
	J_h(\rho\mid\rho_0)=\inf_{q\in\Pi(\rho_0,\rho)}\Ent\bigl(q\mid\rho_0(\dd{x})p_h(x,\dd{y})\bigr),
\end{equation}
where $\Pi(\rho_0,\rho)$ is the set of couplings with the stated marginals. Thus
\[
	\mathbf P(\mu_h^N\approx\rho\mid\mu_0^N\approx\rho_0)\asymp e^{-NJ_h(\rho\mid\rho_0)}.
\]
The heat kernel short time asymptotic suggests the transport term $W_2^2/(4h)$, whilst the entropy of rearranging independent particles supplies the remaining term. Adams--Dirr--Peletier--Zimmer proved the complete $\Gamma$ convergence statement on a bounded interval under restrictive positive density hypotheses. Duong, Laschos and Renger developed the corresponding path action and multidimensional lower bound, and Erbar, Maas and Renger supplied the higher dimensional recovery construction \cite{adams2011large,duong2013wasserstein,erbar2015multiple}.

\begin{remark}\label{adpz-thm}
	The one dimensional theorem of Adams--Dirr--Peletier--Zimmer, the multidimensional lower-bound and dynamic results of Duong--Laschos--Renger, and the multidimensional recovery construction of Erbar--Maas--Renger have different hypotheses and topologies \cite{adams2011large,duong2013wasserstein,erbar2015multiple}. The respective assumptions of those results imply that the renormalised functionals
	\[
		\rho\longmapsto J_h(\rho\mid\rho_0)-\frac{1}{4h}W_2^2(\rho_0,\rho)
	\]
	have the common limiting entropy contribution
	\begin{equation}\label{adpz-gamma-limit-eq}
		\frac12\Ent(\rho)-\frac12\Ent(\rho_0),
	\end{equation}
	with convergence understood in the topology of the cited theorem. This paragraph records the shared asymptotic conclusion rather than a new theorem combining the distinct results.
\end{remark}

The Jordan--Kinderlehrer--Otto step for $\partial_t\rho=\Delta\rho$ is
\begin{equation}\label{minimising-movement-heat-step-eq}
	\rho_h\in\operatorname*{arg\,min}_{\rho}\left\{\frac{1}{2h}W_2^2(\rho_0,\rho)+\Ent(\rho)\right\}.
\end{equation}
The results summarised in Remark \ref{adpz-thm} say that twice the microscopic rate $J_h$, after subtraction of the constant $\Ent(\rho_0)$ to leading order, has the same small time variational structure as \eqref{minimising-movement-heat-step-eq}. The Wasserstein metric and entropy have therefore arisen together from the probability of atypical empirical motion.

\subsection{The dynamic action and energy dissipation}

The dynamic empirical rate for independent Brownian particles has the formal form
\begin{equation}\label{brownian-dynamic-rate-eq}
	I_{[0,T]}(\rho)=\frac14\int_0^T\left\|\partial_t\rho_t-\Delta\rho_t\right\|_{\dot H^{-1}_{\rho_t}}^2\,\dd{t}.
\end{equation}
The identity $\Delta\rho=-\Kop^{\Wass}_{\rho}\dd{\Ent}(\rho)$ converts this expression to
\[
	I_{[0,T]}(\rho)=\frac14\int_0^T\left\|\partial_t\rho_t+\Kop^{\Wass}_{\rho_t}\dd{\Ent}(\rho_t)\right\|_{T_{\rho_t,\Wass}}^2\,\dd{t}.
\]
Expanding the square and using duality gives
\begin{equation}\label{brownian-action-decomposition-eq}
	I_{[0,T]}(\rho)=\frac14\int_0^T\|\partial_t\rho_t\|_{T_{\rho_t,\Wass}}^2\,\dd{t}+\frac14\int_0^T\|\dd{\Ent}(\rho_t)\|_{T^{*}_{\rho_t,\Wass}}^2\,\dd{t}+\frac12\left(\Ent(\rho_T)-\Ent(\rho_0)\right).
\end{equation}
The zero cost curve is the heat flow. Formula \eqref{brownian-action-decomposition-eq} displays the general relation between empirical fluctuations and gradient flow: the tangent norm measures kinematic cost, the cotangent norm measures thermodynamic force, and their cross term is the energy difference \cite{mielke2014variational}.

\subsection{The Stein many particle action}

The following Stein action is formal unless the hypotheses of a specific many particle large deviations theorem are imposed. N\"usken and Renger construct the tangent--cotangent pair and prove long time variational statements under their assumptions, whilst presenting the pathwise large deviations principle formally \cite{nusken2023large}. Regularised Stein flows provide a complementary comparison with Wasserstein transport \cite{he2025regularized}. Regular measure curves have the candidate rate functional
\begin{equation}\label{stein-dynamic-rate-eq}
	I^{\Stein}_{[0,T]}(\rho)=\frac14\int_0^T\left\|\partial_t\rho_t+\Kop^{\Stein}_{\rho_t}\dd{\Ent}_{\pi}(\rho_t)\right\|_{T_{\rho_t,\Stein}}^2\,\dd{t}.
\end{equation}
Their tangent space is the quotient Hilbert range of kernel velocities and their cotangent is the completion induced by the kernel Stein information, agreeing with \S\ref{stein-sec}. Expanding \eqref{stein-dynamic-rate-eq} gives
\begin{align}
	I^{\Stein}_{[0,T]}(\rho)
	&=\frac12\left(\Ent_{\pi}(\rho_T)-\Ent_{\pi}(\rho_0)\right)
	+\frac14\int_0^T\|\partial_t\rho_t\|_{T_{\rho_t,\Stein}}^2\,\dd{t} \notag\\
	&\hspace{1cm}+\frac14\int_0^T\|\dd{\Ent}_{\pi}(\rho_t)\|_{T^{*}_{\rho_t,\Stein}}^2\,\dd{t}. \label{stein-action-decomposition-eq}
\end{align}
The long time leading contribution for empirical ergodic averages is one quarter of the Stein--Fisher information under their normalisation. Hence the same cotangent norm controls deterministic entropy dissipation and the exponential suppression of stochastic empirical deviations.

\begin{remark}\label{deterministic-fluctuation-rem}
	Suppose two geometries induce the same deterministic gradient vector field for a functional along a given curve. Their large deviations actions can nevertheless differ because the tangent and cotangent norms away from the zero-cost direction remain different. Deterministic law equivalence therefore retains the typical macroscopic evolution and forgets the covariance of microscopic fluctuations. Equality of the complete actions is a stronger equivalence than equality of their minimisers.
\end{remark}

\begin{remark}\label{noise-geometry-rem}
	A deterministic gradient equation can admit several metric representations. The microscopic noise covariance distinguishes them: independent Brownian noise yields the Wasserstein tangent norm, whilst the correlated kernel noise in the N\"usken--Renger system yields the Stein tangent norm. The generator to large deviations passage therefore identifies the Onsager operator from the covariance of microscopic fluctuations.
\end{remark}

\section{Generators, semigroups and tangent propagators}\label{semigroup-sec}

\subsection{Four operator levels}

Let $(\Pop_t)_{t\geqslant0}$ be a Markov semigroup on observables with generator $\Lop$, and let $(\Pop_t^{*})_{t\geqslant0}$ be its dual action on measures. Whenever a density belongs to the domain of the adjoint generator,
\begin{gather}\label{linear-semigroup-density-eq}
	\begin{aligned}
		\rho_t & = \Pop_t^{*}\rho_0, \\
		\partial_t\rho_t & = \Lop^{*}\rho_t.
	\end{aligned}
\end{gather}
Mass preservation gives
\[
	\int_M\Lop^{*}\rho\,\dd{\mathrm{vol}}=\int_M\Lop 1\,\rho\,\dd{\mathrm{vol}}=0.
\]
Thus $\Lop^{*}\rho$ is a density tangent whenever it has the required regularity.

\begin{prop}\label{generator-tangent-prop}
	Let $\mathcal C$ be a linear space of smooth densities invariant under $\Pop_t^{*}$ and contained in $\operatorname{Dom}(\Lop^{*})$. Assume $\Lop^{*}\rho\in C^{\infty}_0(M)$ for $\rho\in\mathcal C\cap\Prob^{\infty}_+(M)$. Then
	\[
		\mathcal V_{\Lop}(\rho)=\Lop^{*}\rho
	\]
	defines a vector field on that smooth core. Continuity and linearity of $\Pop_t^{*}$ in an ambient locally convex topology give
	\begin{equation}\label{linear-tangent-propagator-eq}
		D(\Pop_t^{*})_{\rho}[\sigma]=\Pop_t^{*}\sigma.
	\end{equation}
	Without preservation of such a core, $\Lop^{*}\rho$ is a distributional tangent at points of its operator domain rather than a vector field on all of $\Prob^{\infty}_+(M)$.
\end{prop}
\begin{proof}
	Mass preservation gives zero total mass whenever the adjoint action is defined. The smooth-core assumption places this variation in $C^{\infty}_0(M)$. Formula \eqref{linear-tangent-propagator-eq} follows from linearity, and differentiation at $t=0$ gives the infinitesimal tangent propagator $\Lop^{*}\sigma$ on its domain.
\end{proof}

The tangent space itself consists of signed density variations. Linear operators enter at four distinct levels:
\begin{enumerate}[label=(\roman*)]
	\item the continuity operator $\Aop_{\rho}$ sends sample velocities to density tangents;
	\item the Onsager operator $\Kop_{\rho}$ sends density covectors to density tangents;
	\item the adjoint generator $\rho\mapsto\Lop^{*}\rho$ is a vector field on an invariant smooth core, when such a core is preserved;
	\item the derivative $D\Phi_t(\rho)$ of a law flow propagates tangent perturbations.
\end{enumerate}
The same formula may represent more than one level only after the corresponding domains have been identified. A linear semigroup has $D\Phi_t=\Pop_t^{*}$. A nonlinear McKean--Vlasov or Stein flow has a tangent propagator governed by \eqref{law-linearisation-eq}, which depends on the reference trajectory.

\subsection{The backwards equation and cotangent propagation}

The observable semigroup solves the backwards Kolmogorov equation
\begin{gather}\label{backward-kolmogorov-eq}
	\begin{aligned}
		\partial_t f_t & = \Lop f_t, \\
		f_t & = \Pop_tf_0.
	\end{aligned}
\end{gather}
The density semigroup solves the forwards equation \eqref{linear-semigroup-density-eq}. Their pairing is preserved in the sense
\begin{equation}\label{semigroup-duality-eq}
	\int_M(\Pop_tf)\rho\,\dd{\mathrm{vol}}=\int_Mf(\Pop_t^{*}\rho)\,\dd{\mathrm{vol}}.
\end{equation}
A regular dual test function can therefore be propagated backwards by the observable semigroup, whilst a density tangent is propagated forwards by the dual semigroup. This is propagation under the distributional pairing. It need not preserve a density dependent Otto or Stein cotangent norm and should not be read as parallel transport for either weak metric.

Let $\rho_t=\Pop_t^{*}\rho_0$ and let $\sigma_t=\Pop_t^{*}\sigma_0$. Let $f_t=\Pop_{T-t}f_T$. Differentiation gives
\[
	\dv{}{t}\int_Mf_t\sigma_t\,\dd{\mathrm{vol}}=\int_M(-\Lop f_t)\sigma_t\,\dd{\mathrm{vol}}+\int_Mf_t\Lop^{*}\sigma_t\,\dd{\mathrm{vol}}=0.
\]
Thus the backwards observable is the cotangent dual of the forwards tangent propagator.

\subsection{Factorisation of reversible generators}

Suppose $\pi$ is invariant and the diffusion generator is symmetric in $L^2(\pi)$. The overdamped Langevin generator is
\begin{equation}\label{reversible-generator-eq}
	\Lop_{\pi}f=\Delta f+\langle\grad\log\pi,\grad f\rangle_g.
\end{equation}
It satisfies
\begin{equation}\label{reversible-dirichlet-eq}
	-\int_Mf\Lop_{\pi}h\,\pi\,\dd{\mathrm{vol}}=\int_M\langle\grad f,\grad h\rangle_g\pi\,\dd{\mathrm{vol}}.
\end{equation}
The adjoint acting on a density $\rho$ with respect to volume is
\[
	\Lop_{\pi}^{*}\rho=\Delta\rho-\Div(\rho\grad\log\pi).
\]
Lemma \ref{entropy-differential-lem} and \eqref{wasserstein-onsager-eq} give the factorisation
\begin{equation}\label{generator-onsager-factorisation-eq}
	\Lop_{\pi}^{*}\rho=-\Kop_{\rho}^{\Wass}\dd{\Ent}_{\pi}(\rho).
\end{equation}
The generator has therefore become a density space gradient vector field after evaluation at $\rho$. Contrast this with a general nonreversible generator, whose adjoint may contain a current component not represented by the symmetric Wasserstein Onsager operator.

The Stein entropy vector field has the analogous form
\begin{equation}\label{stein-nonlinear-generator-eq}
	\mathcal V_{\Stein}(\rho)=-\Kop_{\rho}^{\Stein}\dd{\Ent}_{\pi}(\rho).
\end{equation}
It is nonlinear even when the target and kernel are fixed, since the adjoint inclusion $\Iop_{\rho}^{*}$ and the continuity operator both depend on $\rho$. Where it is well posed and autonomous, the flow generated by \eqref{stein-nonlinear-generator-eq} is a nonlinear semigroup on measures. No linear Markov semigroup on observables has this flow as its dual in general.

\begin{remark}\label{generator-two-geometry-rem}
	If the same adjoint generator satisfies
	\[
		\Lop^{*}\rho=-\Kop^{\Wass}_{\rho}\dd{\mathcal F}(\rho)=-\Kop^{\Stein}_{\rho}\dd{\mathcal F}(\rho)
	\]
	on a region, then the two geometries are dynamically equivalent for the distinguished force $\dd{\mathcal F}(\rho)$ on that region. The equality does not determine their action on other covectors, their tangent norms or their fluctuation theories. The criterion in Theorem \ref{comparison-levels-thm} says precisely that $\Mop_{\rho}\grad(\delta\mathcal F/\delta\rho)=\grad(\delta\mathcal F/\delta\rho)$.
\end{remark}

\subsection{Linearisation of the Stein flow}

Write
\[
	\mathcal V_{\Stein}(\rho)=-\Aop_{\rho}\Iop_{\rho}^{*}\grad\log\frac{\rho}{\pi}.
\]
Let $\eta$ be a smooth zero mass variation. Euclidean coordinates with a scalar kernel give, after differentiation,
\begin{align}
	D\mathcal V_{\Stein}(\rho)[\eta]
	&=\Div\left(\eta\Iop_{\rho}^{*}\grad\log\frac{\rho}{\pi}\right)
	+\Div\left(\rho\Iop_{\eta}^{*}\grad\log\frac{\rho}{\pi}\right) \notag\\
	&\hspace{1cm}+\Div\left(\rho\Iop_{\rho}^{*}\grad\frac{\eta}{\rho}\right), \label{stein-linearisation-eq}
\end{align}
where
\[
	(\Iop_{\eta}^{*}u)(y)=\int_{\R^d}k(x,y)u(x)\eta(x)\,\dd{x}.
\]
The three terms respectively vary the advected density, the measure entering the kernel integral, and the entropy cotangent. Equation \eqref{stein-linearisation-eq} is the infinitesimal generator of the tangent propagator $D\Phi_t^{\Stein}(\rho_0)$. Its form demonstrates why the tangent dynamics of a nonlinear density flow contains more information than the original sample velocity.

\section{Stein equations from diffusion semigroups and resolvents}\label{resolvent-sec}

\subsection{The reversible Poisson equation}

Fix $q\in\Prob^{\infty}_+(M)$ and consider
\begin{equation}\label{q-langevin-generator-eq}
	\Lop_q u=\Delta u+\langle\grad\log q,\grad u\rangle_g.
\end{equation}
The operator is self-adjoint and nonpositive in $L^2(q)$ on its natural domain, with Dirichlet form
\begin{equation}\label{q-dirichlet-form-eq}
	\mathcal E_q(u,h)=-\int_Mu\Lop_qh\,q\,\dd{\mathrm{vol}}=\int_M\langle\grad u,\grad h\rangle_gq\,\dd{\mathrm{vol}}.
\end{equation}
It factorises through the density Stein operator as
\begin{equation}\label{generator-stein-factorisation-eq}
	\Lop_qu=\Tspace_q(\grad u).
\end{equation}
Let $(\Pop_t^q)_{t\geqslant0}$ be the associated reversible semigroup.

Since $M$ is connected and closed and $q$ is smooth and strictly positive, a Poincar\'e inequality holds. Write $\lambda>0$ for a spectral-gap constant:
\begin{equation}\label{poincare-q-eq}
	\int_M|u-\E_q[u]|^2q\,\dd{\mathrm{vol}}\leqslant\frac1\lambda\int_M|\grad u|_g^2q\,\dd{\mathrm{vol}}.
\end{equation}
Then $\|\Pop_t^qg\|_{L^2(q)}\leqslant e^{-\lambda t}\|g\|_{L^2(q)}$ for centred $g$.

\begin{theorem}\label{semigroup-stein-solution-thm}
	Let $g\in L^2(q)$ satisfy $\E_q[g]=0$. Define
	\begin{equation}\label{poisson-semigroup-eq}
		u_g=\int_0^{\infty}\Pop_t^qg\,\dd{t}
	\end{equation}
	as a Bochner integral in $L^2(q)$. Then $u_g$ belongs to the form domain, solves
	\begin{equation}\label{poisson-equation-eq}
		\Lop_qu_g=-g
	\end{equation}
	weakly, and the vector field
	\begin{equation}\label{semigroup-stein-field-eq}
		\psi_g=\grad u_g
	\end{equation}
	solves the Stein equation
	\begin{equation}\label{semigroup-stein-equation-eq}
		\Tspace_q\psi_g=-g.
	\end{equation}
	Moreover,
	\begin{equation}\label{stein-solution-bound-eq}
		\|\psi_g\|_{L^2(q)}^2=\langle g,(-\Lop_q)^{-1}g\rangle_{L^2(q)}\leqslant\frac1\lambda\|g\|_{L^2(q)}^2.
	\end{equation}
\end{theorem}
\begin{proof}
	Exponential decay makes \eqref{poisson-semigroup-eq} convergent in $L^2(q)$. Fix $T<\infty$ and set $u_g^T=\int_0^T\Pop_t^qg\,\dd{t}$. Semigroup differentiation gives
	\[
		\Lop_qu_g^T=\int_0^T\dv{}{t}\Pop_t^qg\,\dd{t}=\Pop_T^qg-g.
	\]
	The final term tends to zero in $L^2(q)$. Closedness of the generator or, equivalently, the weak form associated with \eqref{q-dirichlet-form-eq}, gives \eqref{poisson-equation-eq}. Formula \eqref{generator-stein-factorisation-eq} proves \eqref{semigroup-stein-equation-eq}.

	Taking the $L^2(q)$ pairing of $-\Lop_qu_g=g$ with $u_g$ gives
	\[
		\|\grad u_g\|_{L^2(q)}^2=\langle g,u_g\rangle_{L^2(q)}=\langle g,(-\Lop_q)^{-1}g\rangle_{L^2(q)}.
	\]
	The spectral gap bounds the inverse of $-\Lop_q$ on centred functions by $1/\lambda$, proving \eqref{stein-solution-bound-eq}.
\end{proof}

Theorem \ref{semigroup-stein-solution-thm} makes the generator to semigroup passage a direct covector construction. A centred observable $g$ is propagated by the semigroup, integrated to obtain a Poisson potential, and differentiated to obtain a particle field. The sufficient statistics of \S\ref{statistical-submanifolds-sec} permit the choice
\begin{equation}\label{statistic-semigroup-field-eq}
	\psi_j^{\eta}=\grad\int_0^{\infty}\Pop_t^{q_{\eta}}\bigl(T_j-m_j(\eta)\bigr)\,\dd{t}.
\end{equation}
These fields satisfy \eqref{statistic-stein-equation-eq}. Smooth statistics give smooth potentials and fields by elliptic regularity. Membership in a prescribed reproducing kernel Hilbert space remains an additional range condition.

\subsection{The resolvent interpretation}

Fix $\alpha>0$ and define
\begin{equation}\label{resolvent-eq}
	\Rop_{\alpha}^{q}g=\int_0^{\infty}e^{-\alpha t}\Pop_t^qg\,\dd{t}=(\alpha-\Lop_q)^{-1}g.
\end{equation}
The zero frequency resolvent on centred functions is $(-\Lop_q)^{-1}$. The particle field in Theorem \ref{semigroup-stein-solution-thm} is
\[
	\psi_g=\grad(-\Lop_q)^{-1}g.
\]
It is therefore an elliptic response of the diffusion to the covector observable $g$.

The Green kernel of a sufficiently smoothing power $(\alpha-\Lop_q)^{-m}$ gives a target-adapted reproducing kernel of scalar potentials when $m$ is large enough for point evaluation to be continuous. Taking gradients of this potential Hilbert space produces a vector-valued reproducing kernel. Such a kernel incorporates the geometry and mixing of the target diffusion: slow spectral modes receive large resolvent weight, whilst high frequency modes are regularised. Standard Stein variational kernels are usually specified independently of $\Lop_q$; the resolvent construction supplies a principled alternative when target dependent kernels are acceptable.

\begin{prop}\label{resolvent-kernel-prop}
	Let $\Lop_q$ be uniformly elliptic on the connected closed manifold $M$, let $\alpha>0$, and let $\mathcal G_{\alpha,m}=\Ran(\alpha-\Lop_q)^{-m/2}$ with graph norm
	\[
		\|u\|_{\mathcal G_{\alpha,m}}=\|(\alpha-\Lop_q)^{m/2}u\|_{L^2(q)}.
	\]
	If $m>d/2+1$, then $\mathcal G_{\alpha,m}\hookrightarrow C^1(M)$ and is a scalar reproducing kernel Hilbert space; its gradient image, furnished with the quotient norm by constant potentials, is a Hilbert space of continuous vector fields with continuous point evaluation. If the $C^1$ vector-field embedding in \eqref{kernel-inclusion-eq} is required, it is sufficient to assume $m>d/2+2$. The lower threshold gives a distributional interpretation of the continuity operator.
\end{prop}
\begin{proof}
	Elliptic functional calculus identifies the graph norm with a Sobolev norm of order $m$. The inequality $m>d/2+1$ gives $C^1$ scalar potentials and hence continuous gradient evaluation. Quotienting by the closed kernel of the gradient produces the vector-field Hilbert space. The stronger inequality gives $C^2$ potentials and $C^1$ gradient fields. Continuity of evaluation supplies the corresponding reproducing kernels \cite{aronszajn1950theory}.
\end{proof}

The proposition does not assert that every preselected Stein kernel is a diffusion resolvent. It establishes a direct class for which kernel geometry is constructed from the generator and semigroup themselves.

The finite rank construction of Proposition \ref{finite-rank-prescription-prop} may be applied to any chosen collection of smooth Poisson fields $\grad(-\Lop_q)^{-1}g_j$. It produces a kernel which agrees exactly with the Otto particle response on their span, or applies any prescribed positive self-adjoint preconditioner there. Resolvent kernels offer an infinite dimensional extension whose comparison operator is determined by the spectral response of the diffusion.

\subsection{Stein factors and regularity}

Bounds on $\psi_g$ and its derivatives are Stein factors. The spectral estimate \eqref{stein-solution-bound-eq} controls the weighted square integrable velocity. Pointwise and higher derivative bounds require elliptic regularity, curvature estimates or coupling estimates for the semigroup \cite{barbour1988stein,gorham2019measuring,le2024diffusion}. These estimates decide whether the semigroup constructed field belongs to a selected reproducing kernel Hilbert space. The range condition in Theorem \ref{exponential-stein-covector-thm} is consequently a regularity statement about the Poisson equation.

\section{Feynman--Kac paths, cotangent equations and reaction transport}\label{feynman-kac-sec}

\subsection{Weighted semigroups}

Let $X_t$ have generator $\Lop$ and let $V\colon M\to\R$ be bounded below. The Feynman--Kac semigroup is
\begin{equation}\label{feynman-kac-semigroup-eq}
	\Pop_t^Vf(x)=\E_x\left[\exp\left(-\int_0^tV(X_s)\,\dd{s}\right)f(X_t)\right].
\end{equation}
It solves
\begin{gather}\label{feynman-kac-backward-eq}
	\begin{aligned}
		\partial_tu & = (\Lop-V)u, \\
		u(0) & = f.
	\end{aligned}
\end{gather}
under the usual regularity hypotheses \cite{kac1949distributions,delmoral2004feynman}. Suppose $V\geqslant-C$. The weighted semigroup then grows at most as $e^{Ct}$ in the natural supremum norm, and the resolvent identity
\begin{equation}\label{feynman-kac-resolvent-eq}
	\int_0^{\infty}e^{-\alpha t}\Pop_t^Vf\,\dd{t}=(\alpha-\Lop+V)^{-1}f
\end{equation}
holds for $\alpha>C$, or more generally for $\alpha$ above the semigroup growth bound. Every $\alpha>0$ is permitted when $V\geqslant0$.
Hence sample paths calculate scalar cotangent potentials solving elliptic or parabolic equations. Their spatial gradients give Otto velocities, and their kernel images give Stein velocities.

Individual diffusion paths serve as stochastic characteristics. Classical characteristics belong to first order equations and carry different pathwise information. The expectation in \eqref{feynman-kac-semigroup-eq} combines their terminal values and accumulated potential. The same paths define, at the law level, an unnormalised measure
\begin{equation}\label{unnormalised-fk-measure-eq}
	\widetilde\mu_t(f)=\E_{\mu_0}\left[\exp\left(-\int_0^tV(X_s)\,\dd{s}\right)f(X_t)\right].
\end{equation}
Its density solves
\begin{equation}\label{unnormalised-fk-forward-eq}
	\partial_t\widetilde\rho_t=\Lop^{*}\widetilde\rho_t-V\widetilde\rho_t.
\end{equation}
The potential term changes total mass and therefore points out of the probability manifold.

\subsection{Normalisation and Fisher--Rao reaction}

Let $Z_t=\int_M\widetilde\rho_t\,\dd{\mathrm{vol}}$ and $\rho_t=\widetilde\rho_t/Z_t$. The identity $\Lop1=0$ gives
\[
	\dot Z_t=-Z_t\E_{\rho_t}[V].
\]
A direct calculation gives
\begin{equation}\label{normalised-fk-eq}
	\partial_t\rho_t=\Lop^{*}\rho_t-\bigl(V-\E_{\rho_t}[V]\bigr)\rho_t.
\end{equation}
The second term has integral zero and is consequently tangent to the probability manifold.

The Fisher--Rao cotangent form on probability densities is
\begin{equation}\label{fisher-rao-cotangent-eq}
	\langle f,h\rangle_{T^{*}_{\rho,\Fish}}=\int_M\bigl(f-\E_{\rho}[f]\bigr)\bigl(h-\E_{\rho}[h]\bigr)\rho\,\dd{\mathrm{vol}}.
\end{equation}
Its Onsager operator is
\begin{equation}\label{fisher-rao-onsager-eq}
	\Kop^{\Fish}_{\rho}f=\rho\bigl(f-\E_{\rho}[f]\bigr).
\end{equation}
Indeed, pairing a covector $h$ with \eqref{fisher-rao-onsager-eq} gives \eqref{fisher-rao-cotangent-eq}. The expectation functional $\mathcal V(\rho)=\int_MV\rho\,\dd{\mathrm{vol}}$ has differential $[V]$; hence
\begin{equation}\label{feynman-kac-fisher-gradient-eq}
	-\Kop^{\Fish}_{\rho}\dd{\mathcal V}(\rho)=-\bigl(V-\E_{\rho}[V]\bigr)\rho.
\end{equation}
The normalised Feynman--Kac reaction is exactly Fisher--Rao descent of the expected potential.

When $\Lop^{*}\rho=-\Kop^{\Wass}_{\rho}\dd{\mathcal F}(\rho)$, equation \eqref{normalised-fk-eq} becomes
\begin{equation}\label{wfr-fk-eq}
	\partial_t\rho=-\Kop^{\Wass}_{\rho}\dd{\mathcal F}(\rho)-\Kop^{\Fish}_{\rho}V.
\end{equation}
Motivated by \eqref{wfr-fk-eq}, one may define the composite Stein--Fisher--Rao equation
\[
	\partial_t\rho=-\Kop^{\Stein}_{\rho}\dd{\mathcal F}(\rho)-\Kop^{\Fish}_{\rho}V.
\]
This changes the transport part to a nonlinear law evolution and is not generally the normalised Feynman--Kac equation of the original linear generator. The Wasserstein equation \eqref{wfr-fk-eq} is related to Hellinger--Kantorovich and Wasserstein--Fisher--Rao geometry \cite{chizat2018interpolating,liero2018optimal}; a precise metric interpretation requires the joint action and boundary conditions.

\subsection{Doob transforms and conditioned transport}

Suppose $h>0$ solves
\begin{equation}\label{ground-state-eq}
	(\Lop-V)h=-\lambda h.
\end{equation}
The Doob transformed generator is
\begin{equation}\label{doob-generator-eq}
	\Lop^hf=h^{-1}(\Lop-V+\lambda)(hf).
\end{equation}
A diffusion generator with covariance $a$ acquires the modified drift $a\grad\log h$. Thus the logarithmic differential $\grad\log h$ converts a path weight into an additional particle transport. The scalar $\log h$ is a cotangent potential obtained from the Feynman--Kac spectral problem, whilst $a\grad\log h$ is its mobility weighted velocity representative.

This construction parallels entropy descent. Relative entropy supplies the potential $\log(\rho/\pi)$ and the current velocity $-\grad\log(\rho/\pi)$. Feynman--Kac conditioning supplies the positive eigenfunction potential $\log h$ and the Doob drift $a\grad\log h$. Both are instances of a scalar cotangent being converted into a sample velocity by a mobility operator.

\subsection{Path measures and evaluation}

Let $Q_t$ be the restriction of the original diffusion law to paths on $[0,t]$ and define a family of tilted path measures by
\begin{equation}\label{tilted-path-law-eq}
	\frac{\dd{Q_t^V}}{\dd{Q_t}}=\frac1{Z_t}\exp\left(-\int_0^tV(X_s)\,\dd{s}\right).
\end{equation}
The normalised Feynman--Kac marginal measure is then exactly
\[
	\mu_t=(e_t)_{\#}Q_t^V.
\]
When $\mu_t$ has a density with respect to volume, that density is the $\rho_t$ of \eqref{normalised-fk-eq}.
The family $(Q_t^V)$ need not be projectively consistent, because extending the terminal time adds a future path weight. A single terminally tilted measure $Q_T^V$ has intermediate marginals involving both a forwards Feynman--Kac density and a backwards Feynman--Kac factor; these marginals are described by a Doob transform rather than by \eqref{normalised-fk-eq} alone. Path weighting therefore occurs before projection to Otto space. Equation \eqref{normalised-fk-eq} is the evaluation image of the time matched family \eqref{tilted-path-law-eq}, and information concerning temporal correlations and the accumulated potential is lost when only $(\rho_t)$ is retained.

The family of tilted path measures and the normalised reaction transport are equivalent under the time matched evaluation observation. Distinct tilts may have the same marginal at one-time, and a single terminal tilt carries backwards information absent from the instantaneous normalised equation. Evaluation therefore supplies a prescribed law equivalence whilst forgetting most of the path weighting.

\section{Abstract Wiener spaces and path laws}\label{infinite-dimensional-sec}

\subsection{A differentiable density core over Gaussian measure}

Let $(E,H,\gamma)$ be an abstract Wiener space: $E$ is a separable Banach space, $H$ is a separable Hilbert space continuously and densely embedded in $E$, and $\gamma$ is the centred Gaussian measure whose Cameron--Martin space is $H$. No translation invariant Lebesgue measure exists on an infinite dimensional Banach space; densities must therefore be defined relative to $\gamma$ or another reference measure.

Let $\mathbf D^{1,2}(\gamma)$ be the first Gaussian Sobolev space and let $D_H$ denote the closed Cameron--Martin derivative. Set
\[
	X=\mathbf D^{1,2}(\gamma)\cap L^{\infty}(\gamma)
\]
with norm $\|f\|_X=\|f\|_{\mathbf D^{1,2}}+\|f\|_{L^{\infty}}$. Consider
\begin{equation}\label{wiener-density-core-eq}
	\Prob_X^+(E,\gamma)=\left\{\rho\in X:\int_E\rho\,\dd{\gamma}=1,\ \essinf\rho>0\right\}.
\end{equation}
The condition $\essinf\rho>0$ is open in $L^{\infty}$, and the integral constraint is a complemented affine hyperplane. Hence \eqref{wiener-density-core-eq} is a Banach manifold modelled on
\[
	X_0=\left\{\sigma\in X:\int_E\sigma\,\dd{\gamma}=0\right\}.
\]
This core is narrower than the statistical manifolds of Newton and Pistone--Sempi, although it has the advantage that $D_H\log\rho=\rho^{-1}D_H\rho$ belongs to $L^2(\gamma;H)$. Newton's balanced and Hilbert charts permit larger classes of equivalent measures and can replace this core when their chart domains are intersected with the domain of the Cameron--Martin derivative \cite{newton2012hilbert,newton2016balanced,pistone1995infinite}.

\subsection{Gaussian continuity and the weighted Dirichlet form}

Let $\mathscr C$ be the smooth cylindrical functions on $E$. An $H$ valued field $v$ for which the right hand side is continuous defines the density tangent $\Aop_{\rho}v$ by
\begin{equation}\label{wiener-continuity-eq}
	\langle\Aop_{\rho}v,f\rangle=\int_E\langle D_Hf,v\rangle_H\rho\,\dd{\gamma}
\end{equation}
for every $f\in\mathscr C$.
This is the Gaussian divergence of $\rho v$ with the sign convention matching the finite dimensional continuity operator. If $\delta$ denotes the Malliavin divergence, defined as the adjoint of $D_H$, then $\Aop_{\rho}v=\delta(\rho v)$ in the distributional convention
\[
	\int_Ef\delta(\rho v)\,\dd{\gamma}=\int_E\langle D_Hf,v\rangle_H\rho\,\dd{\gamma}.
\]

The weighted form is
\begin{equation}\label{wiener-weighted-form-eq}
	\mathcal E_{\rho}(f,h)=\int_E\langle D_Hf,D_Hh\rangle_H\rho\,\dd{\gamma}.
\end{equation}
If $0<m\leqslant\rho\leqslant M<\infty$, this form is equivalent to the classical Ornstein--Uhlenbeck form and is closed on $\mathbf D^{1,2}(\gamma)/\R$. Its dual is the weighted negative Sobolev space. Define
\begin{equation}\label{wiener-otto-onsager-eq}
	\Kop^{\Wass}_{\rho}f=\Aop_{\rho}D_Hf.
\end{equation}

\begin{prop}\label{wiener-otto-prop}
	The operator \eqref{wiener-otto-onsager-eq} is the Riesz isomorphism from $\mathbf D^{1,2}(\gamma)/\R$, furnished with \eqref{wiener-weighted-form-eq}, to its Hilbert dual whenever $\rho\in\Prob_X^+(E,\gamma)$ is bounded above and below. The entropy relative to $\gamma$ has regular cotangent $[\log\rho]$, and its descending flow satisfies
	\begin{equation}\label{ou-density-flow-eq}
		\partial_t\rho=\Lop_{\gamma}\rho
	\end{equation}
	weakly, where $\Lop_{\gamma}$ is the Ornstein--Uhlenbeck generator determined by
	\begin{equation}\label{ou-dirichlet-eq}
		\int_Ef\Lop_{\gamma}h\,\dd{\gamma}=-\int_E\langle D_Hf,D_Hh\rangle_H\,\dd{\gamma}.
	\end{equation}
\end{prop}
\begin{proof}
	Closedness and coercivity on the quotient follow from equivalence with the Gaussian Sobolev form and the Gaussian Poincar\'e inequality. The Riesz argument of Theorem \ref{otto-riesz-thm} therefore applies. Testing $-\Kop^{\Wass}_{\rho}\log\rho$ against $f\in\mathscr C$ identifies the entropy flow:
	\[
		\langle-\Kop^{\Wass}_{\rho}\log\rho,f\rangle=-\int_E\langle D_Hf,D_H\log\rho\rangle_H\rho\,\dd{\gamma}=-\int_E\langle D_Hf,D_H\rho\rangle_H\,\dd{\gamma}.
	\]
	Formula \eqref{ou-dirichlet-eq} identifies this distribution with $\Lop_{\gamma}\rho$.
\end{proof}

Let $\pi=Z^{-1}e^{-V}\gamma$. A density $\mu=\rho\gamma$ has relative entropy
\[
	\Ent_{\pi}(\rho)=\int_E\rho\log\rho\,\dd{\gamma}+\int_EV\rho\,\dd{\gamma}+\log Z.
\]
Its cotangent is $[\log\rho+V]$. The invariant infinite dimensional statement is first made at the level of forms: whenever the corresponding gradient form with logarithmic drift is closable and quasi-regular, entropy descent is its weak Fokker--Planck equation. A Hilbert state space realisation with $E$ and $\gamma=\mathcal N(0,C)$ has the formal generator
\[
	\Lop_Vf=\tr(CD^2f)-\langle x+C\grad V(x),Df(x)\rangle_E
\]
and stochastic equation
\[
	\dd{X_t}=-\bigl(X_t+C\grad V(X_t)\bigr)\,\dd{t}+\sqrt{2C}\,\dd{W_t},
\]
subject to the usual domain, trace and well posedness assumptions. A general abstract Wiener space need not place these terms in one Hilbert state space; the closed Dirichlet-form generator is the invariant formulation \cite{miyahara1981infinite,daPrato2014stochastic,fukushima2011dirichlet}.

\subsection{The infinite dimensional Stein operator}

Let $\Hspace_k$ be a Hilbert space of $H$ valued vector fields on $E$, continuously included in $L^2(\rho\gamma;H)$. An operator valued kernel $k(x,y)\in\mathcal L(H)$ gives such a space under suitable boundedness and measurability assumptions. Let
\[
	\Iop_{\rho}\colon\Hspace_k\to L^2(\rho\gamma;H)
\]
be the inclusion. The adjoint is defined weakly by
\begin{equation}\label{wiener-kernel-adjoint-eq}
	\langle\Iop_{\rho}^{*}u,v\rangle_{\Hspace_k}=\int_E\langle u(x),v(x)\rangle_H\rho(x)\,\dd{\gamma}(x).
\end{equation}
A regular cotangent $f$ determines
\begin{equation}\label{wiener-stein-velocity-eq}
	\Bop^{\Stein}_{\rho}f=\Iop_{\rho}^{*}D_Hf
\end{equation}
and
\begin{equation}\label{wiener-stein-onsager-eq}
	\Kop^{\Stein}_{\rho}f=\Aop_{\rho}\Iop_{\rho}^{*}D_Hf.
\end{equation}
The quotient and completion proof of Theorem \ref{stein-musical-thm} applies verbatim because it uses only Hilbert adjoints and the continuity pairing. The resulting entropy flow is an infinite dimensional Stein variational transport equation. Jia, Li and Meng develop this construction for Hilbert space inverse problems, including operator valued kernels and preconditioning adapted to the Gaussian reference structure \cite{jia2022stein}.

The Cameron--Martin restriction is essential. Translations outside $H$ are singular with respect to $\gamma$, so an admissible sample velocity must take values in directions along which the reference measure admits differential calculus. A covariance or precision operator often appears in the kernel and mobility to ensure that the velocity belongs to $H$. The finite dimensional formula $\grad f$ is therefore replaced by $D_Hf$, not by an arbitrary derivative in the Banach state space.

\subsection{Closability, quasi-regularity and processes}

Fix a reference measure $\mu=\rho\gamma$. Closability and quasi-regularity of
\[
	\mathcal E_{\mu}(f,h)=\int_E\langle D_Hf,D_Hh\rangle_H\,\dd{\mu}
\]
produce a state-space Markov process associated with that fixed symmetric form. Boundedness above and below of $\rho$ gives closability by comparison with the Gaussian form. More singular Gibbs measures require integration-by-parts estimates, logarithmic derivatives and capacity arguments; quasi-regularity then supplies a process outside an exceptional set \cite{fukushima2011dirichlet}. Infinite dimensional Fokker--Planck equations with log-concave reference measures can also be obtained as gradient flows and dual semigroup equations \cite{ambrosio2009fokker}.

This construction differs from a stochastic process whose law would follow the nonlinear density-space Stein equation $\partial_t\rho=-\Kop^{\Stein}_{\rho}\dd{\Ent}(\rho)$. The former is the Markov realisation of a fixed Dirichlet form; the latter is a nonlinear deterministic evolution on laws. The analytic hierarchy is therefore
\begin{enumerate}[label=(\roman*)]
	\item a differentiable density chart supplies regular tangent and cotangent variations;
	\item a closable Cameron--Martin derivative and divergence define the continuity pairing;
	\item weighted and kernel velocity norms define fibrewise Otto and Stein Hilbert pairs;
	\item a fixed quasi-regular Dirichlet form produces a state-space Markov process;
	\item a superposition or martingale problem theorem may connect a weak linear Fokker--Planck solution to path laws.
\end{enumerate}
No step in this hierarchy automatically constructs a pathwise realisation of a nonlinear Stein gradient flow.

\subsection{Wiener path space and marginal velocities}

Take $E=C_0([0,T];\R^d)$ with Wiener measure $\gamma$ and Cameron--Martin space
\[
	H=\left\{h:h(0)=0,\ h \text{ absolutely continuous},\ \int_0^T|\dot h_t|^2\,\dd{t}<\infty\right\}.
\]
Let $Q=\rho\gamma$ be a path law. A path space velocity is an $H$ valued field $h(\omega)$. Consider the infinitesimal perturbation $\omega\mapsto\omega+\varepsilon h(\omega)$ on cylindrical test functionals. Its image under evaluation at time $t$ is determined by a conditional expectation.

\begin{prop}\label{path-marginal-velocity-prop}
	Let $h\in L^2(Q;H)$ and suppose the infinitesimal path perturbation is admissible. Set $\mu_t=(e_t)_{\#}Q$ for each $t$ and
	\begin{equation}\label{conditional-marginal-velocity-eq}
		v_t(x)=\E_Q[h_t\mid\omega_t=x].
	\end{equation}
	Then $v_t\in L^2(\mu_t;\R^d)$ and
	\begin{equation}\label{path-marginal-continuity-eq}
		\partial_{\varepsilon}\big|_{\varepsilon=0}(e_t)_{\#}(\id+\varepsilon h)_{\#}Q=-\Div(\mu_tv_t)
	\end{equation}
	in distributions. Moreover,
	\[
		\|v_t\|_{L^2(\mu_t)}\leqslant\|h_t\|_{L^2(Q)}\leqslant\sqrt{t}\,\|h\|_{L^2(Q;H)}.
	\]
\end{prop}
\begin{proof}
	For $f\in C_c^{\infty}(\R^d)$, differentiation and conditioning give
	\[
		\partial_{\varepsilon}\big|_{\varepsilon=0}\int_Ef(\omega_t+\varepsilon h_t(\omega))\,\dd{Q}(\omega)=\int_{\R^d}\grad f(x)\cdot v_t(x)\,\dd{\mu_t}(x),
	\]
	which is \eqref{path-marginal-continuity-eq}. Conditional Jensen gives the first norm bound, and $|h_t|\leqslant\sqrt{t}(\int_0^t|\dot h_s|^2\,\dd{s})^{1/2}$ gives the second.
\end{proof}

The path-to-marginal tangent map has kernel
\[
	\left\{h:-\Div\bigl(\mu_t\E[h_t\mid\omega_t=\cdot]\bigr)=0\right\}.
\]
The stronger condition $\E[h_t\mid\omega_t]=0$ is sufficient but not necessary. The marginal tangent sees the conditional mean velocity only modulo $\mu_t$ divergence free fields, and Otto's horizontal representative is its $L^2(\mu_t)$ gradient projection. Evaluation loses information twice: conditioning removes fluctuations invisible at time $t$, and the continuity operator removes vertical marginal rearrangements.

The Feynman--Kac change of measure \eqref{tilted-path-law-eq} occurs on this path manifold before evaluation. A path space cotangent such as $\int_0^TV(\omega_t)\,\dd{t}$ changes the entire trajectory law through reweighting, and its evaluation marginals solve the normalised reaction transport equation. When a path space cotangent is converted by a Cameron--Martin Onsager map into an admissible velocity $h$, Proposition \ref{path-marginal-velocity-prop} shows that $\E_Q[h_t\mid\omega_t=x]$ is the induced marginal velocity. Reweighting and transport are consequently distinguished, whilst stochastic characteristics descend to Otto or Stein space through evaluation and conditional expectation.

\section{The particle-to-functional hierarchy and its equivalence relations}\label{synthesis-sec}

The preceding constructions may now be collected without suppressing the different levels at which an equivalence can be asserted.

\begin{theorem}\label{common-otto-stein-thm}
	Let $M$ be a connected closed Riemannian manifold, let $\Prob^{\infty}_+(M)$ be its Fr\'echet manifold of smooth positive probability densities and let $\Hspace_k\hookrightarrow C^1(M;TM)$ be a vector-valued reproducing kernel Hilbert space.
	\begin{enumerate}[label=(\roman*)]
		\item The continuity operator $\Aop_{\rho}v=-\Div(\rho v)$ is the differential of the pushforward action, is surjective from smooth vector fields onto $C^{\infty}_0(M)$ and identifies a smooth density tangent with a particle velocity modulo $\Ker\Aop_{\rho}$.
		\item The weighted form $\int_M\rho\langle\grad f,\grad h\rangle_g\,\dd{\mathrm{vol}}$ defines the Otto cotangent $\dot H^1_{\rho}$, tangent $\dot H^{-1}_{\rho}$ and Riesz isomorphism $\Kop^{\Wass}_{\rho}f=-\Div(\rho\grad f)$.
		\item The quotient range $T_{\rho,\Stein}=\Aop_{\rho}(\Hspace_k)$ with its minimum-$\Hspace_k$ norm is Hilbert. The kernel form $\langle\Iop_{\rho}^{*}\grad f,\Iop_{\rho}^{*}\grad h\rangle_{\Hspace_k}$ defines, after quotienting its full nullspace and completing, a Stein cotangent. The map $\Kop^{\Stein}_{\rho}f=-\Div(\rho\Iop_{\rho}^{*}\grad f)$ is an isometric isomorphism from this cotangent onto the full quotient tangent.
		\item The hierarchy associated with a differentiable functional is
		\[
			\mathcal F\longmapsto\dd{\mathcal F}(\rho)=[f]\longmapsto\grad f\longmapsto-\Bop^{\bullet}_{\rho}f\longmapsto-\Aop_{\rho}\Bop^{\bullet}_{\rho}f.
		\]
		The four arrows respectively pass from a functional to a density covector, from the covector to a sample space force, from the force to a selected particle velocity, and from the velocity to the law tangent.
		\item Let $\Mop_{\rho}=\mathsf{P}^{\mathrm{hor}}_{\rho}\Iop_{\rho}\Iop_{\rho}^{*}$ on the Otto horizontal space. Then $\Kop^{\Stein}_{\rho}f=\Aop_{\rho}\Mop_{\rho}\grad f$. Equality $\Mop_{\rho}\grad f=\grad f$ is equivalent to equality of the law tangents; equality $\Iop_{\rho}\Iop_{\rho}^{*}\grad f=\grad f$ is equivalent to equality of the particle fields; positive proportionality gives the same unparametrised law curve. Formula \eqref{prescribed-equivalence-eq} characterises equality under any prescribed differentiable observation, whilst $\Mop_{\rho}=\id$ on a closed force space characterises equality of the law geometries on all of its directions.
		\item The exact defect is given by \eqref{tangent-defect-eq}. Spectral bounds on $\Mop_{\rho}$ give the cotangent and tangent comparisons \eqref{cotangent-comparison-eq} and \eqref{tangent-metric-comparison-eq}. Every positive self-adjoint response on a finite dimensional smooth horizontal force space is realised by a finite rank vector-valued kernel.
		\item Relative entropy has regular cotangent $[\log(\rho/\pi)]$. Its Wasserstein and Stein descents agree precisely under Corollary \ref{entropy-equivalence-cor}. A regular minimal exponential family admits smooth Stein solutions which realise any prescribed Fisher mirror covector descent; exact membership in a chosen kernel space is the remaining range condition.
	\end{enumerate}
	All assertions concerning Hilbert pairs are fibrewise. Smooth variation of the Stein fibres, a connection or a globally well posed gradient flow requires additional uniform range and regularity hypotheses.
\end{theorem}
\begin{proof}
	Statement (i) is Lemma \ref{pushforward-differential-lem}, Lemma \ref{weighted-poisson-lem} and Proposition \ref{continuity-quotient-prop}. Statement (ii) is Theorem \ref{otto-riesz-thm} and Proposition \ref{minimal-kinetic-prop}. Statement (iii) is Theorem \ref{stein-musical-thm}. Statement (iv) follows from \eqref{transport-pairing-eq}, Proposition \ref{abstract-transport-prop} and \S\ref{functional-sec}. Statement (v) is Lemma \ref{effective-mobility-lem}, Theorem \ref{comparison-levels-thm}, Corollary \ref{full-fibre-equivalence-cor}, Proposition \ref{time-change-prop} and Proposition \ref{prescribed-observation-prop}. Statement (vi) is Theorem \ref{comparison-levels-thm}, Proposition \ref{metric-comparison-prop} and Proposition \ref{finite-rank-prescription-prop}. Statement (vii) is Lemma \ref{entropy-differential-lem}, Corollary \ref{entropy-equivalence-cor}, Theorem \ref{exponential-stein-covector-thm} and Remark \ref{smooth-stein-solutions-rem}.
\end{proof}

The objects in the hierarchy can be distinguished completely.
\begin{enumerate}[label=(\roman*)]
	\item A particle position is a point of the sample space, whilst a particle velocity is a tangent vector evaluated at that point.
	\item A particle field is a vector field on the sample space. Its empirical evaluation moves a finite configuration; its population image under $\Aop_{\rho}$ moves the law.
	\item A density tangent is a zero mass signed variation, smooth on the Fr\'echet core and distributional after completion. It remembers a particle field only modulo $\Ker\Aop_{\rho}$.
	\item A regular density covector is a scalar potential modulo constants. Its spatial differential is a force on the sample space.
	\item A differentiable functional assigns the covector. The chosen geometry assigns the response of particles and hence the tangent of the law.
	\item A generator value $\Lop^{*}\rho$, at a density in its domain, is a distributional density tangent and is smooth only on an invariant regular core. A tangent propagator $D\Phi_t(\rho)$ is an operator carrying one density tangent to another.
\end{enumerate}

The downward variational hierarchy is
\[
	\begin{tikzcd}[column sep=large,row sep=large]
		\mathcal F \arrow[r,"\dd"] & {[f]} \arrow[r,"\grad\text{ or }D_H"] & u_{\rho} \arrow[r,shift left=1.3ex,"\Bop^{\Wass}_{\rho}"] \arrow[r,shift right=1.3ex,"\Bop^{\Stein}_{\rho}"'] & v_{\rho}^{\bullet} \arrow[r,"\Aop_{\rho}"] \arrow[d,"\mathrm{evaluation}"] & \dot\rho^{\bullet} \\
		& & & (v_{\rho}^{\bullet}(x_1),\ldots,v_{\rho}^{\bullet}(x_N)) \arrow[r] & \partial_t\mu^N.
	\end{tikzcd}
\]
The upward descriptive hierarchy is
\[
	\begin{tikzcd}[column sep=large]
		(x_1,\ldots,x_N) \arrow[r,"\mathrm{empirical}"] & \mu^N \arrow[r,"\mathrm{limit\ or\ law}"] & \rho \arrow[r,"\mathcal O\text{ or }\mathcal F"] & \text{observables and functionals}.
	\end{tikzcd}
\]
A comparison prescription chooses where in these diagrams equality is demanded. Equality before evaluation identifies particles. Equality after $\Aop_{\rho}$ identifies laws. Equality after $D\mathcal O$ identifies only the selected observables. Positive colinearity identifies unparametrised law trajectories. Equivalence of the quadratic forms identifies costs and dissipation scales. Equality of large deviations actions additionally identifies the microscopic fluctuation structure. These relations form a hierarchy of their own; each is appropriate to a different question, and none should be substituted for another without stating what information has been retained.

\appendix

\section{Weighted elliptic estimates and smooth dependence}\label{elliptic-appendix-sec}

Let $s>d/2+1$ and define the open subset
\[
	\mathcal U^s_+(M)=\left\{\rho\in H^s(M):\int_M\rho\,\dd{\mathrm{vol}}=1,\ \inf_M\rho>0\right\}.
\]
Fix $m>0$, $R<\infty$ and a subset $\mathcal B\subset\mathcal U^s_+(M)$ on which $\rho\geqslant m$ and $\|\rho\|_{H^s}\leqslant R$. Define
\[
	\Kop^{\Wass}_{\rho}\phi=-\Div(\rho\grad\phi).
\]

\begin{prop}\label{sobolev-elliptic-isomorphism-prop}
	For every integer $r$ with $0\leqslant r\leqslant s-1$, the operator
	\[
		\Kop^{\Wass}_{\rho}\colon H^{r+1}(M)/\R\to H^{r-1}_0(M)
	\]
	is an isomorphism, with the weak interpretation $H^1/\R\to H^{-1}_0$ when $r=0$. The set $\mathcal B$ admits the uniform estimate
	\begin{equation}\label{weighted-elliptic-estimate-eq}
		\|\phi\|_{H^{r+1}/\R}\leqslant C(m,R,M,g,r,s)\|\Kop^{\Wass}_{\rho}\phi\|_{H^{r-1}}.
	\end{equation}
\end{prop}
\begin{proof}
	The weak isomorphism follows from Lax--Milgram and Poincar\'e's inequality. Since $s>d/2+1$, Sobolev multiplication controls the coefficient terms, and elliptic regularity for uniformly elliptic operators with $H^s$ coefficients upgrades the weak solution through the indicated scale \cite{taylor1996pde2}. The kernel consists of constants by the energy identity and the range has zero integral by divergence. Uniform lower ellipticity and the $H^s$ bound give the estimate uniformly on $\mathcal B$.
\end{proof}

\begin{prop}\label{inverse-smooth-dependence-prop}
	For each admissible $r$, the map
	\[
		\rho\longmapsto(\Kop^{\Wass}_{\rho})^{-1}
	\]
	is smooth from $\mathcal U^s_+(M)$ into $\mathcal L(H^{r-1}_0(M),H^{r+1}(M)/\R)$. Its differential is
	\begin{equation}\label{inverse-derivative-eq}
		D(\Kop^{\Wass}_{\rho})^{-1}[\eta]=-(\Kop^{\Wass}_{\rho})^{-1}\bigl(D\Kop^{\Wass}_{\rho}[\eta]\bigr)(\Kop^{\Wass}_{\rho})^{-1},
	\end{equation}
	where
	\begin{equation}\label{operator-density-derivative-eq}
		D\Kop^{\Wass}_{\rho}[\eta]\phi=-\Div(\eta\grad\phi).
	\end{equation}
\end{prop}
\begin{proof}
	The coefficient map $\rho\mapsto\Kop^{\Wass}_{\rho}$ is affine and continuous into the stated operator space. Proposition \ref{sobolev-elliptic-isomorphism-prop} places its image in the open set of bounded invertible operators. Smoothness of inversion and differentiation of $(\Kop^{\Wass}_{\rho})^{-1}\Kop^{\Wass}_{\rho}=\id$ give the formulae.
\end{proof}

The fixed-bound set $\{\rho:m\leqslant\rho\leqslant M\}$ is useful for uniform estimates but is closed in the Sobolev affine hyperplane. Smoothness is asserted on $\mathcal U^s_+(M)$, whilst quantitative bounds are taken on subsets such as $\mathcal B$. Approaching a density which vanishes destroys uniform ellipticity; the metric completion may contain such measures even though the smooth bundle description does not extend to them.

\section{Differential calculus for the kernel operator}\label{kernel-calculus-sec}

Let $M=\R^d$ for notation. Fix an integer $r\geqslant1$ and assume that $k$ and its derivatives in the second variable through order $r$ are bounded. If $u\in C_b^r(\R^d;\R^d)$ and $\rho\in L^1(\R^d)$, then
\[
	(\Iop_{\rho}^{*}u)(y)=\int_{\R^d}k(x,y)u(x)\rho(x)\,\dd{x}
\]
defines an element of $C_b^r$, and $\rho\mapsto\Iop_{\rho}^{*}u$ is bounded linear from $L^1$ to $C_b^r$. The same statement holds for signed variations $\eta\in L^1$. When $u=u_{\rho}$ or $f_{\rho}=\log(\rho/\pi)$ varies, the chosen density topology must additionally control the derivatives of $u_{\rho}$, a positive lower bound for $\rho$ on the relevant region, the term $\grad(\eta/\rho)$ and the integrability required by the kernel operator.

\begin{lemma}\label{kernel-density-derivative-lem}
	Fix $u$. The map $\rho\mapsto\Iop_{\rho}^{*}u$ is affine and satisfies
	\begin{equation}\label{kernel-density-derivative-eq}
		D_{\rho}(\Iop_{\rho}^{*}u)[\eta]=\Iop_{\eta}^{*}u.
	\end{equation}
	If $u_{\rho}$ is differentiable in $\rho$, then
	\[
		D_{\rho}(\Iop_{\rho}^{*}u_{\rho})[\eta]=\Iop_{\eta}^{*}u_{\rho}+\Iop_{\rho}^{*}D_{\rho}u_{\rho}[\eta].
	\]
\end{lemma}
\begin{proof}
	Both identities follow from bounded linearity in the $L^1$ density variable and, for $u_{\rho}$, the Banach-space product rule under the stated $C_b^r$ control.
\end{proof}

Fix $f$. The Stein Onsager operator is
\[
	\Kop_{\rho}^{\Stein}f=-\Div\left(\rho\Iop_{\rho}^{*}\grad f\right).
\]
Its density derivative is
\begin{equation}\label{stein-onsager-density-derivative-eq}
	D_{\rho}\Kop_{\rho}^{\Stein}f[\eta]=-\Div\left(\eta\Iop_{\rho}^{*}\grad f+\rho\Iop_{\eta}^{*}\grad f\right).
\end{equation}
If $f=f_{\rho}$ also varies, a third term appears:
\[
	-\Div\left(\rho\Iop_{\rho}^{*}\grad D_{\rho}f_{\rho}[\eta]\right).
\]
Taking $f_{\rho}=\log(\rho/\pi)$ yields \eqref{stein-linearisation-eq}. These formulae are the regular core from which a Stein connection would have to be constructed. Extending them to the Hilbert completions requires uniform estimates on the varying range and nullspace.

The derivative of the effective mobility additionally contains the derivative of the weighted horizontal projection. A uniformly positive Sobolev chart gives the formula
\[
	\mathsf{P}^{\mathrm{hor}}_{\rho}v=\grad(\Kop^{\Wass}_{\rho})^{-1}\Aop_{\rho}v,
\]
so \S\ref{elliptic-appendix-sec} and the preceding kernel derivatives give a regular-core formula for $D\Mop_{\rho}$. Extending this calculation to a smooth field of completed Stein fibres requires uniform control of the varying kernel range and nullspace.

\section{Fisher geometry of exponential families}\label{exponential-appendix-sec}

The Fisher calculations used in \S\ref{statistical-submanifolds-sec} are recorded here for reference. Differentiating \eqref{exponential-family-eq} gives
\begin{equation}\label{exponential-score-eq}
	\partial_i\log q_{\eta}=T_i-\partial_iA(\eta)=T_i-m_i(\eta).
\end{equation}
Therefore
\[
	G^{\Fish}_{ij}(\eta)=\E_{q_{\eta}}[(T_i-m_i)(T_j-m_j)]=C_{ij}(\eta).
\]
Differentiating $m_i=\partial_iA$ gives $\partial_jm_i=C_{ij}$. The dual coordinate change consequently has Jacobian $D_{\eta}m=C$ and $D_m\eta=C^{-1}$. The identity $\eta=\grad_mA^{*}$ gives $\grad_m^2A^{*}=C^{-1}$.

Let $\beta$ be a covector in expectation coordinates. The Fisher sharp $u$ satisfies
\[
	\beta^{\mathsf T}\dot m=u^{\mathsf T}C^{-1}\dot m
\]
for every $\dot m$, hence $u=C\beta$. Negative Fisher gradient flow gives $\dot m=-C\beta$, as used in \eqref{fisher-mirror-moment-eq}. Natural coordinates convert this to $\dot\eta=C^{-1}\dot m=-\beta$.

\section{Pointwise verification of Gaussian affine transport}\label{gaussian-appendix-sec}

Let $y=x-\mu$ and
\[
	\log q_{\mu,\Sigma}(x)=-\frac d2\log(2\pi)-\frac12\log\det\Sigma-\frac12y^{\mathsf T}\Sigma^{-1}y.
\]
A tangent $(u,U)$ gives, after differentiation,
\begin{align}
	\partial_t\log q
	&=-\frac12\tr(\Sigma^{-1}U)+u^{\mathsf T}\Sigma^{-1}y+\frac12y^{\mathsf T}\Sigma^{-1}U\Sigma^{-1}y. \label{gaussian-log-tangent-eq}
\end{align}
Let $v(x)=u+By$. The Gaussian Stein operator is
\begin{align}
	\Tspace_qv
	&=\tr B-\bigl(\Sigma^{-1}y\bigr)^{\mathsf T}(u+By) \notag\\
	&=\tr B-u^{\mathsf T}\Sigma^{-1}y-y^{\mathsf T}\Sigma^{-1}By. \label{gaussian-affine-stein-eq}
\end{align}
If $B\Sigma+\Sigma B^{\mathsf T}=U$, then
\[
	\tr B=\frac12\tr(\Sigma^{-1}U)
\]
and
\[
	y^{\mathsf T}\Sigma^{-1}By=\frac12y^{\mathsf T}\Sigma^{-1}U\Sigma^{-1}y,
\]
because only the symmetric part of $\Sigma^{-1}B$ contributes to the quadratic form. Equations \eqref{gaussian-log-tangent-eq} and \eqref{gaussian-affine-stein-eq} give
\[
	\partial_t\log q=-\Tspace_qv.
\]
Multiplication by $q$ proves pointwise that $\partial_tq=-\Div(qv)$. This verifies the affine lift at density level, beyond the moment calculation of Lemma \ref{affine-moment-lem}.

\section{Normalisations of diffusion and action functionals}\label{normalisation-appendix-sec}

Let the Brownian equation be $\dd{X_t}=\sqrt{2\kappa}\,\dd{W_t}$, whose generator is $\kappa\Delta$. The entropy gradient flow is $\partial_t\rho=\kappa\Delta\rho$ when the Wasserstein Onsager operator is multiplied by $\kappa$. The short time heat kernel has exponential cost $|x-y|^2/(4\kappa h)$. The one step large deviations expansion therefore contains
\[
	\frac{1}{4\kappa h}W_2^2(\rho_0,\rho)+\frac12\Ent(\rho)-\frac12\Ent(\rho_0).
\]
Multiplication by $2\kappa$ gives the minimizing movement functional
\[
	\frac{1}{2h}W_2^2(\rho_0,\rho)+\kappa\Ent(\rho)-\kappa\Ent(\rho_0).
\]
Likewise, the dynamic action is
\[
	\frac{1}{4\kappa}\int_0^T\|\partial_t\rho_t-\kappa\Delta\rho_t\|_{\dot H^{-1}_{\rho_t}}^2\,\dd{t}.
\]
These factors account for the variants found in the large deviations and gradient flow literature. The body of the paper uses generator $\Delta$, corresponding to $\kappa=1$.

\bibliographystyle{amsalpha}
\bibliography{main}

@article{adams2011large,
  author = {Adams, Stefan and Dirr, Nicolas and Peletier, Mark A. and Zimmer, Johannes},
  title = {From a Large-Deviations Principle to the Wasserstein Gradient Flow: A New Micro--Macro Passage},
  journal = {Communications in Mathematical Physics},
  volume = {307},
  number = {3},
  pages = {791--815},
  year = {2011},
  doi = {10.1007/s00220-011-1328-y}
}

@book{ambrosio2008gradient,
  author = {Ambrosio, Luigi and Gigli, Nicola and Savar\'e, Giuseppe},
  title = {Gradient Flows in Metric Spaces and in the Space of Probability Measures},
  edition = {Second},
  publisher = {Birkh\"auser},
  address = {Basel},
  year = {2008}
}

@article{ambrosio2009fokker,
  author = {Ambrosio, Luigi and Savar\'e, Giuseppe and Zambotti, Lorenzo},
  title = {Existence and Stability for {Fokker--Planck} Equations with Log-Concave Reference Measure},
  journal = {Probability Theory and Related Fields},
  volume = {145},
  number = {3--4},
  pages = {517--564},
  year = {2009},
  doi = {10.1007/s00440-008-0177-8}
}

@article{aronszajn1950theory,
  author = {Aronszajn, Nachman},
  title = {Theory of Reproducing Kernels},
  journal = {Transactions of the American Mathematical Society},
  volume = {68},
  number = {3},
  pages = {337--404},
  year = {1950},
  doi = {10.2307/1990404}
}

@book{ay2017information,
  author = {Ay, Nihat and Jost, J\"urgen and L\^e, H\^ong V\^an and Schwachh\"ofer, Lorenz},
  title = {Information Geometry},
  publisher = {Springer},
  address = {Cham},
  year = {2017},
  doi = {10.1007/978-3-319-56478-4}
}

@article{barbour1988stein,
  author = {Barbour, Andrew D.},
  title = {Stein's Method and Poisson Process Convergence},
  journal = {Journal of Applied Probability},
  volume = {25},
  pages = {175--184},
  year = {1988},
  doi = {10.2307/3214155}
}

@book{bogachev2015fokker,
  author = {Bogachev, Vladimir I. and Krylov, Nicolai V. and R\"ockner, Michael and Shaposhnikov, Stanislav V.},
  title = {Fokker--Planck--Kolmogorov Equations},
  publisher = {American Mathematical Society},
  address = {Providence},
  year = {2015},
  doi = {10.1090/surv/207}
}

@book{cardaliaguet2019master,
  author = {Cardaliaguet, Pierre and Delarue, Fran\c{c}ois and Lasry, Jean-Michel and Lions, Pierre-Louis},
  title = {The Master Equation and the Convergence Problem in Mean Field Games},
  series = {Annals of Mathematics Studies},
  volume = {201},
  publisher = {Princeton University Press},
  address = {Princeton},
  year = {2019}
}

@book{carmona2018probabilistic1,
  author = {Carmona, Ren\'e and Delarue, Fran\c{c}ois},
  title = {Probabilistic Theory of Mean Field Games with Applications I},
  publisher = {Springer},
  address = {Cham},
  year = {2018},
  doi = {10.1007/978-3-319-58920-6}
}

@book{carmona2018probabilistic2,
  author = {Carmona, Ren\'e and Delarue, Fran\c{c}ois},
  title = {Probabilistic Theory of Mean Field Games with Applications II},
  publisher = {Springer},
  address = {Cham},
  year = {2018},
  doi = {10.1007/978-3-319-56436-4}
}

@article{chizat2018interpolating,
  author = {Chizat, L\'ena\"ic and Peyr\'e, Gabriel and Schmitzer, Bernhard and Vialard, Fran\c{c}ois-Xavier},
  title = {An Interpolating Distance between Optimal Transport and {Fisher--Rao} Metrics},
  journal = {Foundations of Computational Mathematics},
  volume = {18},
  number = {1},
  pages = {1--44},
  year = {2018},
  doi = {10.1007/s10208-016-9331-y}
}

@article{duong2013wasserstein,
  author = {Duong, Manh Hong and Laschos, Vaios and Renger, D. R. Michiel},
  title = {Wasserstein Gradient Flows from Large Deviations of Many-Particle Limits},
  journal = {ESAIM: Control, Optimisation and Calculus of Variations},
  volume = {19},
  number = {4},
  pages = {1166--1188},
  year = {2013},
  doi = {10.1051/cocv/2013049}
}

@article{erbar2015multiple,
  author = {Erbar, Matthias and Maas, Jan and Renger, D. R. Michiel},
  title = {From Large Deviations to Wasserstein Gradient Flows in Multiple Dimensions},
  journal = {Electronic Communications in Probability},
  volume = {20},
  number = {89},
  pages = {1--12},
  year = {2015},
  doi = {10.1214/ECP.v20-4315}
}

@book{delmoral2004feynman,
  author = {Del Moral, Pierre},
  title = {Feynman--Kac Formulae: Genealogical and Interacting Particle Systems with Applications},
  publisher = {Springer},
  address = {New York},
  year = {2004},
  doi = {10.1007/978-1-4684-9393-1}
}

@article{duncan2023geometry,
  author = {Duncan, Andrew and N\"usken, Nikolas and Szpruch, Lukasz},
  title = {On the Geometry of {Stein} Variational Gradient Descent},
  journal = {Journal of Machine Learning Research},
  volume = {24},
  number = {56},
  pages = {1--39},
  year = {2023}
}

@article{ebin1970groups,
  author = {Ebin, David G. and Marsden, Jerrold},
  title = {Groups of Diffeomorphisms and the Motion of an Incompressible Fluid},
  journal = {Annals of Mathematics},
  volume = {92},
  number = {1},
  pages = {102--163},
  year = {1970},
  doi = {10.2307/1970699}
}

@article{figalli2008existence,
  author = {Figalli, Alessio},
  title = {Existence and Uniqueness of Martingale Solutions for {SDEs} with Rough or Degenerate Coefficients},
  journal = {Journal of Functional Analysis},
  volume = {254},
  number = {1},
  pages = {109--153},
  year = {2008},
  doi = {10.1016/j.jfa.2007.09.020}
}

@article{fisher1922mathematical,
  author = {Fisher, Ronald A.},
  title = {On the Mathematical Foundations of Theoretical Statistics},
  journal = {Philosophical Transactions of the Royal Society of London. Series A},
  volume = {222},
  pages = {309--368},
  year = {1922},
  doi = {10.1098/rsta.1922.0009}
}

@article{fokker1914mittlere,
  author = {Fokker, Adriaan D.},
  title = {Die mittlere Energie rotierender elektrischer Dipole im Strahlungsfeld},
  journal = {Annalen der Physik},
  volume = {348},
  number = {5},
  pages = {810--820},
  year = {1914},
  doi = {10.1002/andp.19143480507}
}

@book{fukushima2011dirichlet,
  author = {Fukushima, Masatoshi and Oshima, Yoichi and Takeda, Masayoshi},
  title = {Dirichlet Forms and Symmetric Markov Processes},
  edition = {Second},
  publisher = {de Gruyter},
  address = {Berlin},
  year = {2011},
  doi = {10.1515/9783110218091}
}

@article{gibilisco1998connections,
  author = {Gibilisco, Paolo and Pistone, Giovanni},
  title = {Connections on Non-Parametric Statistical Manifolds by {Orlicz} Space Geometry},
  journal = {Infinite Dimensional Analysis, Quantum Probability and Related Topics},
  volume = {1},
  number = {2},
  pages = {325--347},
  year = {1998},
  doi = {10.1142/S021902579800017X}
}

@article{gorham2019measuring,
  author = {Gorham, Jackson and Duncan, Andrew B. and Vollmer, Sebastian J. and Mackey, Lester},
  title = {Measuring Sample Quality with Diffusions},
  journal = {Annals of Applied Probability},
  volume = {29},
  number = {5},
  pages = {2884--2928},
  year = {2019},
  doi = {10.1214/19-AAP1467}
}

@article{haussmann1986time,
  author = {Haussmann, Ulrich G. and Pardoux, Etienne},
  title = {Time Reversal of Diffusions},
  journal = {Annals of Probability},
  volume = {14},
  number = {4},
  pages = {1188--1205},
  year = {1986},
  doi = {10.1214/aop/1176992362}
}

@article{jia2022stein,
  author = {Jia, Jun and Li, Peiyu and Meng, Xiaohui},
  title = {Stein Variational Gradient Descent on Infinite-Dimensional Space and Applications to Statistical Inverse Problems},
  journal = {SIAM Journal on Numerical Analysis},
  volume = {60},
  number = {4},
  pages = {2225--2252},
  year = {2022},
  doi = {10.1137/21M1440773}
}

@article{jordan1998variational,
  author = {Jordan, Richard and Kinderlehrer, David and Otto, Felix},
  title = {The Variational Formulation of the {Fokker--Planck} Equation},
  journal = {SIAM Journal on Mathematical Analysis},
  volume = {29},
  number = {1},
  pages = {1--17},
  year = {1998},
  doi = {10.1137/S0036141096303359}
}

@article{kac1949distributions,
  author = {Kac, Mark},
  title = {On Distributions of Certain {Wiener} Functionals},
  journal = {Transactions of the American Mathematical Society},
  volume = {65},
  number = {1},
  pages = {1--13},
  year = {1949},
  doi = {10.2307/1990512}
}

@book{kunita1990stochastic,
  author = {Kunita, Hiroshi},
  title = {Stochastic Flows and Stochastic Differential Equations},
  publisher = {Cambridge University Press},
  address = {Cambridge},
  year = {1990},
  doi = {10.1017/CBO9780511624331}
}

@article{khesin2013geometry,
  author = {Khesin, Boris and Lenells, Jonatan and Misiolek, Gerard and Preston, Stephen C.},
  title = {Curvatures of Sobolev Metrics on Diffeomorphism Groups},
  journal = {Pure and Applied Mathematics Quarterly},
  volume = {9},
  number = {2},
  pages = {291--332},
  year = {2013},
  doi = {10.4310/PAMQ.2013.v9.n2.a4}
}

@article{le2024diffusion,
  author = {Le, H\'oang-Long and Lewis, Andrew D. and Bharath, Karthik and Fallaize, Christopher J.},
  title = {A Diffusion Approach to {Stein}'s Method on {Riemannian} Manifolds},
  journal = {Bernoulli},
  volume = {30},
  number = {2},
  pages = {1542--1568},
  year = {2024},
  doi = {10.3150/23-BEJ1646}
}

@article{liero2018optimal,
  author = {Liero, Matthias and Mielke, Alexander and Savar\'e, Giuseppe},
  title = {Optimal Entropy-Transport Problems and a New {Hellinger--Kantorovich} Distance between Positive Measures},
  journal = {Inventiones Mathematicae},
  volume = {211},
  number = {3},
  pages = {969--1117},
  year = {2018},
  doi = {10.1007/s00222-017-0759-8}
}

@inproceedings{liu2017stein,
  author = {Liu, Qiang},
  title = {Stein Variational Gradient Descent as Gradient Flow},
  booktitle = {Advances in Neural Information Processing Systems},
  volume = {30},
  pages = {3115--3123},
  year = {2017}
}

@inproceedings{liu2016stein,
  author = {Liu, Qiang and Wang, Dilin},
  title = {Stein Variational Gradient Descent: A General Purpose Bayesian Inference Algorithm},
  booktitle = {Advances in Neural Information Processing Systems},
  volume = {29},
  pages = {2378--2386},
  year = {2016}
}

@article{lott2008geometric,
  author = {Lott, John},
  title = {Some Geometric Calculations on Wasserstein Space},
  journal = {Communications in Mathematical Physics},
  volume = {277},
  number = {2},
  pages = {423--437},
  year = {2008},
  doi = {10.1007/s00220-007-0367-3}
}

@article{lu2019scaling,
  author = {Lu, Jianfeng and Lu, Yulong and Nolen, James},
  title = {Scaling Limit of the {Stein} Variational Gradient Descent: The Mean Field Regime},
  journal = {SIAM Journal on Mathematical Analysis},
  volume = {51},
  number = {2},
  pages = {648--671},
  year = {2019},
  doi = {10.1137/18M1187611}
}

@article{miyahara1981infinite,
  author = {Miyahara, Yoshio},
  title = {Infinite Dimensional Langevin Equation and Fokker--Planck Equation},
  journal = {Nagoya Mathematical Journal},
  volume = {81},
  pages = {177--223},
  year = {1981},
  doi = {10.1017/S0027763000019420}
}

@article{mielke2014variational,
  author = {Mielke, Alexander and Peletier, Mark A. and Renger, D. R. Michiel},
  title = {On the Relation between Gradient Flows and the Large-Deviation Principle, with Applications to {Markov} Chains and Diffusion},
  journal = {Potential Analysis},
  volume = {41},
  number = {4},
  pages = {1293--1327},
  year = {2014},
  doi = {10.1007/s11118-014-9418-5}
}

@article{moser1965volume,
  author = {Moser, J\"urgen},
  title = {On the Volume Elements on a Manifold},
  journal = {Transactions of the American Mathematical Society},
  volume = {120},
  number = {2},
  pages = {286--294},
  year = {1965},
  doi = {10.2307/1994022}
}

@book{nelson1967dynamical,
  author = {Nelson, Edward},
  title = {Dynamical Theories of Brownian Motion},
  publisher = {Princeton University Press},
  address = {Princeton},
  year = {1967},
  note = {Second edition, 2001}
}

@article{newton2012hilbert,
  author = {Newton, Nigel J.},
  title = {An Infinite-Dimensional Statistical Manifold Modelled on Hilbert Space},
  journal = {Journal of Functional Analysis},
  volume = {263},
  number = {6},
  pages = {1661--1681},
  year = {2012},
  doi = {10.1016/j.jfa.2012.05.023}
}

@article{newton2016balanced,
  author = {Newton, Nigel J.},
  title = {Infinite-Dimensional Statistical Manifolds Based on a Balanced Chart},
  journal = {Bernoulli},
  volume = {22},
  number = {2},
  pages = {711--731},
  year = {2016},
  doi = {10.3150/14-BEJ673}
}

@article{newton2018differentiable,
  author = {Newton, Nigel J.},
  title = {Manifolds of Differentiable Densities},
  journal = {ESAIM: Probability and Statistics},
  volume = {22},
  pages = {19--34},
  year = {2018},
  doi = {10.1051/ps/2018003}
}

@article{newton2019sobolev,
  author = {Newton, Nigel J.},
  title = {A Class of Non-Parametric Statistical Manifolds Modelled on Sobolev Space},
  journal = {Information Geometry},
  volume = {2},
  pages = {283--312},
  year = {2019},
  doi = {10.1007/s41884-019-00024-z}
}

@article{nusken2023large,
  author = {N\"usken, Nikolas and Renger, D. R. Michiel},
  title = {Stein Variational Gradient Descent: Many-Particle and Long-Time Asymptotics},
  journal = {Foundations of Data Science},
  volume = {5},
  number = {3},
  pages = {286--320},
  year = {2023},
  doi = {10.3934/fods.2022023}
}

@article{otto2001geometry,
  author = {Otto, Felix},
  title = {The Geometry of Dissipative Evolution Equations: The Porous Medium Equation},
  journal = {Communications in Partial Differential Equations},
  volume = {26},
  number = {1--2},
  pages = {101--174},
  year = {2001},
  doi = {10.1081/PDE-100002243}
}

@article{pistone1995infinite,
  author = {Pistone, Giovanni and Sempi, Carlo},
  title = {An Infinite-Dimensional Geometric Structure on the Space of All the Probability Measures Equivalent to a Given One},
  journal = {Annals of Statistics},
  volume = {23},
  number = {5},
  pages = {1543--1561},
  year = {1995},
  doi = {10.1214/aos/1176324311}
}

@article{planck1917satz,
  author = {Planck, Max},
  title = {\"Uber einen Satz der statistischen Dynamik und seine Erweiterung in der Quantentheorie},
  journal = {Sitzungsberichte der K\"oniglich Preussischen Akademie der Wissenschaften},
  pages = {324--341},
  year = {1917}
}

@book{rockafellar1970convex,
  author = {Rockafellar, R. Tyrrell},
  title = {Convex Analysis},
  publisher = {Princeton University Press},
  address = {Princeton},
  year = {1970}
}

@article{trevisan2016well,
  author = {Trevisan, Dario},
  title = {Well-Posedness of Multidimensional Diffusion Processes with Weakly Differentiable Coefficients},
  journal = {Electronic Journal of Probability},
  volume = {21},
  pages = {1--41},
  year = {2016},
  doi = {10.1214/16-EJP445}
}

@book{villani2003topics,
  author = {Villani, C\'edric},
  title = {Topics in Optimal Transportation},
  publisher = {American Mathematical Society},
  address = {Providence},
  year = {2003}
}

@book{villani2009optimal,
  author = {Villani, C\'edric},
  title = {Optimal Transport: Old and New},
  publisher = {Springer},
  address = {Berlin},
  year = {2009},
  doi = {10.1007/978-3-540-71050-9}
}

@article{wu2020lions,
  author = {Wu, Cong and Zhang, Jianfeng},
  title = {An Elementary Proof for the Structure of {Lions} Derivative},
  journal = {Electronic Communications in Probability},
  volume = {25},
  pages = {1--11},
  year = {2020},
  doi = {10.1214/20-ECP330}
}

@book{daPrato2014stochastic,
  author = {Da Prato, Giuseppe and Zabczyk, Jerzy},
  title = {Stochastic Equations in Infinite Dimensions},
  edition = {Second},
  publisher = {Cambridge University Press},
  address = {Cambridge},
  year = {2014},
  doi = {10.1017/CBO9781107295513}
}

@article{benamou2000computational,
  author = {Benamou, Jean-David and Brenier, Yann},
  title = {A Computational Fluid Mechanics Solution to the {Monge--Kantorovich} Mass Transfer Problem},
  journal = {Numerische Mathematik},
  volume = {84},
  number = {3},
  pages = {375--393},
  year = {2000},
  doi = {10.1007/s002110050002}
}

@article{onsager1931reciprocal1,
  author = {Onsager, Lars},
  title = {Reciprocal Relations in Irreversible Processes. {I}},
  journal = {Physical Review},
  volume = {37},
  number = {4},
  pages = {405--426},
  year = {1931},
  doi = {10.1103/PhysRev.37.405}
}

@article{onsager1931reciprocal2,
  author = {Onsager, Lars},
  title = {Reciprocal Relations in Irreversible Processes. {II}},
  journal = {Physical Review},
  volume = {38},
  number = {12},
  pages = {2265--2279},
  year = {1931},
  doi = {10.1103/PhysRev.38.2265}
}

@article{rao1945information,
  author = {Rao, C. R.},
  title = {Information and the Accuracy Attainable in the Estimation of Statistical Parameters},
  journal = {Bulletin of the Calcutta Mathematical Society},
  volume = {37},
  pages = {81--91},
  year = {1945}
}

@book{cencov1982statistical,
  author = {Cencov, N. N.},
  title = {Statistical Decision Rules and Optimal Inference},
  publisher = {American Mathematical Society},
  address = {Providence},
  year = {1982}
}

@inproceedings{stein1972bound,
  author = {Stein, Charles},
  title = {A Bound for the Error in the Normal Approximation to the Distribution of a Sum of Dependent Random Variables},
  booktitle = {Proceedings of the Sixth Berkeley Symposium on Mathematical Statistics and Probability},
  volume = {2},
  pages = {583--602},
  year = {1972}
}

@book{stein1986approximate,
  author = {Stein, Charles},
  title = {Approximate Computation of Expectations},
  publisher = {Institute of Mathematical Statistics},
  address = {Hayward, California},
  series = {IMS Lecture Notes--Monograph Series},
  volume = {7},
  year = {1986}
}

@inproceedings{liu2016kernelized,
  author = {Liu, Qiang and Lee, Jason D. and Jordan, Michael I.},
  title = {A Kernelized {Stein} Discrepancy for Goodness-of-Fit Tests},
  booktitle = {Proceedings of the 33rd International Conference on Machine Learning},
  series = {Proceedings of Machine Learning Research},
  volume = {48},
  pages = {276--284},
  year = {2016}
}

@inproceedings{gorham2017measuring,
  author = {Gorham, Jackson and Mackey, Lester},
  title = {Measuring Sample Quality with Kernels},
  booktitle = {Proceedings of the 34th International Conference on Machine Learning},
  series = {Proceedings of Machine Learning Research},
  volume = {70},
  pages = {1292--1301},
  year = {2017}
}

@article{micchelli2005learning,
  author = {Micchelli, Charles A. and Pontil, Massimiliano},
  title = {On Learning Vector-Valued Functions},
  journal = {Neural Computation},
  volume = {17},
  number = {1},
  pages = {177--204},
  year = {2005},
  doi = {10.1162/0899766052530802}
}

@article{carmeli2010vector,
  author = {Carmeli, Claudio and De Vito, Ernesto and Toigo, Alessandro and Umanit\`a, Veronica},
  title = {Vector Valued Reproducing Kernel {Hilbert} Spaces and Universality},
  journal = {Analysis and Applications},
  volume = {8},
  number = {1},
  pages = {19--61},
  year = {2010},
  doi = {10.1142/S0219530510001503}
}

@article{hamilton1982inverse,
  author = {Hamilton, Richard S.},
  title = {The Inverse Function Theorem of {Nash} and {Moser}},
  journal = {Bulletin of the American Mathematical Society},
  volume = {7},
  number = {1},
  pages = {65--222},
  year = {1982},
  doi = {10.1090/S0273-0979-1982-15004-2}
}

@book{kriegl1997convenient,
  author = {Kriegl, Andreas and Michor, Peter W.},
  title = {The Convenient Setting of Global Analysis},
  publisher = {American Mathematical Society},
  address = {Providence},
  series = {Mathematical Surveys and Monographs},
  volume = {53},
  year = {1997}
}

@book{dynkin1965markov,
  author = {Dynkin, E. B.},
  title = {{Markov} Processes},
  publisher = {Springer},
  address = {Berlin},
  year = {1965}
}

@article{hormander1967hypoelliptic,
  author = {H\"ormander, Lars},
  title = {Hypoelliptic Second Order Differential Equations},
  journal = {Acta Mathematica},
  volume = {119},
  pages = {147--171},
  year = {1967},
  doi = {10.1007/BF02392081}
}

@incollection{bakry1985diffusions,
  author = {Bakry, Dominique and \`Emery, Michel},
  title = {Diffusions Hypercontractives},
  booktitle = {S\'eminaire de Probabilit\'es XIX, 1983/84},
  series = {Lecture Notes in Mathematics},
  volume = {1123},
  pages = {177--206},
  publisher = {Springer},
  address = {Berlin},
  year = {1985}
}

@inproceedings{korba2020nonasymptotic,
  author = {Korba, Anna and Salim, Adil and Arbel, Michael and Luise, Giulia and Gretton, Arthur},
  title = {A Non-Asymptotic Analysis for {Stein} Variational Gradient Descent},
  booktitle = {Advances in Neural Information Processing Systems},
  volume = {33},
  pages = {4672--4682},
  year = {2020}
}

@inproceedings{salim2022convergence,
  author = {Salim, Adil and Sun, Lukang and Richt\'arik, Peter},
  title = {A Convergence Theory for {SVGD} in the Population Limit under Talagrand's Inequality {T1}},
  booktitle = {Proceedings of the 39th International Conference on Machine Learning},
  series = {Proceedings of Machine Learning Research},
  volume = {162},
  pages = {19139--19152},
  year = {2022}
}

@inproceedings{shi2023finite,
  author = {Shi, Jiaxin and Mackey, Lester},
  title = {A Finite-Particle Convergence Rate for {Stein} Variational Gradient Descent},
  booktitle = {Advances in Neural Information Processing Systems},
  volume = {36},
  pages = {26831--26844},
  year = {2023}
}

@inproceedings{banerjee2025improved,
  author = {Banerjee, Sayan and Balasubramanian, Krishnakumar and Ghosal, Promit},
  title = {Improved Finite-Particle Convergence Rates for {Stein} Variational Gradient Descent},
  booktitle = {International Conference on Learning Representations},
  year = {2025},
  note = {arXiv:2409.08469}
}

@article{carrillo2025stability,
  author = {Carrillo, Jos\'e A. and Skrzeczkowski, Jakub},
  title = {Convergence and Stability Results for the Particle System in the {Stein} Gradient Descent Method},
  journal = {Mathematics of Computation},
  volume = {94},
  number = {354},
  pages = {1793--1814},
  year = {2025}
}

@article{he2025regularized,
  author = {He, Ye and Balasubramanian, Krishnakumar and Sriperumbudur, Bharath K. and Lu, Jianfeng},
  title = {Regularized {Stein} Variational Gradient Flow},
  journal = {Foundations of Computational Mathematics},
  volume = {25},
  number = {4},
  pages = {1199--1257},
  year = {2025},
  doi = {10.1007/s10208-024-09663-w}
}

@book{taylor1996pde2,
  author = {Taylor, Michael E.},
  title = {Partial Differential Equations II: Qualitative Studies of Linear Equations},
  publisher = {Springer},
  address = {New York},
  year = {1996}
}

\end{document}